\documentclass[11pt]{article}
\usepackage[a4paper,left=0.85in,right=0.85in,top=0.85in,bottom=1.15in]{geometry}
\usepackage{graphicx}

\graphicspath{{figures/}}
\usepackage{svg}

\usepackage{caption}
\usepackage{subcaption}
\usepackage[british]{babel}
\usepackage[utf8]{inputenc}

\usepackage{amssymb}
\usepackage{amsmath}
\usepackage{amsthm}
\usepackage{mathtools}
\usepackage{amsfonts}
\usepackage{physics}
\usepackage[scr=dutchcal]{mathalpha}
\usepackage{stmaryrd}

\usepackage{makecell}

\usepackage[pdfencoding=auto, psdextra]{hyperref}
\hypersetup{
    colorlinks=true,
    linkcolor=blue,
    filecolor=magenta,      
    urlcolor=cyan,
    }
\usepackage{cleveref}
\usepackage{authblk}

\usepackage[normalem]{ulem}
\usepackage{cancel}

\usepackage{tikz-cd}
\usepackage{tikz}
\usepackage{pgfplots}
\usetikzlibrary{patterns,patterns.meta}
\usetikzlibrary{intersections}
\usetikzlibrary{math}
\pgfplotsset{compat=1.18}
\usepackage{quiver}
\usepackage{diagbox}
\usepackage{CormorantGaramond}
\usetikzlibrary{decorations.markings}

\usepackage{color}
\usepackage{xcolor}
\definecolor{FTChampagnePink}{HTML}{F2DFCE}
\definecolor{FTOldLace}{HTML}{FFF1E0}
\definecolor{FTFloralWhite}{HTML}{FFF9F5}
\definecolor{FTMetallicSeaweed}{HTML}{0D7680}
\definecolor{FTMaroon}{HTML}{8F223A}
\definecolor{BBGreen}{HTML}{18453B}

\makeatletter
\newcommand{\globalcolor}[1]{%
  \color{#1}\global\let\default@color\current@color
}
\makeatother

\usepackage{booktabs}
\usepackage{diagbox}

\usepackage{cancel}
\usepackage[textsize=footnotesize]{todonotes}

\usepackage{enumitem}
\usepackage{comment}
\usepackage{thmtools, thm-restate}

\usepackage{blkarray, bigstrut}
\usepackage[]{natbib}

\usepackage{array}
\usepackage{upgreek}
\newcommand{\functionname}{chordal distance transform~}
\newcommand{\FunctionName}{Chordal Distance Transform~}
\newcommand{\FN}{CDT~}
\newcommand{\AutoDsq}{\ensuremath{\Gamma}} 
\newcommand{\AutoD}{\bar{\mathsf{D}}}
\newcommand{\DT}{\mathcal{T}}
\newcommand{\DTm}{\mathsf{T}}

\newcommand{\VolT}{\mathsf{Vol}}
\newcommand{\diam}{\mathrm{diam}}
\newcommand{\FSpace}[3]{\ensuremath{C^{#1}({#2}, {#3})}}
\newcommand{\Emb}[3]{\ensuremath{\mathrm{Emb}^{#1}({#2}, {#3})}}
\newcommand{\Jet}[4]{\ensuremath{J^{#1}_{#2}({#3}, {#4} )}}
\newcommand{\prolong}[2]{\ensuremath{j^{#1}_{#2}}}
\newcommand{\transverse}{\ensuremath{\pitchfork}}
\newcommand{\UConf}[2]{\ensuremath{\mathrm{UConf}_{#1}(#2)}}
\newcommand{\Conf}[2]{\ensuremath{\mathrm{Conf}_{#1}(#2)}}
\newcommand{\ExpS}[2]{\ensuremath{\mathbf{F}_{#1}(#2)}}

\newcommand{\ucpt}[1]{\ensuremath{\lbag {#1} \rbag}}

\DeclareMathOperator{\interior}{\mathrm{int}}

\newcommand{\closure}[1]{\mathrm{cl}(#1)}
\newcommand{\innerprod}[2]{\ensuremath{\langle {#1}, {#2} \rangle}}
\newcommand{\Mob}{\ensuremath{\mathcal{M}}}

\newcommand{\stripinf}{\ensuremath{[0,1] \times \R}}

\DeclareMathOperator{\sgn}{sgn}

\newcommand{\Con}{\mathrm{Con}}

\DeclareMathOperator{\pershomf}{PH} 

\DeclareMathOperator{\iden}{\mathrm{id}} 

\DeclareMathOperator{\scre}{\mathscr{e}}
\DeclareMathOperator{\screii}{\tilde{\scre}}

\DeclareMathOperator{\Z}{\mathbb{Z}}
\DeclareMathOperator{\R}{\mathbb{R}}

\DeclareMathOperator{\Id}{\mathrm{id}}

\newcommand{\homeo}{\approx}

\newcommand{\inv}[1]{#1^{-1}}
\newcommand{\fibre}[2]{\inv{#1}{\left({#2}\right)}}
\newcommand{\sublevel}[2]{\inv{#1}(-\infty, {#2}]}

\newcommand{\sett}[1]{\left\{ {#1} \right\}}

\DeclareMathOperator{\imag}{im}

\newcommand{\cU}{\mathcal{U}}
\newcommand{\cV}{\mathcal{V}}
\newcommand{\cW}{\mathcal{W}}
\newcommand{\cN}{\mathcal{N}}

\newcommand{\cI}{\mathcal{I}}
\newcommand{\cJ}{\mathcal{J}}

\newcommand{\sU}{\mathsf{U}}
\newcommand{\sV}{\mathsf{V}}
\newcommand{\sW}{\mathsf{W}}
\newcommand{\sN}{\mathsf{N}}

\newcommand{\nerve}[1]{\sN({#1})}

\let\angle\measuredangle

\usepackage{multirow}

\theoremstyle{plain}
  \newtheorem{theorem}{Theorem}[section]
  \newtheorem{corollary}[theorem]{Corollary}
  \newtheorem{lemma}[theorem]{Lemma}
  
  \newtheorem{proposition}[theorem]{Proposition}

\theoremstyle{definition}
  \newtheorem{definition}[theorem]{Definition}
  
  \newtheorem{ex}[theorem]{Example}
  
  \newtheorem{remark}[theorem]{Remark}
  \newenvironment{example}{\begin{ex}}{\end{ex}}

\newcommand{\Dgm}{\mathbf{Dgm}}

\newcommand*{\seteq}{\stackrel{\text{set}}{=}}

\title{The \FunctionName of Geometric Loops and its Persistent Homology}

\author[1,2]{James A. D. Binnie\thanks{ Email:~\texttt{BinnieJA@cardiff.ac.uk}}}
\author[3,4,5]{Otto Sumray\thanks{Email:~\texttt{sumray@mpi-cbg.de}}}
\author[6]{Ka Man Yim\thanks{ Email:~\texttt{kaman.yim@stats.ox.ac.uk}}}

\affil[1]{School of Mathematics, Cardiff University, U.K.}
\affil[2]{Heilbronn Institute for Mathematical Research, Bristol, U.K.}
\affil[3]{Center for Systems Biology Dresden, Germany}
\affil[4]{Max Planck Institute of Molecular Cell Biology and Genetics, Germany}
\affil[5]{Max Planck Institute for the Physics of Complex Systems, Germany}
\affil[6]{Department of Statistics, University of Oxford, U.K.}

\date{\today}

\begin{document}
\maketitle
\begin{abstract}
We present an isometry and parametrisation invariant of embeddings of $S^1$ into Euclidean space. We do so by representing the distance between pairs of points on the embedded circle as a function on a M\"obius band, the two-point finite subset space of $S^1$. We call this function the chordal distance transform of the embedding. We show that the sublevel set persistent homology of the chordal distance transform satisfies the desired isometry and parametrisation invariance, and is a continuous transform with respect to the Whitney topology on the space of circle embeddings and the bottleneck distance in the space of persistence diagrams. We then considered the generic behaviour of the chordal distance transform for $C^2$ and finite piecewise linear embeddings. In the $C^2$-case, we show that  non-boundary critical points of the chordal distance transform are finite and non-degenerate on an open and dense subset of circle embeddings. Consequently, its persistent homology is pointwise finite dimensional for generic $C^2$-embeddings. In the finite piecewise linear case, we also find piecewise-continuous analogues of non-degenerate critical points, and give generic conditions for the homological critical points of the chordal distance transform to be non-degenerate. In order to gain a geometric interpretation of the chordal distance transform and its persistent homology, we give a geometric characterisation of the $C^2$ and finite piecewise linear non-degenerate critical points. Finally, we consider how the chordal distance transform can be generalised to capture geometric features involving $n\geq 2$ points on an embedded shape, as a function on the $n$-point finite subset space.
\end{abstract}
\newpage 
\tableofcontents
\newpage
\section{Introduction}
\label{sec:intro}

Topological Data Analysis (TDA) concerns the statistical analysis of shapes. A leading paradigm for tackling this problem is to transform shapes into vectors features which express salient topological and geometric properties of the shape. In mathematical terms, this involves desigining a feature map $F$ from a space of shapes, to a vector space $V$. Having expressed characteristics of shapes as vectors, classical statistical tools can then be applied to analyse datasets of shapes.
In this study, we consider how we can perform topological data analysis on \emph{geometric loops}. Our work is motivated by problems in biological and chemical shape analysis, where  structures such as cell contours and ring compounds are modelled as geometric loops~\cite{Bleile2023PersistencePopulation}. 

\begin{figure}[ht]
    \centering
    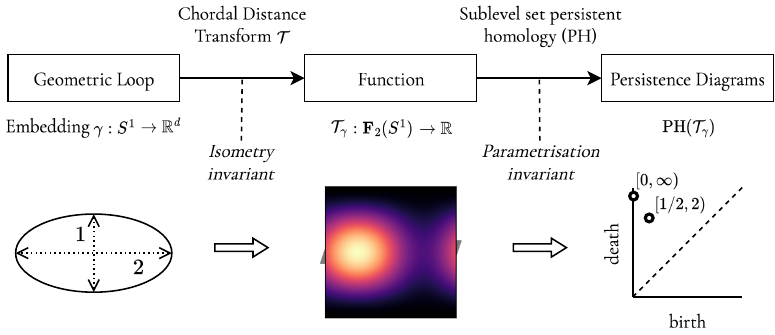
    \caption{The \functionname transforms a geometric loop represented by an embedding $\gamma$ into a function $\DT_\gamma$ on the M\"obius band. This gives an isometry invariant representation of our geometric loop.  The persistent homology of the sublevel set filtration of $\DT_\gamma$ yields a summary of the geometry of the loop that is invariant with respect to reparametrisations of $\gamma$.}
    \label{fig:pipeline}
\end{figure}

The space of geometric loops is parametrised by the space of circle embeddings $\Emb{}{S^1}{\R^d}$ into Euclidean space. However, applying statistical methods to the embeddings directly is a potentially challenging prospect, since there are infinitely many \emph{equivalent} embeddings which represent the same geometric loop. As such, any statistical model inferred on finitely many observations of loop embeddings may fail to generalise well to unseen data.  To overcome this challenge, we consider a feature map (\Cref{fig:pipeline}) which is a composition of two transformations:
\begin{equation} \label{eq:master_feature_map}
  F:\quad   \Emb{}{S^1}{\R^d} \xrightarrow{\DT} C(\ExpS{2}{S^1}, \R) \xrightarrow{\pershomf} \Dgm 
\end{equation}
The first map is the \emph{\functionname} $\DT$, which is a topological realisation of the distance matrix of points along the circle embedding $\gamma$ as a real-valued function $\DT_\gamma$ on a M\"obius band. The M\"obius band is homeomorphic to the two-point finite subset space $\ExpS{2}{S^1}$ of the circle. However, as $\DT_\gamma$ is dependent on how $\gamma$ is parametrised, we apply \emph{sublevel set persistent homology} ($\pershomf$)~\cite{CohenSteiner2007,edelsbrunner2008persistent,edelsbrunner2010computational} to remove this dependency that is not related to the geometry of the loop. Persistent homology is a topological summary of the function that describes the changes in the of sublevel sets 
\begin{equation}
    \AutoDsq^a := \{x \in \ExpS{2}{S^1} \ : \ \DT_\gamma(x) \leq a\}
\end{equation}
as we vary the thresholding parameter $a \in [0,\infty)$. The persistent homology of $\DT_\gamma$ in dimension $i$ is the family of linear maps between homology groups of sublevel sets (with coefficients in some fixed field $\mathbb{F}$), induced by inclusion:
\begin{equation*}
    H_i(\Gamma^\bullet) : (H_i(\Gamma^a)  \to H_i(\Gamma^{b}))_{0 \leq a \leq b < \infty}
\end{equation*}
This family of linear maps, or persistence module, is represented by a persistence diagram $\pershomf(\DT_\gamma)$ that describe the birth and death of homological features as we vary the filtration parameter $a$. Persistence diagrams can then be vectorised to produce features compatible with statistical and machine learning algorithms (see~\cite{Ali2023-xo} for a survey of vectorisation methods for persistence diagrams). 
\begin{table}[h]
\centering
\begin{tabular}{|p{0.15\linewidth}|p{0.35\linewidth}|p{0.35\linewidth}|}
\hline
                              & {Finite, piecewise linear embeddings}  & Smooth embeddings ($C^k$, $k \geq 2$)                                      \\ \hline
Continuity                                  & \multicolumn{2}{p{0.7\linewidth}|}{
\makecell{$\DT: \Emb{k}{S^1}{\R^d} \to C^k(\ExpS{2}{S^1}, \R)$ continuous for $k \geq 0$ (\Cref{lem:continuity_of_distance_transforms}) \\ $\pershomf \circ \DT: \Emb{}{S^1}{\R^d} \to \Dgm$ continuous (\Cref{cor:cont_ph_cdt}).}}                                   \\ \hline
Invariance                                  & \multicolumn{2}{p{0.7\linewidth}|}{$\pershomf \circ \DT$ Invariant with respect to reparametrisations of $S^1$ and Euclidean isometries for any topological embedding (\Cref{prop:sym})} \\ \hline
Tameness of persistence module $\pershomf(\DT_\gamma)$                    & {Yes (\Cref{prop:nerve_to_sublevel,cor:PL_tame})}                       & On a generic (i.e. open and dense) subset of space of embeddings $\Xi \subset \Emb{k}{S^1}{\R^d}$ (\Cref{thm:main_morse})  \\ \hline
Geometric characterisation of critical points of $\DT_\gamma$ (i.e. birth and death of features in $\pershomf(\DT_\gamma)$) & Non-degenerate critical points if generic conditions ~\labelcref{C1,C2,C3} satisfied, geometric and homological description in~\Cref{dfn:pl-critical-index} and ~\Cref{thm:PLmain} respectively.              & Non-degenerate on $\Xi$; described in \Cref{thm:SmoothChar} \\ \hline
\end{tabular}

\caption{A summary of the key properties of $\DT$. The continuity and invariance of $\pershomf \circ \DT$ holds for topological embeddings of circle into Euclidean space. We show for two special cases of embeddings -- where they are either finite piecewise linear, or $C^2$ -- that we can have tame persistence modules and geometric characterisations of critical points with different Morse indices. \label{tab:contribution}}
\end{table}

In this article, we focus on theoretical properties of the feature map $F$ in~\cref{eq:master_feature_map} to justify its use in data analysis of geometric loops. In particular, we consider four desirable properties of $F$, namely:
\begin{enumerate}
    \item The continuity of $F$;
    \item The invariance of $F$ with respect to Euclidean isometries and reparametrisations of the loop $\gamma$; 
    \item The tameness of $F$ (i.e. whether persistence diagrams $F(\gamma) = \pershomf(\DT_\gamma)$ are represented by finitely many points); and
    \item The geometric interpretation of the the persistence diagram $\pershomf(\DT_\gamma)$.
\end{enumerate}
We give a summary of how we have addressed these questions in~\Cref{tab:contribution}. For topological circle embeddings $\Emb{}{S^1}{\R^d}$, we prove that $\pershomf \circ \DT$ is continuous and invariant with respect to reparametrisations of the embedding and Euclidean isometries in in~\Cref{cor:cont_ph_cdt} and~\Cref{prop:sym} respectively. For tameness and geometric interpretation of $\pershomf(\DT_\gamma)$, we restrict ourselves two broad classes of circle embeddings, where the embedding $\gamma$ is either $C^2$ or finite piecewise linear (PL). The PL case arises in applications where we may only observe finitely many samples $(x_i)_{i \in \Z_n}$ of some loop, and we use a piecewise linear interpolation to approximate the underlying shape. They also arise naturally in chemistry where we can represent ring compound configurations as piecewise linear loops. The $C^2$ case is also of independent theoretical interest from a differential geometry perspective. We are also practically motivated to study to $C^2$ case as they describe smooth interpolations of finite data $(x_i)_{i \in \Z_n}$ using splines. 

The restriction to $C^2$ and PL embeddings allows us to take a Morse-theoretic approach to address the tameness of persistence diagrams $\pershomf(\DT_\gamma)$ and the geometric interpretation of homological features described in $\pershomf(\DT_\gamma)$. Morse theory enables us to constrain the topology of the sublevel sets of functions with analytic properties of a function's critical points. If a critical point is non-degenerate, then they correspond to the birth or death of death of a homological feature in the filtration $\Gamma^a$, as summarised in persistence diagrams $\pershomf(\DT_\gamma)$. Thus, we can relate the persistence diagram $\pershomf(\DT_\gamma)$ to the underlying geometry of the loop $\gamma$ by studying how non-degenerate critical points of $\DT_\gamma$ depend on the geometric and analytic properties of $\gamma$.

In the case where embeddings are $C^2$, we use tools from transversality theory in differential topology and classical (smooth) Morse theory to show that for a generic circle embedding $\gamma$, its chordal distance transform $\DT_\gamma$ only has finitely many non-degenerate critical points on the interior of $\ExpS{2}{S^1}$. As a consequence, for a generic embedding $\gamma$, each non-trivial birth or death of a homological feature in the filtration $\Gamma^a$ can be associated with a non-degenerate critical point of $\DT_\gamma$, and the persistence diagram $\pershomf(\DT_\gamma)$ is tame (\Cref{thm:main_morse}). In the case where $\gamma$ is PL, smooth Morse theory fails to apply. However, we can describe the filtration $\Gamma^a$ of $\ExpS{2}{S^1}$ combinatorially instead, as $\Gamma^a$ is homotopy equivalent to a filtration of a finite simplicial complex (\Cref{prop:nerve_to_sublevel}). We then proceed to analyse the simplicial filtration using discrete/combinatorial Morse theory. As we can describe the homological changes with finite data, we  can conclude that $\pershomf(\DT_\gamma)$ is tame. This homotopy equivalence also gives a combinatorial algorithm that allows us to compute the persistent homology $\pershomf(\DT_\gamma)$ exactly in the PL case.  Furthermore, we show that if a set of generic conditions~\labelcref{C1,C2,C3} is satisfied for a PL embedding, then we can identify PL analogues of $C^2$-non-degenerate critical points in the PL-case (\Cref{thm:PLmain}). 

In the generic case described above, we can interpret the birth and death of homological features in $H_i(\Gamma^\bullet)$ by studying the geometric significance of the non-degenerate critical points of $\DT_\gamma$ associated with them. As $\ExpS{2}{S^1}$ is two-dimensional manifold with boundary (in fact a M\"obius band), the non-degenerate critical points of $\DT_\gamma$ in the interior of $\ExpS{2}{S^1}$ are either minima, saddles, or maxima; the birth of a feature in $H_0(\Gamma^\bullet)$ corresponds to a minimum, the death of a feature in $H_1(\Gamma^\bullet)$ corresponds to a maximum, and the death of a feature in $H_0(\Gamma^\bullet)$ or a birth of a feature in $H_1(\Gamma^\bullet)$ corresponds to a saddle point. In the $C^2$-case, we give a geometric description of minima, saddles, or maxima in terms of local tangent and curvature conditions of $\gamma$ in~\Cref{thm:SmoothChar}. In the PL-case, we describe them geometrically in~\Cref{dfn:pl-critical-index} and~\Cref{thm:PLmain}
in terms of the
acuteness or obtuseness
of angles between line segments and chords.

Finally, we discuss how we may generalise the \functionname to encode higher order geometric information about any embeddings $X \hookrightarrow \R^d$ into Euclidean space. We propose a generalisation called the volume transform $\VolT_k$; the case where $X = S^1$ and $k = 1$ recovers the chordal distance transform $\DT$. We discuss challenges in deriving theoretical bounds for such generalisations, as higher order finite subset spaces of manifolds can be stratified spaces where standard differential topology and Morse theory for manifolds cannot be directly applied.

\subsection*{The Chordal Distance Transform}

We propose the \functionname $\DT$ as a representation of the geometry of an embedding $\gamma$, using the distance between pairs of points $\gamma(z_1), \gamma(z_2)$ on the geometric loop. Consider the continuous function $d_\gamma: S^1 \times S^1 \to \R$ on the torus, where $d_\gamma(z_1, z_2) = \norm{\gamma(z_1) - \gamma(z_2)}_2$. Na\"ively, we can represent the geometry of the loop with $d_\gamma$. However, $d_\gamma$ contains redundant information as $d_\gamma$ is symmetric when we permute $z_1$ and $z_2$; the function on one half of the torus carries the same information as the other (see~\Cref{fig:ellipse_distance_matrix}). To remove this redundancy, we consider the induced function on the quotient space $\ExpS{2}{S^1} = (S^1 \times S^1)/S_2$, which is homeomorphic to a M\"obius band $\Mob$. The space $\ExpS{2}{S^1}$ is the space of subsets $\{z_1, z_2\} \subset S^1$ of cardinality at most 2. Since $\{z_1, z_2\}$ defines a \emph{chord} on the circle, we can consider $\ExpS{2}{S^1}$ as the space of chords on a circle. The chordal distance transform is then a function that sends a chord $\{z_1, z_2\}$ to the distance between its end points $\norm{\gamma(z_1) - \gamma(z_2)}_2$ in the embedded loop. 
\begin{definition}
    The chordal distance transform $\DTm: \Emb{}{S^1}{\R^d} \to C(\ExpS{2}{S^1}, \R)$ is the function that sends a circle embedding $\gamma : S^1 \hookrightarrow \R^d$ to the unique continuous map $\DTm_\gamma: \ExpS{2}{S^1} \to \R$ on the quotient space $\ExpS{2}{S^1}$ of $S^1 \times S^1$, such that $d_\gamma = \DTm_\gamma \circ q$:
\begin{equation} \label{eq:Dist}
    \begin{tikzcd}[cramped]
	{S^1 \times S^1} & {\mathbb{R}_{\geq 0}} \\
	{\ExpS{2}{S^1}}
	\arrow["d_\gamma", from=1-1, to=1-2]
	\arrow["q"', two heads, from=1-1, to=2-1]
	\arrow["{\exists ! \DTm\gamma}"', dashed, from=2-1, to=1-2]
\end{tikzcd}.
\end{equation}
Explicitly, $\DTm_\gamma(\{z_1, z_2\}) = \norm{\gamma(z_1) - \gamma(z_2)}_2$. The smoothed chordal distance transform $\DT: \Emb{}{S^1}{\R^d} \to  C(\ExpS{2}{S^1},\R)$ is give by.
\begin{equation}
    \DT_\gamma = \frac{1}{2}\DTm_\gamma^2.
\end{equation} 
\end{definition}
\begin{remark}
    We mainly work with the smoothed chordal distance transform here for the purposes of deriving theoretical results, as $\DT_\gamma$ is a $C^k$-map on $\ExpS{2}{S^1}$ as a manifold with boundary if $\gamma$ is $C^k$. The theoretical properties we derive for $\pershomf(\DT_\gamma)$ extend to $\pershomf(\DTm_\gamma)$, as there is a bijection between the persistence diagrams of $\pershomf(\DT_\gamma)$ and $\pershomf(\DTm_\gamma)$ (\Cref{lem:square_bijection}). In applications, we should use $\DTm$ in practice, as is a 2-Lipschitz map from $\Emb{}{S^1}{\R^d}$ to $C(\ExpS{2}{S^1}, \R)$ (\Cref{lem:continuity_of_distance_transforms}).
\end{remark}

\begin{figure}[t]
    \centering
    \includegraphics[width=0.5\linewidth]{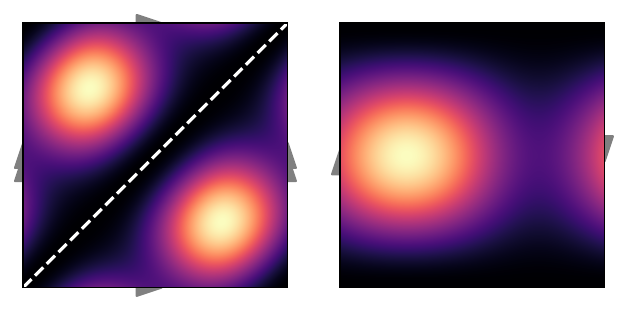}
    \caption{We derive the \functionname of a circle embedding $\gamma$ from the distance matrix of pairs of points, as a function $d: S^1 \times S^1 \to \R$. Since the distance matrix is symmetric with respect to permutation  $d(z_1,z_2) = d(z_2,z_1)$ (visually represented by reflecting across the diagonal),  we remove the redundant information introduced by symmetry by only considering the distance evaluated on pairs of points on the circle, modulo permutation $\{z_1, z_2\}$. This gives us a function $\DTm_\gamma$ on the two-point subset space $\ExpS{2}{S^1}$ of $S^1$, which is homeomorphic to the M\"obius band $\Mob$. We call this function the \functionname of $\gamma$, as each $\{z_1, z_2\}$ in $\ExpS{2}{S^1}$ corresponds to a chord on the circle, and $\DTm_\gamma(\{z_1, z_2\})$ encodes the length of each chord on the embedded loop. On the left, we illustrate the distance matrix of an ellipse $\gamma$ with major axis = 2, minor axis = 1. On the right, we illustrate the chordal distance trasnform $\DTm_\gamma: \ExpS{2}{S^1} \to \R$ as derived from $d$.  }
    \label{fig:ellipse_distance_matrix}
\end{figure}

\begin{figure}[h!]
    \centering
    
    \begin{subfigure}{0.45\textwidth}
        \centering
        \includegraphics[width=\linewidth]{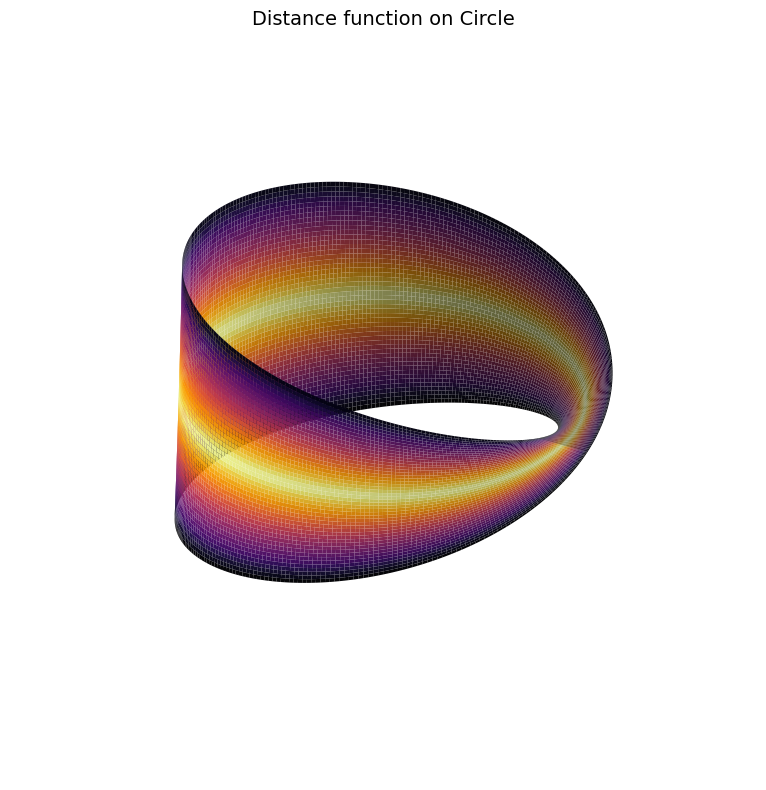}
        \caption{Chordal distance transform of a circle.}
        \label{fig:dist3}
    \end{subfigure}
    \hspace{0.02\textwidth}
    \begin{subfigure}{0.45\textwidth}
        \centering
        \includegraphics[width=\linewidth]{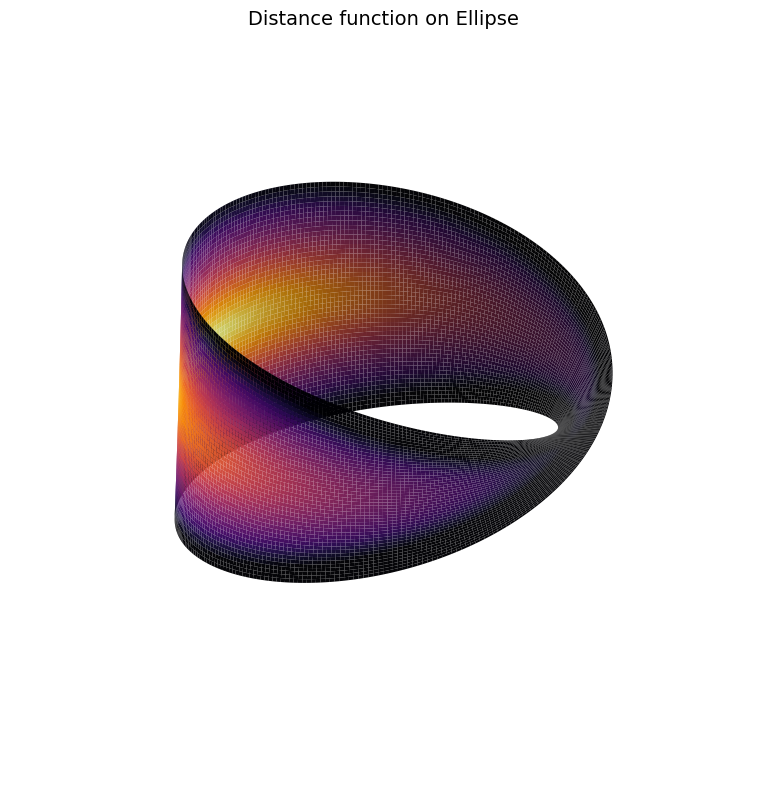}
        \caption{Chordal distance transform of an ellipse.}
        \label{fig:dist4}
    \end{subfigure}

    \caption{Heat map representations of the \functionname of a circle (\cref{fig:dist3}) and an ellipse (\cref{fig:dist4}), as a function on the M\"obius band. The ellipse here has major axis $a=8$ and minor axis $b=1$. }
    \label{fig:bothdist}
\end{figure}

\subsection*{Related Work}

\paragraph{Shape statistics}
In shape analysis, a range of statistical summaries have been developed to quantify and compare the shape of data. TDA-based approaches include the Persistent Homology Transform (PHT)~\cite{Turner2014PHT}
and the Euler Characteristic Transform (ECT)~\cite{HacquardLebovici2024Euler, Munch02012025} 
which
record the persistent homology, respectively Euler characteristic,
of the shape filtered in each direction. For broad classes of shapes, the PHT/ECT is injective and thus uniquely determines the underlying set under mild conditions.
However, both methods are not invariant to rotation/translation and require the shape to be first oriented,
although there exist methods to remove some of this sensitivity (e.g.~\cite{yang2026topological}).
Persistent homology has been directly used to study the shapes of loops
for cell contours~\cite{Bleile2023PersistencePopulation}
and for knotted and unknotted proteins~\cite{Benjamin2023-qb}.
More classical methods include Elliptical Fourier Descriptors (EFDs)~\cite{NixonAguado2012Feature, KUHL1982236}
and Procrustes analysis~\cite{Gower1975GPA, AdamsRohlfSlice2004Geomorph}.
EFDs are essentially the Fourier expansions of the Cartesian coordinate functions of a parametrised closed contour.
EFDs can be normalised to become invariant to translation and rotation (and scale),
however, as with the PHT and ECT this normalisation is non-canonical; translation invariance is usually achieved 
by centring the centroid at the origin.
Procrustes analysis  provides a geometric alignment by optimally removing translational, rotational, and scale differences between shapes (often represented as landmark sets), which allow comparative studies of morphological variation once the extrinsic differences have been eliminated. 

\paragraph{Persistent homology on distance matrices}
Applying persistent homology to distance matrices as a filtered complex is a concept that has appeared in several works related to times series analysis and characterising dynamical systems. In an article by Takashi Ichinomiya~\cite{ichinomiya2025machine} and David Hien's PhD thesis~\cite{hien2025topological}, the distance matrix of time series data, also known as a recurrence plot~\cite{Eckmann1995-vb}, is represented as a a filtered cubical complex called the distance complex. In~\cite{ichinomiya2025machine} the complex is constructed on the full distance matrix, whereas~\cite{hien2025topological} restricts the complex to the upper triangular half of the distance matrix, as the distance matrix is symmetric across the diagonal. This is equivalent to a representation of the unordered two-point configuration space $\ExpS{2}{[0,1]}$ of $[0,1]$ as a cubical complex. Takashi~\cite{ichinomiya2025machine} empirically demonstrated that features derived from the persistent homology of the distance complex can be used to distinguish periodic and chaotic behaviour in dynamical systems, and is useful in classification of electrocardiogram data. In~\cite{hien2025topological}, Hien focused on understanding the theoretical properties of this construction, proving that the $0$\textsuperscript{th} persistent homology of the distance complex captures more information than the $1$\textsuperscript{st} persistent homology of the Vietoris-Rips complex of the time series as a point cloud~\cite[Theorem 5.2.5]{hien2025topological}. 

The chordal distance transform can be thought of as an analogous construction if we interpret embeddings of $S^1$ as cyclic time series data. The sublevel set filtration of the chordal distance transform is a continuous version of the distance complex for explicitly cyclic data. As we are not considering embeddings of $[0,1]$ but rather of $S^1$, the space which is filtered in our case is a M\"obius band $\ExpS{2}{S^1} \homeo \Mob$, rather than one homeomorphic to $[0,1]^2$ or $\ExpS{2}{[0,1]}$ in~\cite{ichinomiya2025machine} and~\cite{hien2025topological} respectively. In our study of piecewise linear embeddings of $S^1$ that interpolate finite cyclic data $(x_i)_{i \in \Z_n}$ (\Cref{sec:pl-case}), we consider the chordal distance transform of such an embedding as a piecewise continuous function on $\Mob$ as a manifold, which is different to the distance complex built on the point set $(x_i)_{i \in \Z_n}$. The filtration of the distance complex is induced solely from the distances between points $x_i,x_j$ in the point set, whereas the chordal distance transform takes into account of distance between points on the piecewise linear embedding interpolated between $x_i$ and $x_{i+1}$. As we see in~\Cref{sec:pl-case}, homological critical values of the chordal distance transform filtration of $\Mob$ can be induced by interpolated points, whereas the distance complex only has distances between points in $(x_i)_{i \in \Z_n}$ as critical values.

\paragraph{Configuration and Finite subset spaces}

Configuration spaces record the positions of finitely many distinct points in a topological space.
If this space is a manifold, the homotopy type of the configuration space is a homeomorphism invariant of the underlying manifold
and as such has been widely studied in algebraic topology
\cite{Knud, CohenConf, kallel2025conf}.
Finite subset spaces were introduced in the early 20th century~\cite{BorsukUlam1931SymmetricProducts, Bott}
with more recent development by Handel~\cite{handel} and Tuffley~\cite{tuffleyexp,TuffPHD} providing homotopic results about finite subset spaces for general topological spaces, surfaces and graphs.
Configuration spaces on graphs have also appeared in application, e.g. ~\cite{ghrist1999configurationspacesbraidgroups}.

A related function
to the \functionname
has been used with configuration spaces in the inscribed squares problem.
In
Vaughan's proof~\cite{vaughan1977rectangles} (see Section 4 in~\cite{meyerson1981balancing})
that any planar simple smooth curve contains a rectangle,
he defines a function $f: \ExpS{2}{S^1} \to \R^3$ where
{
$
f(\{x,y\}) = (\frac{1}{2}(\gamma(x)+\gamma(y)), \|\gamma(x)-\gamma(y)\|).
$
}
If $f$ is injective, it would then induce an embedding of $\mathbb{R}P^2$ into $\R$,
which is not possible. $f$ being injective implies $\gamma$ contains an inscribed rectangle.

\subsection*{Outline}

We give an outline of the content in this article.~\Cref{sec:Background} gives necessary mathematical background for the paper:~\cref{sec:Background:config} introduces configuration spaces and finite subset spaces. We in particular focus on $\ExpS{2}{S^1}$, the 2\textsuperscript{nd} finite subset space of $S^1$, and show that it is homeomorphic to a M\"obius band $\Mob$. We also give an explicit atlas of $\ExpS{2}{S^1}$ as a manifold with boundary, which we use to give local coordinates for functions on $\ExpS{2}{S^1}$ for the remainder of the paper.
~\Cref{sec:Background:ph-sublevelset} gives background on smooth Morse theory and persistent homology of sublevel set filtrations,
relating non-degenerate critical points of the birth and death of features in sublevel set persistent homology.
We introduce the notion of Morse-supported (\Cref{def:finite_morse}) for the functions under our consideration,
namely non-negative functions on a manifold with zero on the boundary,
later used in~\cref{sec:smooth-case}.
For the discrete case we give background on Conley indices, used in ~\cref{sec:pl-case}.
In~\cref{ssec:differential_topology} we give a brief account of strong Whitney topologies for function spaces, jet bundles, and transversality theory, which is necessary for the genericity result in~\cref{ssec:smooth_genericity}.

In~\cref{sec:elementary-props} we provide some elementary desirable properties of the \functionname
and the persistent homology of the \functionname. We prove that $\DT$ and $\pershomf \circ \DT$ are continuous transforms in~\cref{ssec:continuity}. In~\cref{ssec:symmetry}, we show that  that $\pershomf \circ \DT$ they are invariant with respect to reparametrisations and Euclidean isometries. Furthermore, we discuss how the persistent homology of functions on the M\"obius band can depend on the choice of field coefficients for homology in~\cref{ssec:homology_field}.

In~\cref{sec:smooth-case} we focus on smooth embeddings of loops in Euclidean space.
In~\cref{ssec:smooth_genericity}, we prove the main result of this section (\Cref{thm:main_morse}): the \functionname
of (twice-differentiable) embeddings are generically Morse-supported, and hence their associated persistence module is finite dimensional. We address the open and density aspects of genericity in~\cref{ssec:smooth_open,ssec:smooth_dense} respectively.
Having shown that the generic embeddings are Morse-supported, we analyse the non-degenerate critical points of the \functionname in~\cref{ssec:crit_geometric_smooth}, and give a geometric characterisation of the local minima, maxima and saddle points. This thus gives a geometric interpretation of the birth and death of homological features in the persistence diagrams $\pershomf_i\circ \DT (\gamma)$.

\Cref{sec:pl-case} focuses on piecewise-linear embeddings of $S^1$ into Euclidean space.
The main results is~\Cref{thm:PLmain} that gives necessary and sufficient geometric conditions \labelcref{plmain-CAw,plmain-CTog,plmain-CTow} for a point on $\ExpS{2}{S^1}$ to be a non-degenerate homological
critical point of $\DT_\gamma$.
This is achieved in two parts: first by constructing a filtered simplicial complex whose persistent homology is isomorphic to the sublevel set filtration of $\ExpS{2}{S^1}$ by $\DT_\gamma$
(\cref{ssec:what-a-nerve}).
We give a precise construction of the simplicial complex and the filtration values of its simplices, allowing the persistent homology to computed in practice.
Second, under certain genericity of the embedding \labelcref{C1,C2,C3},
we show that the Morse sets of this simplicial complex are simplices themselves
and we give a geometrically characterisation for a point in $\ExpS{2}{S^1}$ to correspond to such a critical simplex and determine its index as an $n$-saddle (\cref{ssec:pl-persistence}).

In Section~\ref{sec:ho-functions} we describe possible extensions of the \functionname
to functions on higher order finite subset spaces. We show that the chordal distance transform can be seen as a special case of the $k$-simplex volume transform that takes embeddings to functions on $(k+1)$\textsuperscript{th} finite subset spaces. We show that this transform is continuous and stable in~\Cref{thm:ContStabVol}.

Finally, in~\Cref{sec:disft}, we have a concluding discussion on possible applications of the chordal distance transform to biology, and theoretical challenges in generalising the Morse-theoretic framework to higher order finite subset spaces.

\section{Background}
\label{sec:Background}

In this section we introduce the spaces that we work with, the configuration spaces, the symmetric product and the $n$-th finite subset space. We provide examples and intuition of the spaces and give some topological properties. We cover the second finite subset space of the circle in detail as this is the space we use in the CDT. We provide the fundamentals of Morse theory and persistent homology that will be used later sections. Further to that, we introduce differential topology concepts such as, the Whitney topology on function spaces, jets and multi-jets. In the final part of this section we discuss results about subsets of positive reach.

\subsection{Configuration Spaces and Finite Subset Spaces}
\label{sec:Background:config}

The main spaces we will be considering throughout are: the ordered and un-ordered configuration spaces, the symmetric product and the $n$-th finite subset space. Throughout we will assume $X$ is a Hausdorff topological space.

\begin{definition}
    Given a space $X$, its \uline{\textbf{ordered}} $n$-\textbf{point configuration space} is defined as 
\begin{align}
    \Conf{n}{X} &:= \qty{(x_1, \ldots, x_n) \ : \ x_i \in X,\ i \neq j \implies x_i \neq x_j }=  \qty(\prod_{i=1}^n X) \setminus \Delta \\
    \qq*{where} \Delta &= \qty{(x_1,\ldots, x_n) \ : \ x_i = x_j \qq{for some} i\neq j} \label{eq:diagonal}
\end{align}
\end{definition}

Naturally the symmetric group $S_n$ acts on \Conf{n}{X} by permutation, which leads to the following:

\begin{definition}
    Taking the orbit space of $\Conf{n}{X}$ by $S_n$, define the \uline{\textbf{unordered}} $n$-\textbf{point configuration space}
\begin{equation}
    \UConf{n}{X} := \Conf{n}{X}/S_n. 
\end{equation}
\end{definition}

This quotient is in fact a covering map as we assume $X$ is Hausdorff and the action of $S_n$ (a finite group) is free. For $n=1$ the following is true, $X=\Conf{1}{X}=\UConf{1}{X}$. 

Given a set $\{x_1, \ldots, x_n\} \subset X$, we use the notation $\ucpt{x_1,\ldots, x_n}$  all elements notated in the brackets are distinct. Using this notation, can write $\UConf{n}{X}$ as a set 
\begin{equation} \label{eq:uconfpt}
    \UConf{n}{X} = \sett{\ucpt{x_1, \ldots, x_n} \ : \ x_i \in X}.
\end{equation}

\begin{remark}
The notation of configuration spaces has varied widely throughout the literature, $\mathcal{F}^{n}(X)$, $F(X,n)$, $B(X,n)$, $C_k(X).$ We shall use of the notation $\Conf{n}{X}$ and $\UConf{n}{X}$ to denote the ordered and unordered configuration spaces.
\end{remark}
\begin{example}\label{ex:confr}
We now consider  $\Conf{n}{\mathbb{R}}$ for $n = 2,3$:
We see that $\Conf{2}{\mathbb{R}}$ is just the plane without the diagonal $\Delta = \{ (x,x) \in \mathbb{R}^2\}$. Therefore, $\UConf{2}{\mathbb{R}}$ is the space that we get by folding over the diagonal $\Delta$. This space is homeomorphic to the upper half-plane.
For $\Conf{3}{\mathbb{R}}$ we have $\mathbb{R}^3$ without the planes $x=y$ , $y=z$ and $x=z$. This yields a space  split into $6$ disjoint areas. The space $\UConf{3}{\mathbb{R}}$, is then made by collapsing the $6$ disjoint areas to one wedge. Hence, this space is homeomorphic to a half-space of $\mathbb{R}^3$.
\end{example}

\begin{example}\label{ex:confint}
    For $\Conf{2}{(0,1)}$ we have a similar situation to $\mathbb{R}$, but with the constraints of lying inside the square $(0,1)^2$. We then have that $\UConf{2}{(0,1)}$ is homeomorphic to $\{(x,y) \in (0,1)^2 \ : \ x<y \}$. For $\Conf{3}{(0,1)}$, this is the unit box cut into 6 areas, and so when we identify them, we see that $\UConf{3}{(0,1)}$ is homeomorphic to the open 3-simplex. In \Cref{fig:confint} we see what a single area looks like. 
\end{example}

 \begin{figure}[h]
     \centering
\includegraphics[width=0.3\linewidth]{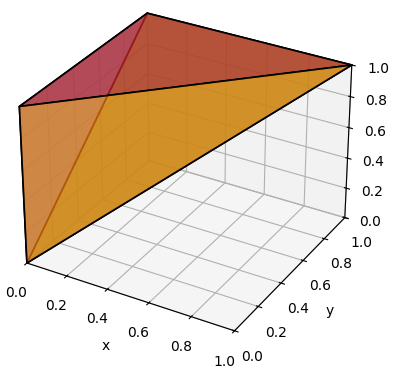}
     \caption{Here we see $\UConf{3}{(0,1)}$ as cut out by the planes, $x=y$, $x=z$ and $y=z$, where we consider $x<y<z$. This is homeomorphic to the open 3-simplex as described in \Cref{ex:confint}.} 
     \label{fig:confint}
 \end{figure}

\begin{remark}
    If $X$ is a smooth $d$-manifold then both \Conf{n}{X} and \UConf{n}{X} are smooth $dn$-manifolds \cite{Knud}. 
\end{remark}

If we do not limit ourselves to $\Conf{n}{X}$, but allow the entire space $X^n$ and the natural action of $S_n$, then we have the following space.

\begin{definition}
    The $n$-th \textbf{symmetric product }of $X$, is defined to be the quotient $$SP^n(X) := X^n / S_n$$
\end{definition}
The action is no longer free, and so for $n \geq 2$ the symmetric product is an orbifold for $X$ a smooth manifold (see \cite{Orbifolds} Ex. 1.13 ). We can express $SP^n(X)$ as the set of unordered $n$-tuples $[x_1,\ldots, x_n]$. 
\\

Although the symmetric product consolidates the symmetries of points of the space into one, we wish to lose all multiplicity of points that have the same set of points. For example, if we consider $n=3$ then we would like to treat $(a,b,b)$ as the same as $(a,a,b)$. However, in the symmetric product these two points have different equivalences classes. So we need to define a new space that will allow us to consolidate these points together.
\\

We now describe the $n$-th \textbf{finite subset space} of $X$, \ExpS{n}{X}. This is not only a quotient of $X^n$ but also a quotient of $SP^n(X)$, as we see in \Cref{prop:SP_to_Exp_quotient}.

 \begin{definition}
     We define $\ExpS{n}{X}$ to be be set of all non-empty finite subsets of $X$ of cardinality at most $n$ \cite{tuffleyexp}. The map $E_n : X^n \rightarrow \ExpS{n}{X} $  given by
\begin{equation}
    (x_1 , x_2 , ... ,x_n) \mapsto \{x_1, x_2 , ... , x_n\} 
\end{equation}
induces the quotient topology on $\ExpS{n}{X}$, the $n$th finite subspace set of $X$. 
 \end{definition}
Using the notation in \cref{eq:uconfpt}, we can also equivalently express the set $\ExpS{n}{X}$ as 
\begin{equation} \label{eq:Fn_UCk}
    \ExpS{n}{X} =  \sett{\ucpt{x_1, \ldots, x_k} \ : \ x_i \in X,\ k \in [n]} = \bigsqcup_{k=1}^n \UConf{k}{X}.
\end{equation}

We note that for $n=2$ we have that $SP^n(X) = \ExpS{n}{X}$ but for $n \geq 3$ this no longer is the case as the $S_3$ equivalence classes of $(x,x,y)$ and $(x,y,y)$ are certainly different in the symmetric product. However, their images under $E_3$ are the same, namely $\ucpt{x,y}$.

\begin{remark}
The $n$-th finite subset space, a term used by Tuffley~\cite{tuffleyexp,tuffley2003finite,TuffPHD}, is also known as the (finite) Ran space $\mathrm{Ran}_{\leq n}(X)$ \cite{AYALA2017903, lazovskis2026simpleconnectednessranspace}.  
It should be noted that the $n$-th finite subset space, as defined here, has been referred to as the symmetric product in the literature. We denote the $n$-th finite subset space as, $\ExpS{n}{X}$. Other notations for this space have been given by: $\text{Sub}(X,n)$, $\text{exp}_n(X)$, $X^{(n)}$ and $\text{Ran}_{\leq n}(X)$ ~\cite{handel,tuffleyexp,Bott,AYALA2017903}.

\end{remark}

\begin{proposition} \label{prop:SP_to_Exp_quotient}
    There exists a unique quotient map, $Q_n : SP^n(X) \rightarrow \ExpS{n}{X}$ such that $E_n = Q_n \circ q$.
\end{proposition}
\begin{proof}
   We have the quotient map $ q: X^{n} \xrightarrow{/S_n} SP^n(X),$ which is continuous by definition. It is open as if $U \subset X^n$ is open then, the pre-image of $[U]_{S_n}$ is the union $\bigcup_{\sigma \in S_n} \sigma \cdot U $ which is open in $X^n$. As $E_n$ is continuous and is invariant under the $S_n$ action on $X^n$. By the universal property of the quotient map, there exists a unique continuous map $Q_n$.
\begin{equation}
    \begin{tikzcd}[row sep=large, column sep=large]
    X^{n} \arrow[r, "/S_n"', "q", two heads] \arrow[dr, "E_n"', two heads] 
        &  SP^n(X) \arrow[d, "\exists ! Q_n", dashed, two heads] \\
    &  \ExpS{n}{X} 
    \end{tikzcd}
\end{equation}
We have that $V \subset \ExpS{n}{X}$ is open  $ \iff E_n^{-1}(V)$ is open in $X^n$ $\iff E_n^{-1}(V) = q^{-1} ( Q_n^{-1}(V))$ is open in $X^n$ (as the triangle commutes) $\iff Q_n^{-1}(V)$ is open in $SP^n(X)$ as $q$ is a quotient map. Hence, $V$ is open in $\ExpS{n}{X}$ if and only if $Q_n^{-1}(V)$ is open in $SP^n(X)$.
\end{proof}

\begin{remark}
    The map $E_n$ is closed (and therefore perfect) \cite{handel}. However, it need not be an open map. As, if we consider $X=\mathbb{R}$ and $n=3$. We have that a point $(0,2,2)$ has a neighbourhood $U = (-1,1) \times (\frac{3}{2},3) \times (\frac{3}{2},3) $ however, the point  $(0,0,2) \in E^{-1}_3 (E_3 (U))$ but no neighbourhood of this point can lie wholly inside $E^{-1}_3 (E_3 (U))$.
    
\end{remark}

\subsubsection{Stratification of Finite Subset Spaces}\label{subsubsec:StratHO}
As any element $\{x_1,\ldots, x_k\} \in \ExpS{k}{X}$ is also an element of $\ExpS{n}{X}$ for $n \geq k$, we have a filtration 
\begin{equation} \label{eq:Fnfiltration}
    \emptyset \subset \ExpS{1}{X} \subset \ExpS{2}{X} \subset \cdots \subset \ExpS{n-1}{X} \subset \ExpS{n}{X}
\end{equation}
The inclusions are in fact embeddings: the topology of $ \ExpS{n-1}{X}$ endowed by the quotient map $E_{n-1}: X^{n-1} \twoheadrightarrow\ExpS{n-1}{X}$ coincides with the subspace topology derived from $ \ExpS{n}{X}$~\cite[Prop 2.04]{handel}. This filtration is related to the decomposition of $\ExpS{n}{X}$ into lower order unordered configuration spaces in~\cref{eq:Fn_UCk}. We can regard this filtration as a \emph{stratification} of $\ExpS{n}{X}$, where the strata --- the difference between successive spaces $\ExpS{i}{X}$ and $\ExpS{i-1}{X}$ in the filtration --- are given by the ordered configuration spaces $\ExpS{i}{X} \setminus \ExpS{i-1}{X} \homeo \UConf{i}{X}$. This stratification has been shown to be conically smooth (Definition A.5.5 of \cite{Lurie2017HigherAlgebra}) in \cite{AYALA2017903}.

We give an explicit description of the stratification via the quotient map $E_n : X^n \twoheadrightarrow \ExpS{n}{X} $
\begin{equation}
    E_n: (x_1 , x_2 , ... ,x_n) \mapsto \{x_1, x_2 , ... , x_n\}, 
\end{equation}
and a partition of $X^n$. We split $X^n$ into the following disjoint union; writing $X^{n}_i$ as the subset of points $(x_1, \ldots, x_n) \in X^n$ with $i$ distinct entries, we have
\begin{equation}\label{eq:xnsplit}
    X^n = X^{n}_1 \sqcup X^{n}_2 \sqcup\cdots  \sqcup X^{n}_n 
\end{equation}

We note that $X^{n}_n=\Conf{n}{X}$. If we recall the generalised diagonal $\Delta^{n-1} \subset X^n$ are the tuples $(x_1, \ldots, x_n)$ that have at most $n-1$ distinct entries, then by definition 
\begin{equation}
    \Delta^{n-1} = X^n_1 \sqcup \cdots \sqcup X^n_{n-1}\qand X^n = \Conf{n}{X} \sqcup \Delta^{n-1}.
\end{equation}
We now describe the inclusion $\ExpS{n-1}{X} \hookrightarrow \ExpS{n}{X}$ in the stratification~\cref{eq:Fnfiltration} by decomposing $\ExpS{n}{X} = E_n(X^n)$ in terms of the images of the components $E_n(\Delta^{n-1})$ and $E_n(\Conf{n}{X})$ whose union is $\ExpS{n}{X}$. We first observe that the images $E_n(\Delta^{n-1})$ and $E_n(\Conf{n}{X})$ are disjoint and thus partition $\ExpS{n}{X}$:
\begin{equation} \label{eq:fn_first_decomposition}
    \ExpS{n}{X} = E_n(\Conf{n}{X)} \sqcup E_n(\Delta^{n-1}).
\end{equation}
This follows immediately from the definitions of the sets $\Conf{n}{X}$ and $\Delta^{n-1}$. Since $E_n(\Delta^{n-1}) \subset \ExpS{n}{X}$ are exactly the finite subsets with at most $n-1$ elements, 
\begin{equation}
    E_{n}(\Delta^{n-1}) = \ExpS{n-1}{X}.
\end{equation}
Because $(x_1,\ldots, x_n) \in \Conf{n}{X}$ are tuples with $n$ distinct elements, elements of $E_n(\Conf{n}{X})$ are exactly subsets $\{x_1, \ldots,x_n\}$ with cardinality $n$, thus $E_n(\Conf{n}{X}) \cap \ExpS{n-1}{X} = \emptyset$. We now show that $E_n(\Conf{n}{X})$ is in fact homeomorphic to $\UConf{n}{X}$; consequently~\cref{eq:fn_first_decomposition} implies the following.
\begin{proposition} \label{prop:fn_top_homeo}
    $\ExpS{n}{X} \setminus \ExpS{n-1}{X}  = E_n(\Conf{n}{X}) \homeo \UConf{n}{X}$.
\end{proposition}

\begin{proof}
    
    Recall from~\Cref{prop:SP_to_Exp_quotient} that $E_n = Q_n  \circ q$. Thus, we can write $E_n(\Conf{n}{X})$ as $Q_n(q(\Conf{n}{X)}) = Q_n(\UConf{n}{X})$. As $\UConf{n}{X} \subset SP^n(X)$ are precisely the $n$ unordered elements of $X$ each with multiplicity 1, the restriction of $Q_n$ to $\UConf{n}{X}$ is a continuous bijection of $\UConf{n}{X}$ onto its image. To show that it is a homeomorphism, we now show that if $V$ is open in $\UConf{n}{X}$, then $Q_n(V)$ is open in $Q_n(\UConf{n}{X})$. 
    
    We note if $V$ is open in $\UConf{n}{X}$, then it is open in $SP^n(X)$ as $\UConf{n}{X}$ is an open subspace of $SP^n(X)$. Also, as $\Conf{n}{X}$ is open in $X^n$, $Q_n(\UConf{n}{X})$ is open in $\ExpS{n}{X}$. Hence, we shall show $Q_n(V)$ is open in the whole of $\ExpS{n}{X}$.
    
    If $V$ is open in $\UConf{n}{X}$ then $q^{-1}(V)$ is open in $X^n$. We have that $Q_n(V)$ is open in $\ExpS{n}{X}$ if and only if $E_n^{-1}(Q_n(V))$ is open in $X^n$. Since $Q_n$ restricted to $\UConf{n}{X}$ is a continuous bijection, we have $E_n^{-1}(Q_n(V))= q^{-1} \circ Q_n^{-1} (Q_n(V)) = q^{-1}(V)$.
    Hence, $Q_n(V)$ is open in $\ExpS{n}{X}$. Thus the restriction of $Q_n$ from $\UConf{n}{X}$ to $Q_n(\UConf{n}{X})=E_n(\Conf{n}{X})$ is a homeomorphism.
    
\end{proof}

Furthermore, for $i \leq n$, because the subspace topology of $\ExpS{i}{X}$ as a subset of $\ExpS{n}{X}$ coincides with the quotient topology endowed by $E_i$, we can show by induction that~\Cref{prop:fn_top_homeo} implies
\begin{equation}
    E_n(X^n_i) = \ExpS{i}{X} \setminus \ExpS{i-1}{X} \homeo \UConf{i}{X}
\end{equation}
for every $i \leq n$ in the stratification~\cref{eq:Fnfiltration}.
This allows us to decompose $\ExpS{n}{X}$ in the following manner, following~\cref{eq:Fn_UCk}: 

\begin{equation}
    \begin{tikzpicture}
    \node (top) at (-2,1) {$X^n$};
    \node (exp) at (-2,0) {$\ExpS{n}{X}$};

    \node (eq) at (-1,1) {$=$};

    \node (eq1) at (-1,0) {$=$};
    \node (X1) at (0,1) {$X^{n}_1$};
    \node (X2) at (4,1) {$X^{n}_2$};
    \node (Xn) at (10.6,1) {$X^n_n = \Conf{n}{X}$};
    
    \node (A1) at (0.2,0) {$E_{n}(X^{n}_1)$};
    \node (A2) at (4,0) {$E_{n}(X^{n}_2)$};
    \node (An) at (10.6,0) {$E_{n}(\Conf{n}{X})$};

    \node (eq2) at (-1,-1) {};

    \node (X) at (0,-1) {$X$};
    \node (X) at (4,-1) {$\UConf{2}{X}$};
    \node (X) at (10.6,-1) {$\UConf{n}{X}$};

    \node at (-2,0.5) {$\downarrow E_n$ };
  
    \node at (1.4,-1) {$\sqcup $};

    \node at (1.4,1) {$\sqcup $};
    \node at (1.4,0) {$\sqcup $};
    \node at (7 ,1) {$\sqcup \dots \sqcup$};
    \node at (7 ,0) {$\sqcup \dots \sqcup$};
    \node at (7 ,-1) {$\sqcup \dots \sqcup$};

    \node at (0,0.5) {$\downarrow E_n$};
    \node at (10.6,0.5) {$\downarrow E_n$};
    \node at (4,0.5) {$\downarrow E_n$};

    \node at (0,-0.5) {$\rotatebox{90}{$\,\approx$} $};
    \node at (10.6,-0.5) {$\rotatebox{90}{$\,\approx$}  $};
    \node at (4,-0.5) {$\rotatebox{90}{$\,\approx$}  $};

\end{tikzpicture}
\end{equation}
Furthermore, the boundary of $\UConf{i}{X}$ in $\ExpS{n}{X}$ is exactly $\ExpS{i-1}{X}$.

\begin{proposition}
\label{prop:boundary}    

    The boundary of $E_{n}(\Conf{n}{X})$ in $ \ExpS{n}{X}$ is $\ExpS{n}{X} \backslash E_{n}(\Conf{n}{X}) =E_n(\Delta^{n-1}) \approx \ExpS{n-1}{X}$.
\end{proposition}

\begin{proof}
     If $X$ is hausdorff then for each point in $q \in \Conf{n}{X}$ we can find a selection of open sets $U_{q_i}$ such that $q \in \prod_{i=1}^{n} U_{q_i} $ and $U_{q_i} \cap U_{q_j} = \emptyset$ for $i\neq j$. Hence $\Conf{n}{X}$ is open. We have that $E_{n}(\Conf{n}{X})$ is open if and only if $E^{-1}_n(E_{n}(\Conf{n}{X}))$ is open in $X^n$ which is precisely $\Conf{n}{X}$ and thus open. Hence, $E_{n}(\Conf{n}{X})$ is open in $\ExpS{n}{X}$.

So the boundary of $E_{n}(\Conf{n}{X})$ in $\ExpS{n}{X}$ is $\partial E_{n}(\Conf{n}{X}) = \text{cl}({E_{n}(\Conf{n}{X})}) \backslash \text{int}(E_{n}(\Conf{n}{X})) = \text{cl}(E_{n}(\Conf{n}{X}) \backslash E_{n}(\Conf{n}{X})$. Suppose that there exists an $ x \in \ExpS{n}{X}$ such that $x \notin \text{cl}(E_{n}(\Conf{n}{X}))$. Then there exists an open set $U \subset \ExpS{n}{X}$ such that $x \in U$ and $ U \cap E_{n}(\Conf{n}{X}) = \emptyset$. We have that $U$ is open when $Y := \{ q \in X^n : [q]_{E_{n}} \in U \}   $  is open. If $q \in Y $ then $q \notin \Conf{n}{X}$ hence when $q \in Y$, $q$ has two entries the same, w.l.o.g. $q=(q_1, q_1, \dots , q_{n-1})$ as $X$ is hausdorff we have that $ q \in U_{q_1} \times U_{q_1} \times \prod_{i=2}^{n - 1} U_{q_i} \subset Y $ where $U_{q_i} \cap U_{q_j} = \emptyset$ for $i\neq j$. This means that $U_{q_1}$ must be the singleton $\{q_1\}$, as $X$ isn't given the discrete topology, this set isn't open (we exclude the discrete topology as all boundaries are empty). Hence we have that all $x \in \ExpS{n}{X}$ are in $\text{cl}(E_{n}(\Conf{n}{X}))$. So we have that, $\partial E_{n}(\Conf{n}{X}) = \ExpS{n}{X} \backslash E_{n}(\Conf{n}{X})=E_n(\Delta^{n-1}) \approx \ExpS{n-1}{X}$.
\end{proof}

\begin{proposition}
    The boundary of $E_n(X^n_i) $ in $\ExpS{i}{X} \subset \ExpS{n}{X}$ is $\ExpS{i}{X} \backslash E_{n}(X^n_i) =E_i(\Delta^{i-1}) \approx \ExpS{i-1}{X}$. 
\end{proposition}

\begin{proof}

We have that $E_n(X^n_i) \seteq E_i(X^i_i) \approx \UConf{i}{X}$, which is open in $\ExpS{i}{X}$. By proposition 2.04 of \cite{handel} we have that $E_n(X^n_i)$ is open in $\ExpS{i}{X} \subset \ExpS{n}{X}$. The closure of $E_n(X^n_i)$ in $\ExpS{i}{X} \subset \ExpS{n}{X}$ is $\ExpS{i}{X}$, with an analogous argument as the above proposition. Hence, in $\ExpS{i}{X} \subset \ExpS{n}{X}$ we have, $\partial E_n(X^n_i) = \ExpS{i}{X} \backslash E_n(X^n_i) = E_i(\Delta^{i-1}) \approx \ExpS{i-1}{X}$.
    
\end{proof}

\begin{figure}[ht]
\centering
\begin{tikzpicture}[
    scale=0.95,
    >=Stealth,
    line width=0.9pt,
    arrow/.style={
        postaction={decorate},
        decoration={markings, mark=at position 0.55 with {\arrow{>}}}
    },
    doublearrow/.style={
        postaction={decorate},
        decoration={markings, mark=at position 0.55 with {\arrow{>>}}}
    }
]
\begin{scope}[shift={(0,0)}]
    \draw (0,0) rectangle (3,3);
    \draw[dashed] (-0.2,-0.2) -- (3.2,3.2);
    \draw[arrow] (0,3) -- (3,3);
    \draw[arrow] (0,0) -- (3,0);
    \draw[doublearrow] (0,1.2) -- (0,1.8);
    \draw[doublearrow] (3,1.2) -- (3,1.8);
\end{scope}
\begin{scope}[shift={(3.6,0)}]
    \draw (0,0) -- (3,3) -- (0,3) -- cycle;
    \draw[dashed] (-0.2,3.2) -- (1.5,1.5);
    \draw[doublearrow] (0,1.1) -- (0,1.9);
    \draw[doublearrow] (1.6,3) -- (2.4,3);
\end{scope}
\begin{scope}[shift={(7,0)}]
   \draw (0,0) -- (0,3) -- (1.5,1.5) -- cycle;
\draw (0.1,3.1) -- (3.1,3.1) -- (1.6,1.6) -- cycle;
\draw[doublearrow] (0,1.1) -- (0,1.9);
\draw[doublearrow] (1.6,3.1) -- (2.4,3.1);
  \draw[arrow] (1.5,1.5) -- (0,3);
    \draw[arrow] (1.6,1.6) -- (0.1,3.1);
\end{scope}
\begin{scope}[shift={(10.5,0)}]
    \draw (1.5,3) -- (3,1.5) -- (1.5,0) -- (0,1.5) -- cycle;
    \draw (1.5,0) -- (1.5,3);
    \draw[doublearrow] (1.5,1.1) -- (1.5,1.9);
    \draw[arrow] (3,1.5) -- (1.5,3);
    \draw[arrow] (0,1.5) -- (1.5,0);   
\end{scope}
\end{tikzpicture}

\caption{A cut and gluing argument showing that $\ExpS{2}{S^1} = (S^1 \times S^1) / S_2$ is the M\"obius band. Expressing the torus as a square with opposite sides identified, the $S_2$ action `folds' the torus across the dashed diagonal. The resultant quotient space is the triangle with the indicated pair of edges identified with the orientation given in the diagram. To see that this space is homeomorphic to a M\"obius band $\Mob$, we cut along the opposite diagonal giving us two triangles and another identification. We then identify the vertical and horizontal sides of the original triangle to obtain the square on the right, with a pair of opposing edges identified with opposite orientation.}
\label{fig:cutpa}
\end{figure}
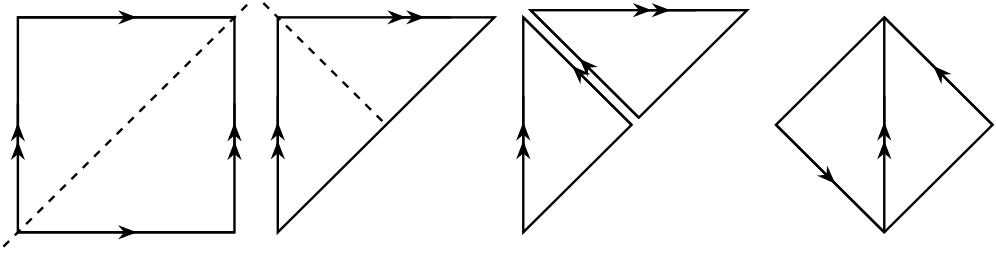

\subsubsection{The Second Finite Subset Space of the Circle as a M\"obius band}
\begin{figure}
    \centering
    \begin{subfigure}{0.45\textwidth}
        \centering
        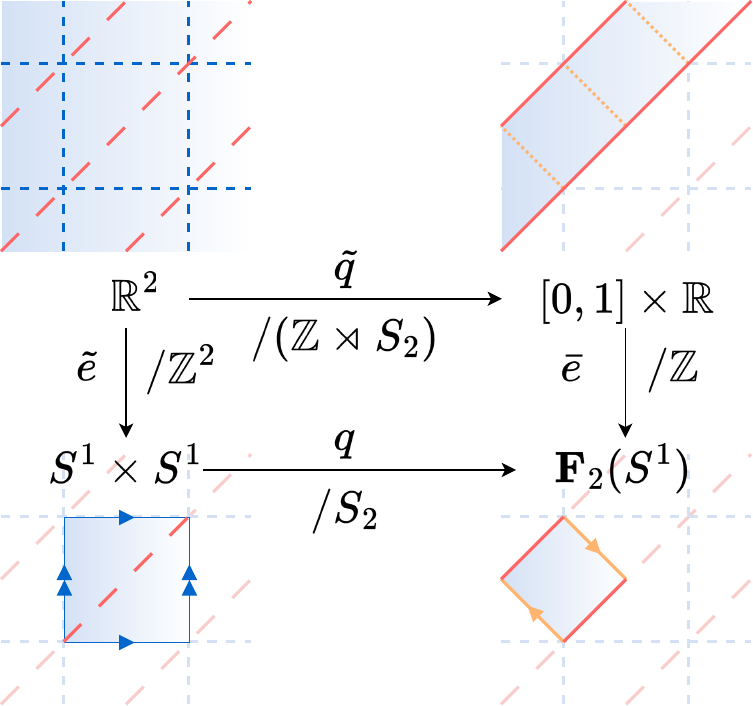
    
    \caption{An illustration of the universal covering $\bar{\scre}$ of $\ExpS{2}{S^1}$, where we choose a square fundamental domain (demarcated by dashed orange lines) for the $\Z$-action on $[0,1] \times \R$.  The action maps successive pairs of the orange dotted lines onto one another with \emph{opposite} orientation,}
    \end{subfigure}\hfill
    \begin{subfigure}{0.45\textwidth}
        \centering
        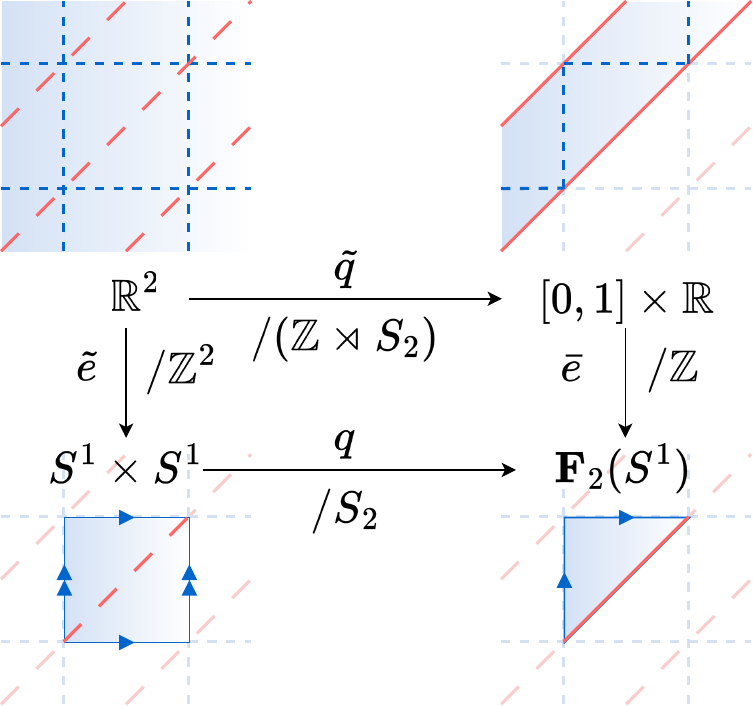
        \caption{An illustration of the universal covering $\bar{\scre}$ of $\ExpS{2}{S^1}$, where we choose a triangular fundamental domain (demarcated by blue dashed lines) for the $\Z$-action on $[0,1] \times \R$. The $\Z$-action on a vertical blue edge of the triangular fundamental domain first reflects it across the red diagonal, and then translates it up by one vertical unit onto the horizontal edge.}
    \end{subfigure}
    \caption{An illustration of~\Cref{prop:Morton}, which describes how the two point finite subset space $\ExpS{2}{S^1}$ can be expressed as a M\"obius band by relating the universal coverings $\tilde{\scre}$ and $\bar{e}$ of $S^1 \times S^1$ and $\ExpS{2}{S^1}$ respectively, by quotient maps $q$ and $\tilde{q}$. The quotient map $q \circ \tilde{\scre}$ can be decomposed into a composition $\bar{\scre} \circ \tilde{q}$, where $\tilde{q}$ is the quotient of $\R^2$ by a $(\Z \rtimes S_2)$-action; the diagonal strip $\homeo [0,1] \times \R$ (as a subset of $\R^2$) is not only a fundamental domain of the $(\Z \rtimes S_2)$-action, but also homeomorphic to the quotient space. The map $\bar{\scre}$ is the quotient by the glide reflection, a $\Z$-action on the infinite strip.  We also note that the diagonal $\Delta \subset S^1 \times S^1$ (indicated by the red-dashed lines) is sent to the circle boundary of $\ExpS{2}{S^1}$; this is a special case of~\Cref{prop:boundary}. It can also be seen by first lifting the diagonals to its preimage in $\R^2$ (also indicated as red dashed lines; as they form the boundary of the manifold with boundary $\R^2/(\Z \rtimes S_2) = [0,1] \times \R$,  and $\bar{\scre}$ is a covering map which sends manifold boundaries to manifold boundaries, the image of $\Delta$ under $q$ is the boundary of $\ExpS{2}{S^1}$. }
    \label{fig:mobius_commutative}
\end{figure}

In this section, we recall that $\ExpS{2}{S^1}$ is homeomorphic to the M\"{o}bius band $\Mob$. This observation was first made in~\cite{BorsukUlam1931SymmetricProducts}, and a visual explanation can be found in~\cite[\S 2.2]{tuffleyexp}, which we reproduce and explain in~\Cref{fig:cutpa}.  Here we use a more technical construction due to~\cite{Morton1967-or}. Morton's construction involves deriving the universal covering $\bar{\scre}$ of $\ExpS{2}{S^1}$ by an infinite strip $[0,1] \times \R$, through the universal covering $\tilde{\scre}$ of $S^1 \times S^1$ by $\R^2$. Recall we have a universal covering of $S^1$ by $\scre: \R \to S^1$ given by $\scre(t) = e^{2\pi \imath t}$, and a universal covering of $S^1 \times S^1$ given by products $\tilde{\scre} = \scre \times \scre: \R^2 \to S^1 \times S^1$. We first observe that we can obtain a quotient map from $\R^2$ to $\ExpS{2}{S^1}$ by composing $\tilde{\scre}: \R^2 \to S^1 \times S^1$ with $q: S^1 \times S^1 \twoheadrightarrow \ExpS{2}{S^1}$. Morton showed that this successive quotient by a $\Z^2$-action on $\R^2$ and an $S_2$-action on $S^1 \times S^1$ is equivalent to a quotient by the action of a semi-direct product $\Z^2 \rtimes S_2$ on $\R^2$. Furthermore, the quotient $q \circ \tilde{e}$ can be decomposed into another pair of successive quotients $\bar{\scre} \circ \tilde{q}$, where the first map $\tilde{q}$ is the quotient of $\R^2$ by the action of a normal subgroup $N \cong \Z \rtimes S_2$ of $G$, and $\bar{\scre}$ is the universal covering of the M\"obius band $\Mob$ by a $\Z$-action on $\R^2/N \homeo [0,1] \times \R$. We summarise his construction in the~\Cref{prop:Morton} below, and give details of the precise construction in~\Cref{app:mobius}. 

\begin{proposition}[\cite{Morton1967-or}]\label{prop:Morton} Let $q: S^1 \times S^1 \twoheadrightarrow \ExpS{2}{S^1}$ be the quotient of the torus by the $S_2$-action $\sigma \cdot (z_1, z_2)  = (z_{\sigma(1)}, z_{\sigma(2)})$, and $\tilde{e}: \R^2 \twoheadrightarrow S^1 \times S^1$ be the universal covering map of $S^1 \times S^1$. Then there are quotients by group actions $\tilde{q}: \R^2 \twoheadrightarrow \R^2/N \xrightarrow{\homeo}  [0,1] \times \R$ , and $\bar{\scre}: [0,1] \times \R \twoheadrightarrow \ExpS{2}{S^1}$, such that the following diagram commutes:
    \begin{equation}\label{eq:master_mob_diagram}
    \begin{tikzcd}[ampersand replacement=\&]
	{\R^2} \& {[0,1] \times \R} \\
	{S^1 \times S^1} \& \ExpS{2}{S^1}
	\arrow["\tilde{q}", "{/(\Z \rtimes S_2)}"', two heads, from=1-1, to=1-2]
	\arrow["{/\Z^2}", "\screii"', two heads, from=1-1, to=2-1]
	\arrow["{/\Z}", "\bar{\scre}"', two heads,  from=1-2, to=2-2]
	\arrow["q", "{/S_2}"', two heads, from=2-1, to=2-2]
\end{tikzcd}.
\end{equation} 
Furthermore, 
\begin{enumerate}[label = (\roman*)]
    \item  $\bar{\scre}$ is the universal covering of $\ExpS{2}{S^1}$ as a M\"obius band $\Mob$; and  
    \item  There is an embedding $\lambda:  [0,1] \times \R \hookrightarrow \R^2$ of $\tilde{q}$ whose image is $\{(x_1, x_2) \ : \ x_2 - x_1 \in [0,1]\}$, and $\tilde{q} \circ \lambda = \iden_{[0,1] \times \R}$.
    \item The boundary $\partial \ExpS{2}{S^1}$, homeomorphic to $S^1$, is the set of points $\{\{z\} \ : \ z \in S^1\} = \ExpS{1}{S^1} = q(\Delta)$, where $\Delta$ is the diagonal $\{(z,z) \ : \ z \in S^1\} \subset S^1 \times S^1$.
\end{enumerate}
\end{proposition}
Using~\Cref{prop:Morton}, we can show that the map $q \circ \tilde{\scre}: \R^2 \to \ExpS{2}{S^1}$ admits local inverses. That is, for every $w \in \ExpS{2}{S^1}$, there is a neighbourhood $W \ni w$ such that there is an embedding $\varphi: W \hookrightarrow \R^2$ such that $(q \circ \tilde{\scre}) \circ \varphi = \iden_W$. We construct $\varphi$ using lifts of the covering map $\bar{\scre}$, along with the embedding $\lambda$. By taking neighbourhoods for each $w \in \ExpS{2}{S^1}$, we obtain a collection of coordinates charts $\Phi = \{(\varphi, W)\}$ for $\ExpS{2}{S^1}$. We show that such charts can be chosen such that the collection forms an atlas for $\ExpS{2}{S^1}$, which is a manifold with boundary as it is homeomorphic to a M\"obius band. 

\begin{proposition} \label{prop:q_local_invert} For every $w \in \ExpS{2}{S^1}$, there is a neighbourhood $W \ni w$ such that there is an embedding $\varphi: W \hookrightarrow \R^2$, so that $(\bar{e} \circ \tilde{q}) \circ \varphi = \iden_W$. If a collection of such neighbourhoods $\cW = \{W_j\}_{j \in \cJ}$ is an open cover of $\ExpS{2}{S^1}$, then any choice of such charts $\Phi = \{(\varphi_j, W_j)\}_{j \in \cJ}$ for each $W_j$ is an atlas for $\ExpS{2}{S^1}$.  
\end{proposition}
\begin{proof}
    Since $\bar{\scre}$ is a covering map, for $W$ sufficiently small neighbourhood of $w \in \ExpS{2}{S^1}$, we can find an embedding $\bar{\varphi}: W \hookrightarrow [0,1] \times \R$. Composing it with the embedding $\lambda: [0,1] \times 
    \R \hookrightarrow \R^2$ in~\Cref{prop:Morton}, we have an embedding $\varphi = \lambda \circ \bar{\varphi}: W \hookrightarrow \R^2$.  Because $\tilde{q} \circ \lambda = \iden_{[0,1] \times \R}$, and $\bar{\scre} \circ \bar{\varphi} = \iden_W$, 
    \begin{equation*}
        (\bar{\scre} \circ \tilde{q}) \circ \varphi = \bar{\scre} \circ (\tilde{q} \circ \lambda) \circ \bar{\varphi} = \bar{\scre} \circ \bar{\varphi} = \iden_W.
    \end{equation*}
    We now show that any collection $\Phi$ of such coordinate charts form an atlas. For $W, W' \subset \ExpS{2}{S^1}$ with non-empty intersection that admit lifts $\bar{\varphi}, \bar{\varphi}'$ respectively of $W$ and $W'$, we have some $n \in \Z$ such that $n \star \bar{\varphi}(W \cap W') = \bar{\varphi}'(W \cap W') $.  Let $T: \bar{\varphi}(W \cap W') \to \bar{\varphi}'(W \cap W')$ be that homeomorphism; then we have a transition map $\varphi (W  \cap W') \xrightarrow{\homeo} \varphi'(W \cap W')$, given by
    \begin{equation}
        (\lambda Tq\rvert_{\varphi(W)} )\circ \varphi\rvert_{W \cap W'} = \lambda T\circ (q \lambda\rvert_{\bar{\varphi}(W)} )\circ \bar{\varphi}\rvert_{W \cap W'} = \lambda (T \bar{\varphi}\rvert_{W \cap W'}) = \lambda \bar{\varphi}'\rvert_{W \cap W'} = \varphi' \rvert_{W \cap W'}. 
    \end{equation}
    where in the first and second equality we used the fact that $\varphi = \lambda \bar{\varphi}$, and $q \lambda\rvert_{\bar{\varphi}(W)} = \iden_{\bar{\varphi}(W)}$ respectively.  
\end{proof}
\begin{corollary}\label{cor:q_local_invert}
    Let $\varphi: W \to \R^2$ be a chart for $\ExpS{2}{S^1}$, as described in~\Cref{prop:q_local_invert}. Then the map $\psi := \tilde{\scre} \circ \varphi : W \to S^1 \times S^1$ is an embedding of $W$ such that $ q\circ \psi = \iden_W$.
\end{corollary}
\begin{proof}
    Because of the commutative square in~\cref{eq:master_mob_diagram}, we can also show that $\psi: W \to S^1 \times S^1$ is an embedding:
    \begin{equation*}
       q\circ \psi = q \circ (\tilde{\scre} \circ \varphi) =  (q \circ \tilde{\scre}) \circ \varphi =  (\bar{\scre} \circ \tilde{q}) \circ \varphi = \iden_W. 
    \end{equation*}
    Because $q\rvert_{\psi(W)}$ is a continuous bijective map, and $\psi$ is also a continuous bijection onto its image, we thus establish that $\psi$ is an embedding (homeomorphism onto its image). Furthermore, $q\circ \psi(W) = W$, so $\psi(W) \subset \fibre{q}{W}$. 
\end{proof}
Furthermore, we show that the coordinate charts $\varphi: W \hookrightarrow \R^2$ inherit from the coordinate charts of the choice of $S_2$-orbit representatives $\psi(W)$ in $S^1 \times S^1$. We first establish some notation. If  $\{(\theta_i, U_i)\}_{i \in \cI}$ is an atlas for $S^1$, then we let $\Theta = \{(\vartheta_{ab}, U_{ab})\}_{a,b \in \cI}$ denote the induced atlas of $S^1 \times S^1$, where $\vartheta_{ab}  := \theta_a \times \theta_b$, and $U_{ab} := U_a \times U_b$. 

\begin{proposition}\label{prop:mobius_atlas} There exists an open cover $\cW = \{W_j\}_{j \in \cJ}$ of $\ExpS{2}{S^1}$, and an atlas $\{(\theta_i, U_i)\}_{i \in \cI}$ for $S^1$ whose charts $\theta_i$ are lifts of the universal covering $\scre: \R \twoheadrightarrow S^1$, where the following are satisfied for any $j \in \cJ$:
\begin{itemize}
    \item  There is a local embeddings $\psi_j : W_j \hookrightarrow S^1 \times S^1$, such that $q\circ \psi_j = \iden_{W_j}$;
    \item There exists some $a,b \in \cI$ where $\psi_j(W_j) \subseteq U_{ab} = U_a \times U_b$; and
    \item For $\vartheta_{ab} := \theta_a \times \theta_b$, the map $\varphi_j =\vartheta_{ab} \circ \psi_j: W_j \to \R^2$ is a chart.
\end{itemize}
Furthermore, the set of charts constructed as thus $\Phi = \{(\varphi_j, W_j)\}_{j \in \cJ}$ is an atlas for $\ExpS{2}{S^1}$.
\end{proposition}
\begin{proof}
    We prove this proposition by showing that the atlas for $\ExpS{2}{S^1}$ in~\Cref{prop:q_local_invert} satisfies these conditions, where the open cover consists of neighbourhoods $W$ of all $w \in \ExpS{2}{S^1}$. For each such $W$, we have a chart $\varphi = \lambda \circ \bar{\varphi}$ where $\bar{\varphi}: W \hookrightarrow [0,1] \times \R$ is a lift of the covering map $\bar{\scre}$. The lift $\bar{\varphi}$ is completely determined by where it sends $w$ in $\fibre{\bar{\scre}}{w}$. Since $\lambda$ is an embedding, this is equivalent to saying $\varphi$ is completely determined by where it sends $w$ in $\fibre{(\bar{\scre} \circ \tilde{q})}{w}$. Recall from~\Cref{prop:q_local_invert} that these charts form an atlas regardless of the choice of representative of $w$ in $\fibre{(\bar{\scre} \circ \tilde{q})}{w}$.

   We now note that the subset $F  =\{(x_1,x_2) \in [0,1) \times [0,1) \ : \ x_2 -x_1 \in [0,1]\} \subset \R^2$ is in the image of $\lambda$, and contains exactly one point in the fibre $\fibre{(\bar{\scre} \circ \tilde{q})}{w}$ for any $w \in \ExpS{2}{S^1}$. Let us set the charts $\varphi: (W,w) \hookrightarrow (\R^2, (x,y))$ in the atlas to be those that sends $w$ to $(x,y) \in F$. For each $(x,y) \in F$, consider $\tilde{U}_{xy} = \tilde{U}_{x} \times \tilde{U}_{y}$, where $\tilde{U}_{x} = (x-1/2, x+ 1/2)$. We note that $\scre\rvert_{\tilde{U}}$ is a homeomorphism onto its image $U_x := \scre(\tilde{U}_x) \subset S^1$.  Let $\theta_x: U_x \to \tilde{U}_x$ be the lift of the covering ${\scre}$ at $U_x \subset S^1$ such that ${\scre} \theta_x = \iden_{U_x}$. As these are lifts of covering maps from Euclidean space, the collection of such charts $\{(\theta_x,U_x)\}_{x \in [0,1)}$ form atlas for $S^1$ which owes its transition maps to the $\Z$-action on $\R$. Similarly, the product of charts $\Theta = \{(\vartheta_{xy}, U_{xy})\}_{x,y \in [0,1)}$ also form an atlas for $S^1 \times S^1$, where $\vartheta_{xy}(U_{xy}) = \tilde{U}_{xy}$, and $\tilde{\scre}\circ \vartheta_{xy} = \iden_{U_{xy}}$. 

   Now let us assume we have chosen the neighbourhood $W$ of $w$ to be sufficiently small such that $\varphi(W) \subset \tilde{U}_{xy}$, which we are allowed to do as $\varphi$ is continuous. Let $\psi = \tilde{\scre} \circ \varphi: W \to S^1\times S^1$ be the embedding such that $q\circ \psi = \iden_{W}$, as described in~\Cref{cor:q_local_invert}. As $\psi(W) = \tilde{\scre} \circ \varphi(W) \subset \tilde{\scre}(\tilde{U}_{xy}) = U_{xy}$, we can apply $\vartheta_{xy}$ to the image of $\psi$; thus
   \begin{equation*}
       \vartheta_{xy} \circ \psi =  (\vartheta_{xy} \circ \tilde{\scre}\rvert_{\varphi(W)}) \circ \varphi = \varphi.
   \end{equation*}
\end{proof}
The existence of compatible atlases implied by~\Cref{prop:mobius_atlas} allows us to translate local coordinate descriptions of symmetric maps of $S^1 \times S^1$ to its induced map on $\ExpS{2}{S^1}$.  If $f: S^1 \times S^1 \to X$ satisfies $f(z_1, z_2) = f(z_2,z_1)$, then recall there exists a unique map $\bar{f}$ such that
\[\begin{tikzcd}[ampersand replacement=\&]
	{S^1 \times S^1} \& X \\
	{\ExpS{2}{S^1}}
	\arrow["f", from=1-1, to=1-2]
	\arrow["q"', two heads, from=1-1, to=2-1]
	\arrow["{\exists ! \bar{f}}", dashed, from=2-1, to=1-2]
\end{tikzcd}.\]
If we restrict $\bar{f}$ to a cover element $W_j \subset \ExpS{2}{S^1}$ from the atlas $\Phi$, we can express $\bar{f}\rvert_{W_j}$ in local coordinates $\varphi_j(V) \subset \R^2$ via $\bar{f} \circ \inv{\varphi_j}$. We can then apply~\Cref{prop:mobius_atlas} to relate it to a local coordinate expression of $f$ on a choice of orbit representatives $\psi_j(W_j) \subset S^1 \times S^1$:
\begin{align}
    \bar{f} \circ \inv{\varphi_j} &= \bar{f} \circ \inv{\vartheta_{ab} \circ \psi_j} = (\bar{f} \circ \inv{\psi_j}) \circ  \inv{\vartheta_{ab}}\rvert_{\varphi_j(W_j)} = (\bar{f} \circ q\rvert_{\psi_j(W)}) \circ  \inv{\vartheta_{ab}} \rvert_{\varphi_j(W_j)} \nonumber \\
    &= f  \circ  \inv{\vartheta_{ab}} \rvert_{\varphi_j(W)} . 
\end{align}

We see immediately if $X$ is a manifold, and $f \in C^k(S^1 \times S^1,X)$, then $\bar{f} \in C^k(\ExpS{2}{S^1}, X)$.  Expressing functions on $\ExpS{2}{S^1}$ in local coordinates of $S^1 \times S^1$ is necessary for the differential topology arguments in~\Cref{sec:smooth-case}.

\subsection{Morse Theory and Persistent Homology}
\label{sec:Background:ph-sublevelset}

We now briefly introduce the standard theory surrounding the fundamental areas we use to prove our results. We give a basic introduction of the terms and main ideas in persistent homology, and discuss how smooth Morse theory can help describe the persistent homology of sublevel set filtrations of manifolds with boundary by smooth functions. We also touch on ideas from Conley index theory to help describe the persistent homology of filtered simplicial complexes. 

\subsubsection{Persistent Homology of Sublevel Set Filtrations}

Let $f: X \to \R$ be a continuous function, and $X^t := \sublevel{f}{t}$ be the sublevel set of $f$ under $t$. Consider then following filtration of sublevel sets, indexed over $\R$.  That is, a collection of inclusions maps $X^s \hookrightarrow X^t$, whenever $s \leq t$:
\begin{equation}
    X^\bullet: (X^s \xhookrightarrow{s\leq t} X^t)
\end{equation}
Let $H_i(-)$ denote the $i$-th singular homology functor with field coefficients $\mathbb{F}$, that takes maps between topological spaces to linear maps between their degree $i$ homology groups. Passing the inclusions in the filtration $X^\bullet$ to maps between homology groups we obtain a sequence of maps, the $i$\textsuperscript{th}-\emph{persistence module} ~\cite[sec. 1.1]{Chazal2016-bo} of the sublevel set filtration of $X$:
\begin{equation}
  H_i(X^\bullet) :\quad   \left(H_i(X^s)  \xrightarrow{s \leq t} H_i(X^t) \right)
\end{equation}
A persistence module $\mathbb{V}$ over $\mathbb{R}$ is an  $\R$-indexed family of vector spaces $V^s$, equipped with linear maps $$v^t_s : V^s\rightarrow V^t$$ for $s \leq t$, which satisfy the commutation relation $v_s^t \circ v_r^s = v_r^t$ whenever $r\leq s\leq t $, and $v_t^t = \iden_{V^t}$.

We say a persistent module is quadrant tame or \emph{q-tame}, if the rank of each map $v_s^t$ is finite for $s < t$. If all vector spaces $V^t$ are finite-dimensional, then the persistence module is said to be \emph{pointwise finite dimensional} or simply \emph{tame}. If $X$ is homeomorphic to the geometric realisation of a finite simplicial complex, and $f$ is continuous, then the persistent homology of the sublevel set  filtration $H_i(X^\bullet)$ is q-tame~\cite[Theorem 2.22]{Chazal2016-bo}.

Given a persistence module, we can obtain a practically convenient invariant called a persistence diagram. We define persistence diagrams for q-tame persistence modules, following~\cite[Prop. 2.3]{Chazal2016-bo}, based on the rank functions of the persistence module and the notion of persistence measures. Let $\bar{\R} = \R \cup \{\pm \infty\}$, and consider the open extended upper half plane $\mathcal{H} \subset \bar{\R}^2$
\begin{equation*}
   \mathcal{H}:= \{(s,t) \ : \ -\infty \leq s < t \leq \infty\}.
\end{equation*}
Let $\mathcal{R}$ be the set of all rectangles contained in $\mathcal{H}$
\begin{equation*}
    \mathcal{R} = \{[a,b]\times[c,d]  \ : \ a<b<c<d\}.
\end{equation*}
For a persistence module $\mathbb{V}$, let $r^t_s := \rank v^t_s$ be the \emph{rank function} that assigns to every pair $s \leq t$ the rank of the linear map $v^t_s$.
The \emph{persistence measure} of a q-tame persistence module $\mathbb{V}$ \cite[Prop. 2.3]{Chazal2016-bo} is then the function $\mu_\mathbb{V}: \mathcal{R} \to \Z$, 
\begin{equation} \label{eq:pers_measure}
\mu_{\mathbb{V}}([a,b] \times [c,d]) = r^c_b - r^c_a -r^d_b + r^d_a,
\end{equation}
Since the image of the composition of linear maps satisfies $\imag (fg) \cong \imag g / (\ker f \cap \imag g)$ , we can write~\cref{eq:pers_measure} in the following form, following~\cite[Prop. 2.5]{Chazal2016-bo}: 
\begin{equation} \label{eq:pers_measure_b}
\mu_{\mathbb{V}}([a,b] \times [c,d]) = \dim \left(\frac{\imag v_b^c \cap \ker v_c^d}{\imag v_a^c \cap \ker v_c^d} \right). 
\end{equation}
As such, the persistence measure is non-negative. From~\cref{eq:pers_measure_b}, $\mu_{\mathbb{V}}([a,b] \times [c,d])$ is zero if and only if $v_c^d$ is  injective and $v_a^b$ is surjective.  We can interpret the persistence measure on $\mu_{\mathbb{V}}([a,b] \times [c,d])$ as the dimension of the vector space consisting of equivalence classes of vectors that are `born' between $[a,b]$ in the filtration, and `die' between $[c,d]$. 

This interpretation of the persistence measure as an enumeration of equivalence classes of vectors in the persistence module can be made precise using the following formulation in~\cite{Chazal2016-bo}. Let $\Dgm$ be the space of multisets of $\mathcal{H}$ that are finite on each $R \in \mathcal{R}$; we call such multisets \emph{persistence diagrams}. Consider a q-tame persistence module $\mathbb{V}$. By~\cite[Theorem 2.8, Corr. 2.15]{Chazal2016-bo},  there exists a unique persistence diagram $\mathrm{dgm}(\mu_\mathbb{V})$, such that the measure evaluated on rectangles $R \in \mathcal{R}$ is given by the cardinality of points in $\mathrm{dgm}(\mathbb{V}) \cap R$:
\begin{equation}
\mu_\mathbb{V}(R) = \#\{(b,d) \in \mathrm{dgm}(\mathbb{V}) \mid (b,d) \in R\}.
\end{equation}
We call points $(b,d)$ in $\mathrm{dgm}(\mathbb{V})$ a feature, and $b$ and $d$ the birth and death times of the feature in the filtration respectively. 
\begin{remark}
    The standard way to define persistence diagrams of pointwise finite dimensional persistence modules over $\R$ is via interval decomposition~\cite{CohenSteiner2007}. A priori, the persistent homology of sublevel set filtrations $H_i(X^\bullet)$ are not necessarily tame, and therefore their persistence module may not admit an interval decomposition. 
\end{remark}

Recall if we consider continuous functions on a finitely triangulable space, then the persistence module $H_i(X^\bullet)$ is q-tame~\cite[Theorem 2.22]{Chazal2016-bo}. As such, we can consider persistent homology as a function
\[
\pershomf_i : C(X, \mathbb{R}) \;\longrightarrow\; \Dgm
\]
which assigns to each $f \in C(X, \mathbb{R})$ the diagram $\pershomf_i(f):=\mathrm{dgm}(H_i(X^\bullet))$, where $X^\bullet$ is the sublevel set filtration of $X$ by $f$. Following~\cite{CohenSteiner2007}, we say the a map $f$ \emph{tame} if the persistence module $H_i(X^\bullet)$ is pointwise finite dimensional, and the persistence diagram $\pershomf_i(f)$ is finite for all $i$. The space of persistent diagrams $\Dgm$ admits a metric, the \emph{bottleneck distance} (see section 4.2 of ~\cite{Chazal2016-bo} for a full construction). The map $\pershomf_i$ is in fact 1-Lipschitz with respect to the bottleneck distance.
\begin{theorem} \label{thm:stability_ph}
(Stability) \cite[Cor 4.17]{Chazal2016-bo} Let $X$ be a triangulable space with continuous functions $f, g: X \rightarrow \mathbb{R}$. Then the persistence diagrams of the sublevel set filtrations satisfy $d_B(\pershomf_i(f),\pershomf_i(g)) \leq \| f - g\|_{\infty}$ for any integer $i$.
\end{theorem}
This property ensures that small perturbations of the input function lead to only small changes in the corresponding persistence diagram, making persistent homology a stable tool for data analysis.

\begin{remark}
    Consider the persistence module $H_i(X^\bullet)$ derived from a filtration of a space $X$ by a continuous function $f$. For any pair $X^{s}, X^t$ in the sequence where $s < t$, we can consider the long exact sequence for relative homology
$$
\cdots 
\rightarrow H_{i}(X^{s})
\rightarrow H_{i}(X^{t})
\rightarrow H_{i}(X^{t}, X^{s})
\rightarrow H_{i-1}(X^{t})
\rightarrow \cdots.
$$
Exactness of this sequence implies
$H_{i}(X^{s}) \to H_i(X^{t})$
is an isomorphism for all $i$
if and only if
$H_i(X^{t}, X^{s}) = 0$
for all $i$. We say a parameter $t$ is a \emph{homological regular value} of the filtration, if for sufficiently small $\epsilon > 0$ we have $H_i(X^{t-\epsilon}, X^{t+\epsilon}) = 0$ for all $i$; else it is a \emph{homological critical value}. Given the characterisation of persistence diagrams in~\cref{eq:pers_measure_b}, only homological critical values can be a birth or death of a feature in $\mathrm{dgm}(H_i(X^\bullet))$.
\end{remark}

\subsubsection{Smooth Morse Theory for Manifold with Boundaries}

For a smooth function $f \in C^k(M, \R)$ on a manifold $M$, let us denote $M_a = \sublevel{f}{a}$ as the sublevel set of $f$ under $a$. We recall a critical point $p$ of $f$ is where $\dd{f}(p) = 0$. A critical point is non-degenerate, if for any local chart $\varphi: U \to \R^m$ of $p$, the Hessian matrix $(H_f)_{ij} = \pdv[2]{(f \circ \inv{\varphi})}{x_i}{x_j}$ has non-zero determinant. The Morse index $\lambda(p)$ of a non-degenerate critical point is the number of negative eigenvalues of $H_f$. A function is called a Morse function of all its critical points are non-degenerate. 

Morse theory allows us to analyse the changes in homotopy type when we include a sublevel set $M_a \hookrightarrow M_{a +t}$ into another.  We refer the reader to many excellent texts in literature for a treatment of Morse theory, such as~\cite{Milnor1963-mi,Banyaga2004-my,Nicolaescu2011-oh}.  The key insight from Morse theory is that the homotopy type of the sublevel set $M_{a +t}$ only changes as $t$ passes through critical values of $f$. In particular, if $f$ is a Morse function, then the change in homotopy type is governed by the Morse indices of the critical points. If $[a_-,a]$ contains only one critical value $b$, lying in $(0,a]$, then $M_{a_-} $ is homotopy equivalent to $M_a$ with $\lambda(p)$-dimensional cells $e^{\lambda(p)}$ attached along the boundary, for every critical point $p$ in  $\fibre{f}{b}$ (\cite[Theorem 2.2.3, Remark 2.2.5]{Nicolaescu2011-oh},~\cite[\S 3.2]{Milnor1963-mi}):
\begin{equation}
    M_{a} \simeq M_{a_-} \cup e^{\lambda(p_1)} \cup \cdots \cup e^{\lambda(p_n)}.
\end{equation}
Consequently, the relative homology of the pair $(M_a, M_{a-})$ is given by the following (\cite[\S 5]{Milnor1963-mi},~\cite[Theorem A.4.1-2]{Mazzucchelli2011-tr},~\cite[Theorem 2.2.3]{Nicolaescu2011-oh}): 
\begin{equation}
    H_\bullet(M_a,M_{a-}) = \bigoplus_{i=1}^n \tilde{H}_\bullet(S^{\lambda(p_i)}).
\end{equation}

In our application of smooth Morse theory, we are interested in the case where the manifold has non-empty boundary. This case is treated in~\cite{Turner2024Extended}. As our main object of study is a smooth function on the M\"obius band, that is positive on the interior and zero on the boundary, we restrict ourselves to the following class of functions. 

\begin{definition}\label{def:finite_morse}
    For $k \geq 2$, consider  $f \in C^k(M, \R)$ a non-negative function on a compact manifold with boundary $(M, \partial M)$ such that $\inv{f}(0) = \partial M$. We say $f$ is \emph{Morse-supported}, if the following conditions are satisfied:
    \begin{enumerate}[label=(\emph{M\arabic*})]
        \item \label{M1} Restricted to the interior $\interior{M}$ of $M$, the map $f$ is a Morse function  (its critical points are non-degenerate);
        \item \label{M2} $f\rvert_{\interior{M}}$ has only finitely many critical points; and
        \item \label{M3} There is some sufficiently small $a > 0$, such that for all $t \in [0,a]$, the boundary $\partial \Mob = \fibre{f}{0}$ is a deformation retract of the sublevel set $\inv{f}[0,a]$. 
    \end{enumerate}
\end{definition}
If $0$ is a regular value of $f$, then the second condition~\labelcref{M3} follows from the standard deformation lemma of Morse theory of regular level sets (see~\cite[Theorem 3.20]{Banyaga2004-my} for statement in particular for manifolds with boundary). However, we state this extra condition as the boundary if the M\"obius band is a critical level set of the \functionname $\DT_\gamma$ with critical value 0. 

We adapt the observations of~\cite{Turner2024Extended} on persistent homology of Morse functions on manifolds with boundary for the case of Morse-supported functions in the following lemma.
\begin{lemma}\label{lem:morse_supported_tame}
    Let $f$ be a Morse-supported function with critical values  $0 \leq a_0 < a_1 < \cdots < a_N = \infty$. For each critical value $a_i$, let $C_i$ be the set of critical points in $\fibre{f}{a_i}$. Then:
    \begin{itemize}
        \item For $a_i  \leq s \leq t < a_{i+1}$, $H_\bullet(M_s) \xrightarrow{\cong} H_\bullet(M_t)$; and
        \item For $s \in (a_{i-1}, a_{i}]$, and $t \in (a_{i}, a_{i+1})$, 
        \begin{equation}
            H_n(M_t, M_s) \cong \bigoplus_{p \in C_i} \tilde{H}_n(S^{\lambda(p)}).
        \end{equation}
    \end{itemize}
    Consequently, the persistence module $H_n(M_\bullet)$ is pointwise finite dimensional for any $n$, and $a_i$ are precisely the homological critical values of the filtration. 
\end{lemma}
We now give a sufficient condition for $f \in C^2(M, \R)$ having finitely many non-degenerate critical points, a necessary condition for a function being Morse-supported. Recall if $M$ is compact and  without boundary, then a Morse function $f$ must have finitely many critical points. We can show that on a manifold with boundary, we can similarly control the cardinality of critical points on the interior of the manifold, if we can isolate them from the boundary. 
\begin{lemma} \label{lem:smooth_isolate_finitemorse}
    Consider $f \in C^2(M, \R)$ where $M$ is a compact manifold with boundary $\partial M$. If there is a neighbourhood $V$ of $\partial M$ on which there are no critical points of $f$ on $V \setminus \partial M$, then $f$ being Morse on the interior of $M$  implies there are finitely many critical points in the interior of $M$.
\end{lemma}
\begin{proof}
    The proof follows that for compact manifolds without boundary. Suppose there are infinitely many critical points of $f$ on the interior. Since $M$ is compact we can consider a convergent sequence $(p_n)$ of such critical points, which by continuity of $\dd{f}$, must converge to some critical point $p \in M$. Because $V$ isolates the boundary, this critical point must be on the interior. Since $f$ is Morse on the interior, $p$ is non-degenerate and hence isolated; however this contradicts there being an infinite sequence $p_n \to p$. Hence there cannot be infinitely many critical points on the interior.
\end{proof}

\subsubsection{Morse Theory for Filtered Simplicial Complexes}
\label{ssec:morse_sets}

Suppose $X$ is a finite simplicial complex, and we have a a filtration of subcomplexes $X^\bullet$ of $X$ over $\R$. If the filtration is a sublevel set filtration of a function $f: X \to \R$ which is monotone with respect to the face poset of the simplicial complex, we say $f$ is a \emph{filter function}.

Consider the sequence of inclusions 
\begin{equation}
    \emptyset = X^{a_0} \subset \cdots \subset X^{a_i}  \subset X^{a_n} = X^{a_{i+1}} \subset \cdots \subset X 
\end{equation}
where $a_0 < \cdots <a_n$ are the finitely many values of $\R$ where new simplices are added to the filtration.  In this case, each sublevel set $X^a$ is a subcomplex of $X$. In order to describe the homological critical values of the persistence module $H_i(X^\bullet)$, we wish to consider how individual additions of simplices in a filtration contribute to the relative homology $H_k(X^{a_i}, X^{a_{i-1}})$. We recall some observations from Conley-Morse theory that relates the relative homology to how the additional simplices are attached to an existing complex in the filtration. 

Consider an inclusion of finite simplicial complexes $X^- \hookrightarrow X$. Let $S = X \setminus X^-$
and denote $\closure{S}$ to be the closure of $S$ (the smallest subcomplex in $X$ containing $S$). We call  $L = \closure{S} \cap X^-$  the \emph{link} of $S$; that is, the faces of simplices in $S$ that are in $X^-$ rather than $S$. Borrowing terminology from combinatorial dynamics, we call connected components $S_i$ of $S = \bigsqcup_i S_i$ the \emph{Morse sets} of the filtration, and call $L_i = \closure{S_i} \cap X^-$ the links of the Morse set $S_i$. It follows from the definitions that $L_i$ are closed in $X$ and $L = \bigcup_i L_i$.  We also have the following characterisation.

\begin{lemma}
    Let $S_i$ be a connected component of  $X \setminus X^-$ where $X^-$ is a subcomplex of $X$ a finite simplicial complex. Then $\closure{S_i} =  S_i \sqcup L_i$, for $L_i = \closure{S_i} \cap X^-$. 
\end{lemma}
\begin{proof}
    Suppose there is some simplex $\sigma \in \closure{S_i} \setminus S_i$ that is not in $L_i$. Since $L_i = \closure{S_i} \cap X^-$ and $X$ is a complex, $\sigma$ must be in set $S = X \setminus X^-$. Furthermore since $\sigma$ is in the closure of $S_i$, $\sigma$ is connected to $S_i$. However, since $S_i$ is a connected component, $\sigma$ must be in $S_i$ as well, which is a contradiction of $\sigma \in \closure{S_i} \setminus S_i$ but not in $L_i$. Thus, $\closure{S_i} \setminus S_i \subseteq L_i$. On the other hand, since $S_i \cap X^- = \emptyset$, and $L_i \subseteq X^-$, we have $L_i \subseteq \closure{S_i} \setminus S_i$. Therefore $L_i = \closure{S_i} \setminus  S_i$. 
\end{proof}
We refer to the relative homology group 
\begin{equation}
    \Con_\bullet(S_i) := H_\bullet(\closure{S_i}, L_i)
\end{equation}
as the \emph{Conley index} of the Morse set $S_i$. 
If the Conley index of $S_i$ is trivial, then we also say $S_i$ is \emph{regular}. Otherwise we say $S_i$ is a \emph{critical set}.  
In particular, if $H_\bullet(\closure{S_i}, L_i) \cong \tilde{H}_\bullet(\mathbb{S}^n)$, then we call the Morse set $S_i$ an \emph{$n$-saddle} in the spirit of smooth Morse theory. We let $\lambda(S_i) = n$ denote the Morse index of the saddle. We note that if $\closure{S_i}$ is contractible, then the Conley index of $S_i$ is characterised by the homology of its link: by the reduced homology long exact sequence of the pair $(\closure{S_i}, L_i)$ we have an isomorphism for homology degree $k \geq 1$:
\begin{equation}
{H}_k(\closure{S_i}, L_i)  \cong \tilde{H}_k(\closure{S_i}, L_i) \xrightarrow{\cong} \tilde{H}_{k-1}(L_i). 
\end{equation}
For $k = 0$, if $L_i$ is empty, then $H_0(\closure{S_i}, L_i) = \mathbb{F}$; else  $H_0(\closure{S_i}, L_i) = 0$. 

Using excision-type arguments, we relate the relative homology $H_\bullet(X,X^-)$ to the  Conley indices. 

\begin{lemma} \label{lem:conley_relative_homology} Let $(X,X^-)$ be a pair of finite simplicial complexes. Then the inclusion of pairs $(\closure{S}, L) \hookrightarrow (X,X^-)$ induces an isomorphism 
    \begin{equation}
    \bigoplus_i \Con_\bullet(S_i) \xrightarrow{\cong} H_\bullet(X,X^-).
\end{equation}
where $S_i$ are connected components of $X \setminus X^-$. 
\end{lemma}
\begin{proof}
    Consider relative simplicial chain complexes $C_\bullet(X,X^-)$ and $C_\bullet(\closure{S}, L)$. We can check by direct computation that the inclusion of pairs $(\closure{S}, L) \hookrightarrow (X, X^-)$ induces an isomorphism of chain complexes $C_\bullet(\closure{S}, L) \xrightarrow{\cong}C_\bullet(X,X^-)$.  Furthermore, we can also check by direct computation that $C_\bullet(\closure{S}, L)$ is isomorphic to a direct sum of relative chain complexes associated to the Morse sets and their links $\bigoplus_i C_\bullet(\closure{S_i}, L_i)$. To show this, we first note that for $i\neq j$, because $S_i \cap x_j = \emptyset$, we have
    \begin{equation*}
        \closure{S_i} \cap \closure{x_j} = (\closure{S_i} \setminus S_i) \cap (\closure{x_j} \setminus x_j) = L_i \cap L_j.
    \end{equation*}
    Hence if we consider the relative Mayer-Vietoris short exact sequence of chain complexes
    \begin{equation*}
0 \to C_\bullet(\closure{S_i} \cap \closure{x_j}, L_i \cap L_j) \to C_\bullet(\closure{S_i} , L_i ) \oplus C_\bullet(\closure{x_j}, L_j)  \to C_\bullet(\closure{S_i} \cup \closure{x_j}, L_i \cup L_j)  \to 0, 
    \end{equation*}
    The fact that $C_\bullet(\closure{S_i} \cap \closure{x_j}, L_i \cap L_j)$ is trivial implies we have an isomorphism of chain complexes 
    \begin{equation*}
C_\bullet(\closure{S_i} , L_i ) \oplus C_\bullet(\closure{x_j}, L_j)  \xrightarrow{\cong} C_\bullet(\closure{S_i} \cup \closure{x_j}, L_i \cup L_j).
    \end{equation*}
     Inductively we thus have an isomorphism of chain complexes $ \bigoplus_i C_\bullet(\closure{S_i}, L_i) \to C_\bullet(\closure{S}, L)$, inducing an isomorphism of homology groups.
\end{proof}
As such, if all critical Morse sets of a filtration are saddles, then the rank of $H_n(X,X^-)$ is given by the number of $n$-saddles. Via the relative homology long exact sequence, we can thus relate the contribution of each Morse set to a homological change $H_\bullet(X^-) \to H_\bullet(X)$ in one degree, of rank precisely one. We distinguish this filtration with the following nomenclature.

\begin{remark}
   In the framework of Conley-Morse theory for combinatorial dynamics, the partition of the complex $X$ into $S, X^-$, and $S$ into a disjoint union $S_i$, is a \emph{combinatorial multivector field}; that is, a partition of the complex into locally closed subsets. The individual partition elements $S_i$ of $X^-$ are called \emph{multi-vectors}, and the link is also called a mouth. Our definition of the Conley index coincides with theirs; we refer the reader to ~\cite[\S 2]{mrozek2025dynamics} for an extensive account of Conley theory in the combinatorial setting. 
\end{remark}

\subsection{Differential Topology}
\label{ssec:differential_topology}
\paragraph{Function Space Topologies}
We now define the weak and strong (Whitney) topology on $C^k(M,N)$, following~\cite[\S 2.1]{hirsch2012differential}. We denote by $C^k_W(M,N)$ and $C^k_S(M,N)$ the space of functions endowed with weak and strong topology respectively. 

We use introduce a (sub)base to define these topologies. Let $M,N$ be $C^k$ manifolds. Consider local charts $\varphi: U \to \R^m$ and $\psi: V \to \R^n$ of $U \subset M$ and $V \subset N$, and $K \subset U$ compact. If $f \in C^k(M,N)$ is such that $f(K) \subset V$, consider maps $g \in C^k(M,N)$, such that $g(K) \subset V $, and for $\epsilon > 0$
\begin{align} \label{eq:jet_nbhd_property}
    \norm{D^r(\psi f \inv{\varphi})\rvert_x - D^r(\psi g \inv{\varphi})\rvert_x} & \leq \epsilon & \forall x \in K, \quad r = 0,\ldots, k.
\end{align} 
Then the \emph{weak subbasic neighbourhood} of $f$~\cite[P. 35, \S 2.1]{hirsch2012differential} is defined as
\begin{equation} \label{eq:weak_subbasic_nbhd}
    \cN(f; (\varphi,U), (\psi, V), K, \epsilon):= \{g \in C^k(M,N) \ : \ g(K) \subset V,\ \text{ and $g$ satisfies \cref{eq:jet_nbhd_property}} \}.
\end{equation}
 The \emph{weak} (or $C^k$ compact-open) topology on $C^k(M,N)$ has the collection $\{\cN(f; (\varphi,U), (\psi, V), K, \epsilon) \}$ of subsets of maps (varying over $U,V$ open, $K$ compact, $\epsilon > 0$) as a subbase. In other words, any $f$ has a neighbourhood consisting of finitely many intersections of weak subbasic neighbourhood. A sequence of maps $(f_i)$ in $C^k(M,N)$ converges to another $f$ in $C^k(M,N)$ in weak topology, if for any $W$ that is a finite intersection of weak subbasic neighbourhoods, there is some $j$ such that $f_i \in W$ for $i \geq j$. In other words, convergence in weak topology is equivalent to requiring for any choice of $K$ compact, and compatible charts $\varphi, \psi$, the \emph{uniform convergence} of $D^r(\psi f_i\varphi^{-1})$ to $D^r(\psi f\varphi^{-1})$ on the compact subset $\varphi(K) \subset \R^m$, for $r = 0,\ldots,k$.

On the other hand, the \emph{Whitney} or \emph{strong} topology on $C^k(M,N)$ is given by a base consisting of subsets of the form $\bigcap_{i \in \cI} \cN(f; (\varphi_i,U_i), (\psi, V_i), K_i, \epsilon_i)$, where $\{U_i\}_{i \in \cI}$ and $\{V_i\}_{i \in \cI}$ are locally finite covers of $M$ and $N$. Since the index set is only locally finite, the strong topology is finer than the weak topology. Convergence in strong topology is a more stringent condition, requiring maps to coincide everywhere but a compact subset $K$, and converge uniformly on such $K$ (see~\cite[41.7 \& 41.10]{kriegl1997convenient}).  However, if $M$ is compact, then the weak and strong topologies coincide~\cite[\S 2.1]{hirsch2012differential}. 

\begin{remark}
    If the target manifold $N$ is some Euclidean space $\R^n$, then we let $ \cN(f; (\varphi,U), K, \epsilon)$ denote the weak subbasic neighbourhood where $\psi$ is the identity on $\R^n$ and the neighbourhood $V$ consists of maps $g$ s.t. $\norm{f-g} < \epsilon$ on all $K$.
\end{remark}

\begin{remark}
    We recall some salient further facts about topological properties of subsets of $C^r_S(M,N)$.  We recall a \emph{residual} subset of a topological space is a countable intersection of open dense subsets.
    \begin{enumerate}
        \item $C^r_S(M,N)$ is a Baire space: any residual subset is dense (\cite[Theorem 2.4.4]{hirsch2012differential}). 
        \item For $r \geq 1$, the set of $C^r$ embeddings $\Emb{r}{M}{N}$ is open in $C^r_S(M,N)$ (\cite[Theorem 2.1.4]{hirsch2012differential}). 
        \item If $U$ is an open set of manifold $M$, the restriction $f \mapsto f\rvert_U$ as a function from $C^r(M,N)$ to $C^r(U,N)$ is continuous for the weak topology, but not necessarily for the strong topology (\cite[Problem 2.1.16]{hirsch2012differential}).
    \end{enumerate}
    
\end{remark}

\paragraph{Jets}
We can also describe the weak and strong topologies on $C^r(M,N)$ in terms of \emph{jets}, which is a convenient framework for manipulating  derivatives of smooth maps between manifolds. We follow the exposition of~\cite[\S2]{golubitsky2012stable} and~\cite[\S4]{hirsch2012differential} and refer the reader to them for a full formal exposition. Let $M$ and $N$ be $C^k$ manifolds; for $x \in M$ and $y \in N$, consider the set of maps 
\begin{equation}
    \FSpace{k}{M}{N}_{x,y} = \Bqty{f \in \FSpace{k}{M}{N} \ : \ f(x) = y}.
\end{equation}
For $r \leq k$ a positive integer, and (any and all) local charts $\varphi: (U,x) \to (\R^m,0)$ and $\psi: (V,y) \to (\R^n,0)$ at $x,y$ respectively, consider the equivalence relation on $\FSpace{k}{M}{N}_{x,y}$
\begin{equation}
    f \sim_{x,r} g \qq{if} D^j{(\psi f \inv{\varphi})}\rvert_0 = D^j{(\psi g \inv{\varphi})}\rvert_0 \quad \forall \ j \in 0,\ldots, r.
\end{equation}
Note that the equivalence relation does not depend on the choice of charts. We define the set of $r$-jets $\Jet{r}{}{M}{N}_{x,y}$ from $x$ to $y$, to be $\FSpace{k}{M}{N}_{x,y} / \sim_r$, and the set of $r$-jets from $M$ to $N$ to be
\begin{equation}
    \Jet{r}{}{M}{N} = \bigsqcup_{x \in M,\ y \in  N} \Jet{r}{}{M}{N}_{x,y}. 
\end{equation}
We can write any element of \Jet{r}{}{M}{N} as $(x,y, [f]_{x,r})$, where $x$ and  $y$ are respectively the sources and targets of $f$, and $[f]_{x,r}$ is the equivalence class of $f$ under $\sim_{x,r}$. For any $f \in \FSpace{k}{M}{N}$, the $r$-\emph{prolongation} of $f$ is the map $\prolong{r}{}f \in \FSpace{k-r}{M}{\Jet{r}{}{M}{N}}$, where $\prolong{r}{}f(x) = (x,f(x), [f]_{x,r})$. We equip \Jet{r}{}{M}{N} with the topology where open sets are  $\Jet{r}{}{U}{V} = \bigsqcup_{x \in U, y \in V} \Jet{r}{}{M}{N}_{x,y}$ for $U \subseteq M$, $V \subseteq N$ open.  

We now explicitly derive a canonical atlas of \Jet{r}{}{M}{N} given atlases  
$\Bqty{\varphi_a:  U_a \to \R^m}$ and $\Bqty{\psi_b:  V_b \to \R^n}$ of $M$ and $N$ respectively. For $(x,y,[f]) \in \Jet{r}{}{U_a}{V_b}$, we can specify $[f]$ as the truncated Taylor polynomial expansion of $f_{ab} :=\psi_b f \inv{\varphi_a}: \R^m \to \R^n$ at $\varphi(x)$ to order $r$. Let $A^j_m$ be the $\R$-vector space of real polynomials in $m$-variables of degree $\leq j$ with constant term = 0, and define $B^j_{m,n} = \bigoplus_{i=1}^n A^j_m$. We have a bijection $\zeta_{ab}: \Jet{r}{}{U_a}{V_b} \to \varphi_a(U_a) \times \psi_b(V_b) \times B^j_{m,n}$, where 
\begin{equation}
    (x,y,[f]_{x,r}) \mapsto (\varphi_a(x), \psi_b(y), Df_{ab}\rvert_{\varphi(x)}, \ldots, D^rf_{ab}\rvert_{\varphi(x)}).
\end{equation}
In fact, with the topology we have chosen for \Jet{r}{}{M}{N}, $\zeta_{ab}$ is a homeomorphism between \Jet{r}{}{U_a}{V_b} and a Euclidean space. The collection of maps $\Bqty{\zeta_{ab}, \Jet{r}{}{U_a}{V_b})}$ satisfy the axioms for an atlas, as such this imposes a manifold structure for $\Jet{r}{}{M}{N}$ (we refer the reader to~\cite[\S2]{golubitsky2012stable} and~\cite[\S4]{hirsch2012differential} for details). 

Since jets carry information about higher order derivatives, they can also be used to define the weak and strong $C^k$ topologies. The prolongation of $f \in \FSpace{k}{M}{N}$ as an assignment from $\FSpace{k}{M}{N}$ to the set of continuous maps between $M$ and $\Jet{k}{}{M}{N}$. 
\begin{align}
    j^k:  \FSpace{k}{M}{N} &\to  \FSpace{0}{M}{\Jet{k}{}{M}{N}} \\
    f &\mapsto j^kf. \nonumber
\end{align}
The strong (weak) topology on $\FSpace{k}{M}{N}$ is precisely 
the coarsest topology such that prolongations $j^k$ are continuous, when $\FSpace{0}{M}{\Jet{k}{}{M}{N}} $ is endowed with the strong (weak) topology~\cite[41.9-10]{kriegl1997convenient}. 

\paragraph{Multi-jets} The set of $s$-fold $r$-jets of \FSpace{k}{M}{N}, is defined to be
\begin{equation}
    \Jet{r}{s}{M}{N}:= \{(\prolong{r}{}f(x_1) \ldots, \prolong{r}{}f(x_s))\ : \ f \in \FSpace{k}{M}{N}, (x_1, \ldots, x_s) \in \Conf{n}{M} \}.
\end{equation} 
The set of $s$-fold $r$-jets is a submanifold of $\Jet{r}{}{M}{N}^s$, where the charts are inherited from those of $\Jet{r}{}{M}{N}^s$ as a product manifold:
\[\Bqty{(\zeta_{a_1, b_1}, \ldots, \zeta_{a_s, b_s}), \Jet{r}{}{U_{a_1}}{V_{b_1}} \times \ldots \times \Jet{r}{}{U_{a_s}}{V_{b_s}}}.\]
We can similarly prolong $f \in \FSpace{k}{M}{N}$ to an $n$-fold $r$-jet $\prolong{r}{s} f: \Conf{s}{M} \to \Jet{r}{s}{M}{N}$:
\begin{equation}
    \prolong{r}{s} f : (x_1, \ldots, x_s) \mapsto (\prolong{r}{}f(x_1) \ldots, \prolong{r}{}f(x_s)).
\end{equation}

\paragraph{Transversality} We follow the exposition on transversality in~\cite[\S 2]{hirsch2012differential} and~\cite[\S 4]{golubitsky2012stable}. Let $M$ and $N$ be differentiable manifolds. A differentiable map $f$ intersects a submanifold $W$ of $N$ transversely on $V \subseteq M$, if either $f(V) \cap W = \emptyset$, or where $x \in V$ is mapped into $W$, 
\begin{equation}  \label{eq:transversality}
    T_{f(x)} W + \dd{f}_x (T_xM) = T_{f(x)} M.
\end{equation}
We denote this by $f \transverse_V W$. If $V = M$, then we simply write  $f \transverse W$. If $W$ is a Whitney stratified space, then we say $f \transverse_V W$ if $f \transverse_V W_i$ as manifolds for all strata $W_i$. 

Note that if $\rank \dd{f}$ is less than the codimension of $W$ in $N$ where $f(M) \cap W$, then the $f \transverse W$ implies $f(M) \cap W = \emptyset$. We exploit this fact in the proof of \Cref{prop:crit_ndg}, where the genericity of a transversality condition implies the genericity of a property defined by an empty intersection due to this dimensionality argument. 

The main theorem in transversality theory that we exploit in this paper is the following by Damon~\cite[Corollary 1.11]{damon1997generic}; we state a specific version of the theorem as presented in~\cite{arnal2023critical} (Theorem 9.1). 

\begin{theorem}[Damon] \label{thm:damon}
    Consider $k \geq r+1$ and $s \geq 2$. Let $M,N$ be $C^k$ manifolds where $M$ is compact, and $V \subset \Conf{s}{M}$ and $W \subset \Jet{r}{s}{M}{N}$ be closed Whitney stratified subsets. Then the following subset of $\FSpace{k}{M}{N}$ with Whitney $C^k$-topology is \emph{residual}:
    \begin{equation}
        \Gamma = \Bqty{f \in \FSpace{k}{M}{N} \ : \ \prolong{r}{s}f \transverse_V W}.
    \end{equation}
    Furthermore, if $V$ is compact, then $\Gamma$ is also open in Whitney $C^{k+1}$-topology.
\end{theorem}

Since $\Emb{k}{M}{N}$ is open in $C^k_S(M,N)$~\cite[5.3]{kriegl1997convenient}, if $\Gamma$ is residual in $C^k_S(M,N)$, then $\Gamma \cap \Emb{k}{M}{N}$ is also residual in $\Emb{k}{M}{N}$.

\subsection{Subsets of Positive Reach}
The reach of a closed subset $A \subset \R^d$ of Euclidean space was introduced by~\cite{federer1959curvature} to quantify the size of tubular neighbourhoods of $A$. Let  $A^\epsilon = \{x \in \R^d \ : \ d(x,A) \leq \epsilon\}$ denote the $\epsilon$-neighbourhood of $A$ in $\R^d$. A tubular neighbourhood of $A$ is an $\epsilon$-neighbourhood of $A$, such that for any point $x$ in the $A^\epsilon$ has a unique nearest neighbour on $A$. The \emph{reach} of $A$, denoted by $\tau$, is the supremum over all $\epsilon$'s such that $A^\epsilon$ is a tubular neighbourhood of $A$. If $\tau > 0$, we say $A$ has \emph{positive reach}. 

Subsets of positive reach include convex subsets (where $\tau = \infty$), and compact $C^2$-submanifolds~\cite[Remarks 4.20-21]{federer1959curvature}. We recall some results in literature where $\tau$ is shown to be a useful characteristic length scale of a closed subset $A$, below which the local geometry of $A$ can be well approximated by Euclidean data. First, we note a result from~\cite{boissonnat2019reach}, which describes how Euclidean distances bound the geodesic distance $d_A$ in $A$, as induced by the ambient Euclidean distance.
\begin{lemma}[{\cite[Lemma 3]{boissonnat2019reach}}] \label{lem:reach_distance} If $A \subset \R^d$ is a closed set, and $p,q \in A$ have Euclidean distance $\norm{p-q} < 2 \tau$, then 
\begin{equation} \label{eq:geodesic_reach}
    \norm{p-q} \leq d_A(p,q) \leq 2\tau \arcsin\frac{\norm{q-p}}{2\tau}.
\end{equation}
\end{lemma}
In Federer's introduction of the reach~\cite{federer1959curvature}, he shows that the reach also controls the rate at which a tangent plane approximation of a manifold at a point deteriorates as we increase the size of the neighbourhood. 
\begin{theorem}[{\cite[Theorem 4.18]{federer1959curvature}}] Let $\tau$ be the reach of a a submanifold $M \subset \R^d$. For $T_p$ the tangent plane of $M$ at $p$, the distance of $q \in M$ to $T_p$ is bounded above by 
\begin{equation}
    d(q, T_p) \leq \frac{\norm{q-p}^2}{2\tau}.
\end{equation}
\end{theorem}
Note that $ d(q, T_p) $ is the length of the orthogonal projection of the vector $q-p$ onto the plane $-p + T_p$. Since $d(q, T_p) \leq \norm{q-p}$, the bound is only meaningful when $\norm{q-p} \leq 2\tau$.

In the case where $M= \gamma(S^1)$ is the image of a $C^1$-embedding of $S^1$, we can obtain the following expression which is useful in controlling the critical points of $\DT_\gamma$.  Since $d(q, T_p) = \norm{p-q} |\sin \angle (q-p, \dot{\gamma}(p))|$, 
\begin{equation} \label{eq:reach_tangent}
   \left\langle {\dot{\gamma}(p)}, {\gamma(q) - \gamma(p)}\right \rangle^2 =  \norm{\dot{\gamma}(p)}^2{\|\gamma(q) - \gamma(p)\|}^2\cos^2 \angle (q-p, \dot{\gamma}(p)) \geq  1-\qty(\frac{\|\gamma(q) - \gamma(p)\|}{2\tau})^2. 
\end{equation}



\section{General Properties of the Persistent Homology of Chordal Distance Transforms}
\label{sec:elementary-props}

In this section, we give some general properties of \FN for general topological embeddings of circles. We show that $\DT: \Emb{}{S^1}{\R^d} \to C^k(\ExpS{2}{S^1}, \R)$ is a continuous transform of circle embeddings into functions that satisfy favourable symmetry properties. We also discuss how the choice of field coefficients for persistent homology can affect the persistence diagrams $\pershomf(\DT_\gamma)$. 

We first note that any property we prove about the persistence diagram $\pershomf \DT_\gamma)$ also holds for $\pershomf(\DTm_\gamma)$, as there is a bijection between the diagrams. 
\begin{lemma}\label{lem:square_bijection}
     The homeomorphism $\Phi:\R^2_{\geq 0} \to \R^2_{\geq 0}$
    given by
    $\Phi: (s, t) \mapsto (\frac{1}{2}s^2, \frac{1}{2}t^2)$
    induces a bijection between the persistence diagrams
    $\pershomf \circ \DT(\gamma)$ and 
    $\pershomf \circ \DTm(\gamma)$.
\end{lemma}
\begin{proof}
By definition,
the sublevel sets 
$
\{w \in \ExpS{2}{S^1} \ : \ \DTm_\gamma(w) \leq t\}
$
and
$
\{w \in  \ExpS{2}{S^1} \ : \  \DT_\gamma(w) \leq \frac{1}{2} t^2\}
$
are equal for all $t \geq 0$,
hence the corresponding persistence modules, and thus their diagrams, only differ by a monotonic
reparametrisation $s \mapsto s^2/2$ of $[0,\infty)$ .
\end{proof}

\subsection{Continuity}\label{ssec:continuity}
We first show that if $\gamma: S^1 \to X$ is continuous, then $\DT_\gamma$ and $\DTm_\gamma$ are also continuous functions on $\ExpS{2}{S^1}$. We respectively endow $C(S^1, X)$ and $C(\ExpS{2}{S^1}, \R)$ with the metrics
\begin{equation*}
    d(\gamma, \tilde{\gamma}) = \sup_{z \in S^1} d_X(\gamma(z), \tilde{\gamma}(z))\qc d(f, \tilde{f}) = \sup_{\{z_1, z_2\} \in \ExpS{2}{S^1}} |f(\{z_1, z_2\}) - \tilde{f}(\{z_1, z_2\}|.
\end{equation*}
These metrics metricise the weak (a.k.a compact-open) topologies of  $C(S^1, X)$ and $C(\ExpS{2}{S^1}, \R)$ respectively. 
 As a simple application of the triangle inequality, we can show that $\DTm$ is Lipschitz continuous.
\begin{lemma} \label{lem:continuity_of_distance_transforms} The distance transforms 
    $\DTm, \DT: C(S^1, X) \to C(\ExpS{2}{S^1}, \R)$ are continuous maps between function spaces. In particular, $\DTm$ is a 2-Lipschitz map.
\end{lemma}
\begin{proof}
Consider 
    \begin{equation*}
        d(\DTm\gamma, \DTm\tilde{\gamma}) = \sup_{\{z_1, z_2\} \in \ExpS{2}{S^1}} |d_X(\gamma(z_1),\gamma(z_2)) - d_X(\tilde{\gamma}(z_1),\tilde{\gamma}(z_2))|.
    \end{equation*}
    The term within the supremum can be bounded via the application of the triangle inequality twice, and obtain that $\DTm$ is 2-Lipschitz continuous:
    \begin{align*}
        d(\DTm\gamma, \DTm\tilde{\gamma}) &\leq  \sup_{\{z_1, z_2\} \in \ExpS{2}{S^1}} (d_X(\gamma(z_1),\tilde{\gamma}(z_1)) + d_X(\gamma(z_2),\tilde{\gamma}(z_2))) \\
        &= 2\sup_{z \in S^1} d_X(\gamma(z), \tilde{\gamma}(z))= 2d(\gamma, \tilde{\gamma}).
    \end{align*}
    Since $\DT$ is the composition of $\DTm$ with a continuous map it follows that $\DT$ is also continuous.
\end{proof}
Having defined a continuous function $\DT_\gamma : \ExpS{2}{S^1} \to \R$, we can construct a sublevel set filtration and consider its persistent homology. Since $\ExpS{2}{S^1}$ is a manifold with boundary and thus finitely triangularisable, continuous functions on $\ExpS{2}{S^1}$ always have $q$-tame persistent homology~\cite[Theorem 2.22]{Chazal2016-bo}. Thus, we can consider the persistence diagrams $\pershomf_i(\DT_\gamma)$ in the metric space of diagrams $\Dgm$, equipped with bottleneck distance. Because the map $\pershomf$ is a continuous transformation of functions to the space of diagrams  from~\Cref{lem:continuity_of_distance_transforms}
and by the stability theorem (\cref{thm:stability_ph}), the composition of $\DT: C(S^1, X) \to C(\ExpS{2}{S^1}, \R)$ with $\pershomf_i$ yields a continuous function $\pershomf_i \circ \DT$ from embeddings to persistence diagrams. 
\begin{corollary} \label{cor:cont_ph_cdt}
 $\pershomf \circ \DT: C(S^1,X) \to \Dgm$ continuous w.r.t. bottleneck distance on $\Dgm$. Furthermore, $\pershomf \circ \DTm: C(S^1,X) \to \Dgm$ is 2-Lipschitz.  
\end{corollary}

\subsection{Symmetry invariance} \label{ssec:symmetry}
As $\DT_\gamma$ is purely a function of distance between points in Euclidean space, it is an isometry invariant: any isometry $\omega: \R^d \to \R^d$ leaves the \functionname invariant.  If $\gamma$ is transformed by an isometry $\alpha: \R^d \to \R^d$ of Euclidean space, then the transformed map $\alpha \circ \gamma: S^1 \to \R^d$ has the same \functionname as $\DT_\gamma = \DT_{\alpha \circ \gamma}$ only depends on distances. Explicitly, 
\begin{equation}
    \DT_{\alpha \circ \gamma}(\{z_1,z_2\}) = \frac{1}{2} \norm{\alpha \circ \gamma(z_1) - \alpha \circ \gamma(z_2)}^2 =  \frac{1}{2} \norm{\gamma(z_1) - \gamma(z_2)}^2   = \DT_{\gamma}(\{z_1,z_2\}).
\end{equation}
Importantly for application, these observations mean that two different circle embeddings do not need to be centred nor aligned by rotation in the ambient Euclidean space in order to obtain geometric features that meaningfully compare their geometry.

However, the \functionname is nonetheless sensitive to reparametrisations of $\gamma$. By \emph{reparametrisation}, we mean a homeomorphism $\omega: S^1 \to S^1$. Under this transformation, we can re-parametrise the curve $\gamma \circ \omega: S^1 \to \R^d$, and obtain a different embedding $\gamma \circ \omega$ that has the same image, but potentially different chordal distance transform: $\DT_{\gamma \circ \omega}$ may not be the same as $\DT_\gamma$. Nonetheless, because sublevel sets of the functions $\DT_\gamma$ and $\DT_{\gamma \circ \omega}$ are \emph{homeomorphic}, the persistent homology of \functionname is invariant with respect to parametrisation. 
\begin{lemma}
    Let $\omega: S^1 \to S^1$ be a homeomorphism. Then $\omega$ induces an isomorphism between persistence modules of the sublevel sets filtrations of $\DT_{\gamma \circ \omega}$ and $\DT_{\gamma}$. Consequently, $\pershomf_\bullet(\DT_{\gamma \circ\omega}) = \pershomf_\bullet(\DT_\gamma)$. 
\end{lemma}
\begin{proof}
    By the universal property of quotient maps, the homeomorphism which $\omega$ induces on the product space  $\omega \times \omega: S^1 \times S^1 \to S^1 \times S^1$ also induces a homeomorphism $\tilde{\omega}: \ExpS{2}{S^1} \to \ExpS{2}{S^1}$, given by 
\begin{equation}
    \tilde{\omega}(\{z_1, z_2\} ) = \{ \omega(z_1), \omega(z_2)\}.
\end{equation}
We can then write $\DT_{\gamma \circ \omega} = \DT_\gamma \circ \tilde{\omega}$. This implies for any point $\{z_1, z_2\}$ where $\DT_{\gamma \circ \omega} \leq a$, there is another point $\tilde{\omega}(\{z_1, z_2\})$ where $\DT_\gamma$ is less than $a$, and vice versa. In other words, $\tilde{\omega} : \ExpS{2}{S^1} \to \ExpS{2}{S^1}$ restricts to a homeomorphism between sublevel sets of $\DT_{\gamma \circ \omega}$ and $\DT_{\gamma}$ at the same level. Because $\tilde{\omega}$ is a homeomorphism on the whole of $\ExpS{2}{S^1}$, by restriction we have commutative diagrams where the vertical maps are homomorphisms:
\begin{equation}
    \begin{tikzcd}
    (\DT_{\gamma \circ \omega})^{-1}[0,a] \arrow[r, hook] \arrow[d, "\omega"] & (\DT_{\gamma \circ \omega})^{-1}[0,b] \arrow[d, "\omega"] \\
    \DT_{\gamma}^{-1}[0, a] \arrow[r, hook] & \DT_{\gamma}^{-1}[0, b]
\end{tikzcd}.
\end{equation}
Applying the singular homology functor to this diagram, we obtain an isomorphism between the resulting persistence modules of the sublevel sets filtrations of $\DT_{\gamma \circ \omega}$ and $\DT_{\gamma}$. Hence the map $\pershomf \circ \DT: C^k(S^1,\mathbb{R}^d) \to \Dgm$
is invariant to homeomorphic reparametrisations of $S^1$.
\end{proof}

We summarise these invariant properties of $\pershomf \circ \DT$ with the following proposition.
\begin{proposition} \label{prop:sym}
    Let $\DT,\DTm: \Emb{}{S^1}{\R^d} \to C(\ExpS{2}{S^1}, \R)$ be the chordal distance transforms on topological embeddings of circles in Euclidean space. If $\alpha: S^1 \to S^1$ is a homeomorphism, and $\omega: \R^d \to \R^d$ is a Euclidean isometry, then the feature maps $\pershomf \circ \DT$ and $\pershomf \circ \DTm$ are invariant with respect to such transformations:
    \begin{equation}
    \pershomf (\DT_{\omega \circ \gamma\circ \alpha}) = \pershomf (\DT_\gamma)\qc \pershomf (\DTm_{\omega \circ \gamma\circ \alpha}) = \pershomf (\DTm_\gamma). 
\end{equation}
\end{proposition}

\subsection{Dependency of Persistent Homology on coefficient field.}\label{ssec:homology_field}

As the M\"obius band is homotopic to a circle,
it has relative homology groups $\tilde{H}_i(\Mob; \Z_p)$ as $\Z_p$
if $i=1$ and $0$ otherwise, regardless of $p$.
However, as the boundary $\partial \Mob$ wraps twice
around compared to the centre line of the M\"obius band,
the persistent homology of the M\"obius band is sensitive
to $p$, as the next proposition shows.

\begin{proposition}
\label{prop:maxmin}
    Let $g: \ExpS{2}{S^1} \to [0,\infty)$ be a Morse-supported function. 
    Consider the persistence diagram 
    $\pershomf_1{(g; \Z_p)}$,
    where homology is taken over the field
    $\Z_p$ for some prime $p$.

    There are two specific points
    $(0, d_1), (l, d_2) \in \pershomf_1{(g; \Z_p)}$
    where $l \leq \min_s \max_t g(\{s, t\})$.
    if $p=2$, then $d_1 = \max{g}$ and $d_2 = \infty$.
    Otherwise, $d_1 = \infty$ and $d_2 = \max{g}$.

    \begin{center}
    \begin{tikzpicture}[scale=1, xscale=5]
    \draw     (0, 0) -- (2, 0);
    \draw (0.8, -0.25) -- (2.5, -0.25);
    \draw (0, -1) -- (2.5, -1);
    \draw     (0.8, -1.25) -- (2, -1.25);

    \draw[dotted] (0, -1.5) -- (0, 0.5);
    \draw[dotted] (1, -1.5) -- (1, 0.5);
    \draw[dotted] (2, -1.5) -- (2, 0.5);

    \node at (-0.5, 0) {$p = 2$};
    \node at (-0.5, -1) {$p \neq 2$};
    \node at (-0.5, -1.75) {$t$};

    \node at (0, -1.75) {$0$};
    \node at (1, -1.75) {$\min_s \max_t g(\{s,t\})$};
    \node at (2, -1.75) {$\max g(\{s,t\})$};
    \node at (2.5, -1.75) {$\infty$};
    \end{tikzpicture}
    \end{center}
\end{proposition}
\begin{proof}

In the following, we will write $\Mob = \ExpS{2}{S^1}$.
Suppose for now that $g$ has a unique global maximum.
Let $t_*$ be a value just before this global maximum
so that $\Mob_{t_*}$ is a punctured $\Mob$.
Consider the homology classes $[a], [b], [c]$ of $H_1(M_{t_*})$ as in the following diagram:
$$
\begin{tikzpicture}[scale=3, >=stealth]
\draw[thick] (0,0) rectangle (1,1);

\draw [thick] (0.4, 0.05) -- (0.3, 0) -- (0.4, -0.05);
\draw [thick] (0.6, 1.05) -- (0.7, 1.0) -- (0.6, 0.95);

\draw[red, very thick] (0, 0) -- (0, 1);
\node at (-0.1, 0.5) {$a$};
\draw[red, very thick] (1, 0) -- (1, 1);
\node at (1.1, 0.5) {$a$};

\draw[blue, very thick] (0.5, 0) -- (0.5, 1);
\node at (0.55, 0.25) {$b$};

\filldraw [color=green, fill=green!0, very thick] (0.25, 0.75) circle (0.1);
\node at (0.25, 0.75) {$\circ$};
\node at (0.25, 0.6) {$c$};
\end{tikzpicture}
$$
This gives a submodule of the persistence module restricted to $0, t_*$ and $\max{g}$
as given by the following commutative diagram:
$$
\begin{tikzcd}[ampersand replacement=\&, column sep = 7em]
    \Z{[a]} \arrow[r, "{\begin{bmatrix} 2 & 1 \end{bmatrix}}^T"]  \arrow[d, hook]
    \& 
    {\Z[b] \oplus \Z[c]} \arrow[r, "{\begin{bmatrix} 1 & 0 \end{bmatrix}}"]  \arrow[d, hook]
    \& 
    {\Z[b].} \arrow[d, hook]
    \\
    H_1(\Mob_0) \arrow[r, "H_1(\Mob_0 \subset \Mob_{t_*})"] \& H_1(\Mob_{t_*}) \arrow[r, "H_1(\Mob_{t_*} \subset \Mob_{\max{g}})"] \& H_1(\Mob_{\max{g}}). 
\end{tikzcd}
$$
By the Universal Coefficient Theorem
over $\Z_p$ we have for $p \neq 2$
$$
\begin{tikzcd}[ampersand replacement=\&, column sep = 7em]
    \Z_p \arrow[r, "{\begin{bmatrix} 2 & 1 \end{bmatrix}}^T", swap] 
    \arrow[rr, "2", bend left=20]
    \& 
    \Z_p^2 \arrow[r, "{\begin{bmatrix} 1 & 0 \end{bmatrix}}", swap]
    \& 
    \Z_p
\end{tikzcd}
$$
and for $p = 2$ we have
$$
\begin{tikzcd}[ampersand replacement=\&, column sep = 7em]
    \Z_2 \arrow[r, "{\begin{bmatrix} 0 & 1 \end{bmatrix}}^T", swap] 
    \arrow[rr, "0", bend left=20]
    \& 
    \Z_2^2 \arrow[r, "{\begin{bmatrix} 1 & 0 \end{bmatrix}}", swap]
    \& 
    \Z_2.
\end{tikzcd}
$$
This gives us the rank function restricted to $0, t_*, \max{g}$:
$$
\begin{tikzpicture}
    \begin{scope}
    \node at (-1, 4) {$p\neq 2$};
    \filldraw[gray!30] (0, 3) -- (0, 0) -- (3, 3);
    \filldraw[gray!30] (0.5, 2.5) -- (0.5, 0.5) -- (2.5, 2.5);
    \filldraw[gray!60] (0.5, 1.5) -- (0.5, 0.5) -- (1.5, 1.5);
    
    \draw[->] (-1,-1) -- (3, -1);
    \draw[->] (-1,-1) -- (-1, 3);
    \draw[gray] (-1, -1) -- (3,3);
    \draw[dotted] (0, -1) -- (0, 3);
    \draw[dotted] (-1, 0) -- (3, 0);
    \draw[dotted] (1, -1) -- (1, 3);
    \draw[dotted] (-1, 1) -- (3, 1);
    \draw[dotted] (2, -1) -- (2, 3);
    \draw[dotted] (-1, 2) -- (3, 2);
    \draw[dotted] (3, -1) -- (3, 3);
    \draw[dotted] (-1, 3) -- (3, 3);

    \node at (0, -1.4) {$0$};
    \node at (1, -1.4) {$t_*$};
    \node at (2, -1.4) {$\max{g}$};
    \node at (3, -1.4) {$\infty$};
    \node at (-1.7, 0) {$0$};
    \node at (-1.7, 1) {$t_*$};
    \node at (-1.7, 2) {$\max{g}$};
    \node at (-1.7, 3) {$\infty$};

    \node[red] at (0, 0) {$1$};
    \node[red] at (1, 1) {$2$};
    \node[red] at (2, 2) {$1$};

    \node[red] at (0, 1) {$1$};
    \node[red] at (0, 2) {$1$};
    \node[red] at (1, 2) {$1$};

    \node at (0, 3) {$\times$};
    \node at (0.5, 1.5) {$\times$};
    \end{scope}
    \begin{scope}[shift={(7, 0)}]
    \node at (-1, 4) {$p = 2$};
    \filldraw[gray!30] (0, 1.5) -- (0, 0) -- (1.5, 1.5);
    \filldraw[gray!30] (0.5, 3) -- (0.5, 0.5) -- (3, 3);
    \filldraw[gray!60] (0.5, 1.5) -- (0.5, 0.5) -- (1.5, 1.5);
    
    \draw[->] (-1,-1) -- (3, -1);
    \draw[->] (-1,-1) -- (-1, 3);
    \draw[gray] (-1, -1) -- (3,3);
    \draw[dotted] (0, -1) -- (0, 3);
    \draw[dotted] (-1, 0) -- (3, 0);
    \draw[dotted] (1, -1) -- (1, 3);
    \draw[dotted] (-1, 1) -- (3, 1);
    \draw[dotted] (2, -1) -- (2, 3);
    \draw[dotted] (-1, 2) -- (3, 2);
    \draw[dotted] (3, -1) -- (3, 3);
    \draw[dotted] (-1, 3) -- (3, 3);

    \node at (0, -1.4) {$0$};
    \node at (1, -1.4) {$t_*$};
    \node at (2, -1.4) {$\max{g}$};
    \node at (3, -1.4) {$\infty$};
    \node at (-1.7, 0) {$0$};
    \node at (-1.7, 1) {$t_*$};
    \node at (-1.7, 2) {$\max{g}$};
    \node at (-1.7, 3) {$\infty$};

    \node[red] at (0, 0) {$1$};
    \node[red] at (1, 1) {$2$};
    \node[red] at (2, 2) {$1$};

    \node[red] at (0, 1) {$1$};
    \node[red] at (0, 2) {$0$};
    \node[red] at (1, 2) {$1$};

    \node at (0, 1.5) {$\times$};
    \node at (0.5, 3) {$\times$};
    \end{scope}
\end{tikzpicture}
$$
In both cases,
this implies the existence of two
points $(0, d_1)$ and $(l, d_2)$
for some $l \leq t_*$.
If $p = 2$, then $d_1 = \max{g}$
and $d_2 = \infty$
and vice-versa for $p = 2$.
Now we only need to show that $l \leq \min_s\max_t g(\{s, t\})$.

Let $x_*$ be such that $\max_t{g(\{s, t\})}$ is minimal. 
Consider the subset of $\Mob$ 
(under the identification with $\ExpS{2}{S^1}$) given by 
$$
    d = \{\{x, y\} \in \ExpS{2}{S^1} : x = x_* \text{ or } y = x_*\}.
$$
If $\min_s \max_t g(s,t) = \max{g}$
we are done, otherwise $[d]$ is a cycle in $H_1(\Mob_{\min_s \max_t g(\{s,t\})})$
not equal to $[a]$
and is equal to $[b]$ in $H_1(\Mob_{\max{g}})$.
Thus we have $r_{\min\max{g}}^{\min\max{g}} \geq 2$, $r_{\min\max{g}}^{t_*} = 2$,
and $r_{\min\max{g}}^{\max{g}} = 1$
showing $l \geq \min_s \max_t g(\{s, t\})$.
The proof for multiple global maxima
is similar.
\end{proof}

Therefore for $p \neq 2$, the barcode of the persistence module with $\Z_p$ coefficients is a pairing between all saddle and maxima of $g$ (in the interior, by assumption) as bars of finite length. This makes it a more convenient choice, since most vectorisation methods of persistence modules work assuming the persistence module has only bars of finite length; the infinite bar here $[0,\infty)$ essentially does not contain any relevant geometric and can be discarded. 
\begin{example}
As we will see in Example~\ref{ex:ellipse1}, 
for a ellipse with major axis $2a$ and minor axis $2b$
the only interior critical point is given by the minor axis where the distance is $2b$,
strictly smaller than $\min_s\max_t \AutoD(\{s,t\})$,
showing the upper bound in Proposition~\ref{prop:maxmin} is not always an equality.
\end{example}

\section{Smooth Circle Embeddings in Euclidean Space}
\label{sec:smooth-case}
We now restrict to the case where the loop is smoothly mapped to some Euclidean space $\gamma \in C^k(S^1, \R^d)$. In~\Cref{prop:DT_Ck_continuous}, we show that $\DT$ is a continuous map from $C^k(S^1, \R^d)$ to $C^k(\Mob, \R)$. The other results of this section focus on embeddings $\Emb{2}{S^1}{\R^d} \subset C^k(S^1, \R^d)$ of $S^1$ into Euclidean space. In~\Cref{ssec:crit_geometric_smooth}, we relate non-degenerate critical points of the distance function $\DT_\gamma$ of a smooth embedding  $\gamma$ to local geometric quantities of the $\gamma$, such as local curvature. Furthermore, we show that local geometric properties of $\gamma$ determine the Morse index of such critical points. We then show in~\Cref{ssec:smooth_genericity} that our analysis of critical points in~\Cref{ssec:crit_geometric_smooth} applies to a \emph{generic} (open dense) set of smooth embeddings, on which their distance functions are Morse-supported (\Cref{def:finite_morse}). This result is summarised in~\Cref{thm:main_morse}. We address the open and dense criteria separately in~\Cref{ssec:smooth_open,ssec:smooth_dense} respectively.

For the technical discussions in this section, we often invoke the fact that we can choose an atlas for $\Mob$ such that is compatible with an atlas for $S^1$, which we have shown in~\Cref{prop:mobius_atlas}. This allows us to manipulate derivatives of $\DT_\gamma$ in terms of derivatives of $\gamma$, in some shared set of local coordinates. We recall since $\DT_\gamma$ is induced by a symmetric function on $S^1 \times S^1$, we have an atlas $\Theta = \sett{(\vartheta_i, U_i)}_{i \in \cI}$ of $S^1$, and an atlas $\Phi = \sett{(\varphi_j, W_j)}_{j \in \cJ}$ of $\Mob$, such that on $\varphi_j(W_j) \subset \R^2$,  there are $i_1,i_2 \in \cI$ where
    \begin{equation}
        \DT_\gamma \circ \fibre{\varphi_j}{t_1, t_2} = \norm{\gamma \circ \fibre{\vartheta_{i_1}}{t_1} - \gamma \circ \fibre{\vartheta_{i_2}}{t_2}}^2
        \end{equation}
Since $\gamma \in C^k(S^1, \R^d)$, the derivatives of $\DT_\gamma \circ \inv{\varphi_j}$ are also $k$-times continuously differentiable, so $\DT_\gamma \in C^k(\Mob, \R)$.

Given these atlases, we can compute the gradient and Hessian of $\DT_\gamma$ in local coordinates. Abusing notation, for local coordinates $(t_1,t_2)$ on $W_j \subset \Mob$, we write 
\begin{equation}
    \pdv{\DT_\gamma}{t_a} := \pdv{t_a}( \DT_\gamma \circ \inv{\varphi_j}),
\end{equation}
and similarly for higher derivatives. Likewise for $t$ a local coordinate for $U_i \subset S^1$, we write 
\begin{equation} \label{eq:abuse_of_notation}
   \gamma(t) := (\gamma \circ \inv{\vartheta_i})(t)
\qc \dot{\gamma}(t) := \dv{t} (\gamma \circ \inv{\vartheta_i}) \qand \ddot{\gamma}(t) := \dv[2]{t} (\gamma \circ \inv{\vartheta_i}).
\end{equation}
With this notation, we can express the gradient and Hessian as follows:
\begin{align}
    \mqty( \pdv{t_1} \DT_\gamma \\  \pdv{t_2}\DT_\gamma   ) &=  \mqty( \innerprod{\dot{\gamma}(t_1)}{ \gamma(t_1)- \gamma(t_2)} \\  \innerprod{\dot{\gamma}(t_2)}{ \gamma(t_2)- \gamma(t_1)}  )
    \label{eq:dF} \\
    \mqty(\pdv[2]{\DT_\gamma}{t_1} &  \pdv{\DT_\gamma}{t_1}{t_2}  \\
   \pdv{\DT_\gamma}{t_1}{t_2}   & \pdv[2]{\DT_\gamma}{t_2}) &= 
   \mqty(p( \gamma(t_1)- \gamma(t_2), \dot{\gamma}(t_1), \ddot{\gamma}(t_1))   &  r(\dot{\gamma}(t_1), \dot{\gamma}(t_2))\\
   r(\dot{\gamma}(t_1), \dot{\gamma}(t_2)) & q(\gamma(t_1)- \gamma(t_2), \dot{\gamma}(t_2), \ddot{\gamma}(t_2)) ),
   \label{eq:d2F} \\
   \text{where } \ p(v, x_1', x_1'') &= \norm{x_1'}^2  + \innerprod{x_1''}{v} \label{eq:pqr}\\
    q(v, x_2', x_2'') &= \norm{x_2'}^2- \innerprod{x_2''}{v} \nonumber \\
    r(x_1', x_2') &= -\innerprod{x_1'}{x_2'}.\nonumber
\end{align}
The compatibility of atlases $\Theta$ and $\Phi$ allow us to relate the Whitney-$C^k$ topologies on $C^k(S^1,\R^d)$ and $C^k(\Mob, \R)$. Recall from~\Cref{ssec:differential_topology} that Whitney topology of smooth functions on compact manifolds is defined via uniform convergence of derivatives relative to (any) choice of charts on compact subsets of the manifold. By relating derivatives of $\DT_\gamma$ on a chart from $\Phi$ to that of $\gamma$ on a chart from $\Theta$, we are able to show that $\DT$ is continuous under Whitney topology.   
\begin{proposition} \label{prop:DT_Ck_continuous}
    $\DT: C^k(S^1,\R^d) \to  C^k(\Mob,\R)$ is continuous under Whitney $C^k$-topology. 
\end{proposition}
\begin{proof}
    We want to show for each open subset $\sU \subset C^k(\Mob, \R)$, the preimage $\fibre{\DT}{\sU}$ is open in $C^k(S^1, \R)$. It suffices to show that for any $\DT_\gamma \in \sU$, and a sufficiently small neighbourhood of $\sU(\DT_\gamma)\subseteq U$, there is a some neighbourhood $\sV(\gamma)$ of $\gamma$ such chat $\sV(\gamma) \subseteq \fibre{\DT}{\sU(\gamma)}$. If that is so, then we can write
    \begin{equation*}
        \fibre{\DT}{\sU} = \bigcup_{\gamma \in \fibre{\DT}{\sU}} \sV(\gamma),
    \end{equation*}
    and thus $\fibre{\DT}{\sU}$ is open. 
    
    Consider the charts $\Theta = \sett{(\theta_i, U_i)}_{i \in \cI}$ of $S^1$, and an atlas $\Phi = \sett{(\varphi_j, W_j)}_{j \in \cJ}$ of $\Mob$, as described in~\Cref{prop:mobius_atlas}. Since $\Mob$ is compact, the weak and strong topologies coincide. Thus, an open neighbourhood $\sU$ of $\DT_\gamma$ in $C^k(\Mob, \R)$ can be expressed as a finite intersection of weak subbasic neighbourhoods~\cref{eq:weak_subbasic_nbhd}
    \begin{equation}
       \sU(\DT_\gamma)  = \bigcap_{a=1}^m \sU_a\qc \sU_a = \cN(\DT_\gamma, (\varphi_a, W_a), (\iden_{\R^d}, O_a),  K_a, \epsilon_a)
    \end{equation}
    We note that the restriction that the image of maps in the neighbourhood be within $O_a \subset \R$ is subsumed by the bound $\norm{ f - \DT_\gamma} < \epsilon_a$ on $K_a$.    be within   We claim we can find a neighbourhood
    \begin{equation} \label{eq:gamma_nbhd_projected}
    \sV_a = \cN({\gamma}, (\theta_b, U_b),  (\iden_{\R^d}, V_b), Q_b, \delta_a) \cap \cN({\gamma}, (\theta_c, U_c), (\iden_{\R^d}, V_c), Q_c, \delta_a),
    \end{equation}
  such that $\DT(\sV_a) \subset \sU_a$. If that is so, then we the following finite intersection of open sets is an open set satisfying $\sV(\gamma) \subseteq \fibre{\DT}{\sU(\gamma)}$:
  \begin{equation*}
      \sV(\gamma) := \bigcap_{a=1}^m \sV_a \subset \bigcap_{a=1}^m \fibre{\DT}{\sU_a}= \fibre{\DT}{\sU(\gamma)}.
  \end{equation*}

  We now prove the existence of $\sV_a$ such that $\sV_a \subset \fibre{\DT}{\sU_a}$. We first recall from the definition of the weak subbasic neighbourhood, that $\DT_{\tilde{\gamma}} \in \sU_a$, iff for all $(t_1,t_2) \in K_a \subset \varphi_a(W_a)$, and  $r = 0,\ldots, k$
  \begin{equation} \label{eq:weak_subbasic_Tgamma}
      \norm{D^r  (\DT_\gamma \circ \inv{\varphi_a})\rvert_{({t_1,t_2})} -  D^r(\DT_{\tilde{\gamma}} \circ \inv{\varphi_a})\rvert_{({t_1,t_2})}  } < \epsilon_a
  \end{equation}
  Because charts $\Theta$ and $\Phi$ are coupled in the manner described in~\Cref{prop:mobius_atlas}, there is a pair of charts $(\theta_b, U_b)$ and $(\theta_c, U_c)$ in $\Theta$, such that for $(t_1,t_2) \in \varphi_a(W_a) \subset \theta_b(U_b) \times \theta_c(U_c) =: \vartheta_{bc}(U_{bc})$, 
    \begin{equation}
        \DT_{\gamma} \circ \fibre{\varphi_a}{t_1, t_2} = \norm{\gamma \circ \fibre{\theta_{b}}{t_1} - \gamma \circ \fibre{\theta_{c}}{t_2}}^2 .
    \end{equation}
  We make two observations. 
  First, since $K_a \subset \varphi_a(W_a) \subset \vartheta_{bc}(U_{bc})$, we can obtain subsets $Q_{b} \subset \theta_b(U_b)$ and $Q_c \subset \theta_c(U_c)$, which are projections of $K_a$ onto the respective factors, which are compact as $K_a$ is compact and projections are continuous. Furthermore, $K_a \subseteq Q_b \times Q_c$. Second, we observe that $D^r(\DT_{\gamma} \circ \inv{\varphi_a})$ is a polynomial in $D^l(\gamma \circ \inv{\theta_b})$ and $D^l(\gamma \circ \inv{\theta_c})$ for $l = 0,\ldots, r$. Hence there is some sufficiently small $\delta_a > 0$, such that if for $r = 0,\ldots,k$
    \begin{align*}
        \norm{D^r(\gamma \circ \inv{\theta_b})\rvert_x-D^r(\tilde{\gamma} \circ \inv{\theta_b})\rvert_t } &< \delta_a\qc \forall x \in Q_b \\
        \norm{D^r(\gamma \circ \inv{\theta_c})\rvert_x-D^r(\tilde{\gamma} \circ \inv{\theta_c})\rvert_t } &< \delta_a\qc \forall x \in Q_c , 
    \end{align*}
    then the condition in~\cref{eq:weak_subbasic_Tgamma} for $\DT_{\tilde{\gamma}} \subset \sU_a$ is satisfied. The set of $\tilde{\gamma}$ satisfying these conditions are precisely those in the neighbourhood $\sV_a$, as defined in~\cref{eq:gamma_nbhd_projected}. Hence, we have shown that there exists an open neighbourhood $\sV_a$  of $\gamma$, such that $\sV_a \subset \fibre{\DT}{\sU_a}$.
\end{proof}

\begin{remark}
    Recall weak and strong (Whitney) topology are distinct when the domain manifold is not compact. 
    Consider the distance transform restricted to the interior $\interior{\Mob}$ given by $\gamma \mapsto \DT_\gamma \rvert_{\interior{\Mob}}$.  This function is continuous in weak topology because restriction of domain is continuous in weak topology. However, $C^k_S(\interior{\Mob}, \R)$ is too fine a topology for $\DT$ to be continuous.
\end{remark}

\subsection{Genericity of Embeddings that have Morse-supported \functionname} \label{ssec:smooth_genericity}
We now prove one of our main contributions. We show that the subset of smooth embeddings of $S^1$ into Euclidean space whose \functionname is Morse-supported (\Cref{def:finite_morse}) is generic: that is, it is both open and dense in Whitney $C^k$-topology, for $k \geq 2$. We recall the Morse-supported condition contains two constraints; the first two,~\labelcref{M1,M2}, require that the critical points in the interior be non-degenerate, and there be finitely many of them. The last,~\labelcref{M3}, ensures that homotopy type of sublevel sets of $\DT_\gamma$ are constant as they are extended from the boundary $\partial\Mob$ where $\DT_\gamma$ is zero. We first show in~\Cref{ssec:boundary_critical_points} how restricting \functionname to embeddings constrains the problem. If $\gamma \in \Emb{k}{S^1}{\R^d}$ for $k \geq 2$, then~\labelcref{M3} is automatically satisfied (\Cref{prop:M3}, and~\labelcref{M1} implies~\labelcref{M2} (\Cref{cor:M1impliesM2}). Consequently, the genericity of Morse-supported $\DT_\gamma$ only depends on the genericity of~\labelcref{M1}. We then complete the proof of~\Cref{thm:main_morse} by showing that~\labelcref{M1} holds for \functionname of embeddings on an open and dense subset in~\Cref{ssec:smooth_open,ssec:smooth_dense} respectively.

\begin{restatable}{theorem}{SmoothMorse}
\label{thm:main_morse} Consider $C^k$-embeddings $\Emb{k}{S^1}{\R^d}$ for $k \geq 2$. Let $\Xi$ be the subset of embeddings $\Emb{k}{S^1}{\R^d}$ whose \functionname are Morse-supported (\Cref{def:finite_morse}):
    \begin{equation} \label{eq:Xi}
        \Xi := \{\gamma \in \Emb{k}{S^1}{\R^d} \ : \ \DT_\gamma \text{  satisfies~\labelcref{M1,M2,M3}}\}.
    \end{equation}
    Then $\Xi$ is open and dense in $\Emb{k}{S^1}{\R^d}$ in Whitney $C^k$-topology. Furthermore, this implies $\pershomf(\DT_\gamma)$ and $\pershomf(\DTm_\gamma)$ are pointwise finite dimensional (tame) on $\Xi$.
\end{restatable}

\begin{proof}
    We defer the proof that $\Xi$ is open and dense to~\Cref{prop:smooth_open} and~\Cref{thm:crit_ndg} respectively. If $\Xi$ is the set of embeddings where~\labelcref{M1,M2,M3} are satisfied, then due to the characterisation of Morse-supported functions in~\Cref{lem:morse_supported_tame}, $\pershomf(\DT_\gamma)$ is tame. Note that this also implies  $\pershomf(\DTm_\gamma)$ is tame due to the correspondence between $\pershomf(\DT_\gamma)$  and $\pershomf(\DTm_\gamma)$ described in~\Cref{lem:square_bijection}.
\end{proof}
\subsubsection{Critical Points Near the Boundary} \label{ssec:boundary_critical_points}
We first consider the behaviour of the distance function near the boundary of $\ExpS{2}{S^1}$. In~\Cref{lem:critx_away_from_bdry}, we first show that for $C^2$-embeddings, there are no critical points in the interior of $\Mob$ that are arbitrarily close to the boundary. Thus, we can apply~\Cref{lem:smooth_isolate_finitemorse} to show that~\labelcref{M1} provides sufficient conditions for~\labelcref{M2}.  We also use~\Cref{lem:critx_away_from_bdry} and show in~\Cref{prop:M3} that~\labelcref{M3} is satisfied for all embeddings $\gamma$: there is some regular level $a > 0$ such that $\fibre{\DT_\gamma}{0}$ is a deformation retract of $\sublevel{\DT_\gamma}{a}$. 

In order to establish a lower bound on interior critical values of $\DT_\gamma$, we make use of the fact that $C^2$ compact submanifolds of Euclidean space have positive reach $\tau > 0$~\cite[Remarks 4.20-21]{federer1959curvature}. This constrains the separation of the loop from the tangent plane at any given point on the loop. This places constraints on how close two distinct points along the loop can be paired to form an interior critical point of the distance function.

\begin{lemma} \label{lem:critx_away_from_bdry}
    Consider for $k \geq 2$ an embedding $\gamma \in \Emb{k}{S^1}{\R^d}$ with reach $\tau > 0$. Then the interval $(0, 2\tau)$ are regular values of $\DT_{\gamma}$.
\end{lemma}
\begin{proof}
    Since $\gamma$ is a $C^2$ embedding of a compact manifold, it has positive reach~\cite[Remarks 4.20-21]{federer1959curvature}. We first recall how the reach constrains the variation of a separation vector $\gamma(t_2)-  \gamma(t_1)$ away from the tangent space of the curve spanned by $\dot{\gamma}(t_1)$ at $t_1$: as a consequence of~\cite[Theorem 4.18]{federer1959curvature} we recall~\cref{eq:reach_tangent}: whenever $t_1 \neq t_2$, we have a bound on the Euclidean inner product
    \begin{equation}
        \left\langle \frac{\dot{\gamma}(t_i)}{\norm{\dot{\gamma}(t_i)}}, \frac{\gamma(t_2) - \gamma(t_1)}{\|\gamma(t_2) - \gamma(t_1)\|}\right \rangle^2 \geq  1-\qty(\frac{\|\gamma(t_2) - \gamma(t_1)\|}{2\tau})^2,\quad \text{ for } i = 1,2. 
    \end{equation}
    If $\DT_{\gamma}(t_1,t_2) = \norm{\gamma(t_1) -\gamma(t_2)} \in (0, 2\tau)$, then the right hand side is bounded away from zero. This inequality then bounds $\dd{(\DT_{\gamma})}$ away from zero, given \cref{eq:dF}. Hence there are no critical values $\DT_{\gamma}$ in $(0, 2\tau)$. 

\end{proof}
\begin{corollary} \label{cor:M1impliesM2}
    If  $k \geq 2$ an embedding $\gamma \in \Emb{k}{S^1}{\R^d}$, and all critical points of $\DT_\gamma$ on the interior $\interior{\Mob}$ are non-degenerate~\labelcref{M1}, then there are only finitely many critical points of $\DT_\gamma$ on the interior $\interior{\Mob}$~\labelcref{M2}. 
\end{corollary}
\begin{proof}
    We apply~\Cref{lem:smooth_isolate_finitemorse}: if we have a neighbourhood of $\partial \Mob$  where $\DT_\gamma$ has no interior critical points, and $\DT_\gamma$ is Morse on the interior, then there are only finitely many interior critical points of $\DT_\gamma$. Given~\Cref{lem:critx_away_from_bdry}, we can take $\inv{\DT_\gamma}[0, 2\tau)$ as such a neighbourhood of $\partial \Mob = \fibre{\DT_\gamma}{0}$. 
\end{proof}

\begin{remark}
    The lower bound on the non-zero regular values in \Cref{lem:critx_away_from_bdry} is tight. Consider $\gamma$ being the circle of constant radius $r$. Then the reach is $\tau = r$, and the only non-zero critical value is the maximum $\max \DT_\gamma = 2r$. Note that the maximum is attained by all pairs $\{(z,-z) \ : \ z \in S^1 \}$ in $\UConf{2}{S^1}$ constituting the diameter of the circle. 
\end{remark}

We now show that the absence of interior critical points near the boundary is sufficient to prove~\labelcref{M3} hold for all $\DT_\gamma$ of embeddings: for $a < 2\tau$, the boundary $\fibre{\DT_\gamma}{0} = \partial \Mob$  is a deformation retract of $\sublevel{\DT_\gamma}{a}$. We note the technical complication that $0$ is critical value of $\DT_\gamma$, which limits the application of the standard deformation lemma between sublevels sets of regular values~\cite[Theorem 3.20]{Banyaga2004-my}; nonetheless, we can find an auxiliary function $h$ such that sublevels sets of $h$ coincide with that of $f$ near the boundary, and regular on and near the boundary too. Applying the standard deformation lemma to sublevel sets of $h$ then yields the deformation retraction of sublevel sets of $\DT_\gamma$ near the boundary.

\begin{lemma}\label{lem:boundary_deformation_retract}
    Let $\gamma \in \Emb{k}{S^1}{\R^d}$ for $k \geq 2$. If, for some $a > 0$, we have an interval of regular values  $(0,a]$ for $\DT_\gamma$, then $\fibre{\DT_\gamma}{0} = \partial \Mob$  is a deformation retract of $\sublevel{\DT_\gamma}{a}$ .
\end{lemma}

\begin{proof}
    As a shorthand let $f = \DT_\gamma$. We first construct the auxiliary function $h: \Mob \to \R$.  
    
    We first note that the distance function $\nu: \Mob \to \R$ to the boundary is a smooth function off the image of the axis of symmetry of the infinite strip $\stripinf$ under the quotient $\stripinf \twoheadrightarrow \stripinf/\Z = \Mob$.  Let us construct a function $h  =f + \sigma \circ \nu$, where $\sigma$ is a smooth, monotone non-decreasing bump function, such that $\sigma(0) = 0, \sigma'(0) = 1$, $\sigma(y) = 1$ for $y \geq \epsilon$, and $\sigma(y) \in (0,1)$ for $y  \in (0,\epsilon)$. Here we choose $\epsilon>0$ to be sufficiently small such that $\nu$ is smooth on the sublevel set $\sublevel{\nu}{\epsilon}$, and $\DT_\gamma < a$ on such a sublevel set. The latter constraint is possible because $f$ is continuous, implying there being a sufficiently small $\epsilon > 0$ such that $f(\ucpt{p,q}) < a$. 

    We now claim that $\sublevel{h}{a+1} = \sublevel{f}{a}$. First, $f \leq a$ implies $h \leq a + \sigma\circ \nu \leq a+1$. Conversely, $h \leq a+1$ implies $f \leq a + (1-\sigma \circ \nu)$. The upper bound is only strictly greater than a iff $\sigma \circ \nu < 1$, which is only possible if $\nu < \epsilon$. However we have chosen $\epsilon$ to be sufficiently small such that  $\nu < \epsilon$ implies $f < a$. Hence $h \leq a+1$ implies $f \leq a$. 

    Since $\sigma(\nu(x)) = 0$ iff $x \in \Mob$, we also have $\fibre{h}{0} = \fibre{f}{0} = \partial \Mob$. 
    
    We now show that the $\epsilon$ parameter can be further tuned, such that $[0,a+1]$ are regular values of $h$.
    Without loss of generality we consider a chart $\varphi: U \to \R^2$ such that the distance to $\partial \Mob$ is expressed as the distance  to the diagonal boundary $t_1 = t_2$ of the infinite strip $\stripinf$, i.e. $\nu = t_2 -t_1$. Then, 
    \begin{equation}
        \norm{\nabla h}^2 = \norm{\nabla f}^2  + (2\sigma') (-1,1)\cdot \nabla f + 2(\sigma')^2. 
    \end{equation}
    On $\inv{f}[0,a]$, the first term is positive for $\nu > 0$; the last term is positive everywhere on $\nu \in [0, \epsilon)$. To ensure that $ \norm{\nabla h}^2$ is positive on $\nu \in [0, \epsilon]$, we need to choose $\epsilon$ to be sufficiently small, such that on $\inv{\nu}[0,\epsilon]$, we have $ (-1,1)\cdot \nabla f  >0$ (since by construction $\sigma' > 0$). By Taylor expansion, we obtain
    \begin{equation}
        \pdv{h}{t_1} = \innerprod{\dot{\gamma}(t_1)}{\gamma(t_1) - \gamma(t_2)} = -\norm{\dot{\gamma}(t_1)}^2 (t_2 -t_1) + \frac{(t_2 -t_1)^2}{2} M
    \end{equation}
    where in the remainder term of the Taylor expansion, the term $M$ has magnitude bounded only by finite global bounds on the norm of $\dot{\gamma}$ and $\ddot{\gamma}$. Thus for sufficiently small $\epsilon_1 > 0$ for $t_2-t_1$ we can bound $ \pdv{h}{t_1}$ away from  zero. By symmetry the same analysis applies to $ \pdv{h}{t_2}$ and we can find a sufficiently  small $\epsilon_1 > 0$ for $t_2-t_1$ to bound  $ \pdv{h}{t_2}$ away from zero. Since $t_2-t_1$ expresses the distance to the boundary in local coordinates, by taking $\epsilon = \epsilon_1$, we satisfy $\norm{\nabla h}^2 > 0$ on this chart. Since $\Mob$ is compact we can take a finite open cover and take $\epsilon >0$ to be the minimum of all such $\epsilon_1$'s across the charts to ensure $\nabla h \neq 0$ on  $\inv{f}[0,a] = \inv{h}[0,a+1]$. 

    Because $[0,a+1]$ are regular values of $h$, the standard Morse deformation lemma of regular sublevel sets~\cite[Theorem 3.20]{Banyaga2004-my}  implies $\fibre{h}{0}$ is a deformation retract of $\inv{h}{[0,a+1]}$. Since $\fibre{h}{0} = \fibre{f}{0}$ and $\inv{h}[0,a+1] = \inv{f}[0,a]$, we obtain the desired result. 
    \end{proof}

\begin{proposition} \label{prop:M3}
     For $k \geq 2$ and $\gamma \in \Emb{k}{S^1}{\R^d}$, there is some sufficiently small $a > 0$, such that for all $t \in [0,a]$, the boundary $\partial \Mob = \fibre{\DT_\gamma}{0}$ is a deformation retract of the sublevel set $\sublevel{\DT_\gamma}{t}$ (condition~\labelcref{M3},~\Cref{def:finite_morse}).
\end{proposition}
\begin{proof}
    From~\Cref{lem:critx_away_from_bdry}, we have for $k \geq 2$, the existence of some regular interval $(0,a]$ of $\DT_\gamma$. \Cref{lem:boundary_deformation_retract} then implies $\fibre{\DT_\gamma}{0} = \partial \Mob$  is a deformation retract of $\sublevel{\DT_\gamma}{a}$.
\end{proof}

\subsubsection{Openness}
We now turn to showing that $C^2$-embeddings that have Morse-supported distance functions are \emph{open} in Whitney topology (\Cref{prop:smooth_open}). We approach this problem by covering $\Mob$ with a neighbourhood $L$ of the boundary $\partial \Mob$, and a compact subset $K$ of the interior. We then consider the distance critical points near the boundary and in the interior separately. We address the critical points near the boundary in~\Cref{lem:smooth_open_isolating}, where we show that we can fix $L$ such that there are no critical points on $L$ for distance functions of embeddings in a neighbourhood $\sU$ of a Morse-supported embedding. As such, for embeddings in $\sU$, the only critical points in the interior are contained in some fixed  $K \subset \interior{\Mob}$. We then show in~\Cref{lem:smooth_morse_compact} that for a Morse-supported embedding, we can find a neighbourhood $\sW$ of embeddings such that critical points are non-degenerate in $K$. On the neighbourhood $\sU \cap \sW$ of $\gamma$, the combination of both lemmas then account for all of the separation critical points of embeddings, and show them to be finite and non-degenerate.

\begin{lemma}
\label{lem:smooth_open_isolating}
    Let $k\geq 1$. Assume for $\gamma \in \Emb{k}{S^1}{\R^d}$, the distance transform $\DT_\gamma \in \FSpace{k}{\Mob}{\R}$ has finitely many critical points in the interior $\interior{\Mob}$. Then there is an open neighbourhood $\sU \subset \Emb{k}{S^1}{\R^d}$ of $\gamma$ on which the  following holds: there is a neighbourhood $L \subset \Mob$ of $\partial \Mob$, whereby $\DT_{\tilde{\gamma}}$ has no critical points in $L$, for all $\tilde{\gamma} \in \sU$.
\end{lemma}

\begin{proof}
Let $\{ \ucpt{a_1, b_1}, \ldots,\ucpt{a_n, b_n}\}$ be the set of critical points of $\DT_\gamma$ on the interior of $\Mob$, and consider $R = \{a_1, \ldots, a_n, b_1, \ldots, b_n\} \subset S^1$, the projection of the critical points onto $S^1$. We colloquially call $R$ the critical sites of $\gamma$. Because $R$ is finite, we can cover $S^1$ with closed intervals $J_1, \ldots, J_m$ with non-empty interior, such that each $J_i$ contains at most one element of $R$, and $S^1 = \bigcup_{i=1}^m \interior{J_i}$. 

Let $q: S^1 \times S^1 \twoheadrightarrow \Mob$ be the quotient map, and consider
\begin{align*}
    L &= q\left(\bigcup_{i=1}^m (J_i \times J_i)\right) =  \{\{u,v\} \ : \ \exists i \in [m],\ \text{s.t.}\ u,v \in J_i\}= \bigcup_{i=1}^m L_i \\
    L_i &= q(J_i \times J_i) = \{\{u,v\} \ : \ u,v \in J_i\}
\end{align*}
We now show that by construction, $\partial \Mob \subset L$, and $L$ is a neighbourhood of $\partial \Mob$. For any $\ucpt{u} \in \partial \Mob$, the point $u \in S^1$ is in the interior of some $J_i$, and thus there is some $\delta$ such that the neighbourhood $\{\{v_1,v_2\} \ : \ d(u,v_1), d(u,v_2) < \delta\} $ of $\ucpt{u}$ is contained in $L_i$. Hence $L$ is a neighbourhood of $\partial \Mob$.

Furthermore, a point $\ucpt{u,v} \in L \setminus \partial \Mob$ cannot be a critical point of $\DT_\gamma$, for the following reasons. By construction, there is some $J_i$ such that $u,v \in {J_i}$, and $J_i$ can at most admit one critical site. Thus, $v_1,v_2$ cannot both be critical sites, which is a necessary condition for $\ucpt{u,v} \notin \partial \Mob$ being an interior critical point of $\DT_\gamma$. 

We now show that there is an open neighbourhood $U$ of $\gamma$ in $\Emb{k}{S^1}{\R^d}$, such that for $\tilde{\gamma} \in U$, the interior critical points of $\DT_{\tilde{\gamma}}$ are not in $L$. We show this by contradiction. Suppose there is no such open neighbourhood of $\gamma$ in $\Emb{k}{S^1}{\R^d}$. Then there is a sequence of embeddings $\gamma_n \to \gamma$ in $\Emb{k}{S^1}{\R^d}$ such that for each $\gamma_n$, there is some interior critical point $\ucpt{u_n,v_n} \in L \setminus \partial \Mob$ of $\DT_{\gamma_n}$.  By construction there is some $i(n)$ such that $\ucpt{u_n,v_n} \in L_{i(n)}$. Furthermore, because the number $m$ of cover elements of $L$ is finite, there is some fixed $i \in [m]$ and a subsequence of $(\gamma_n)$ where by each $\gamma_n$ in the subsequence there is an interior critical point $\ucpt{u_n,v_n} \in L_{i}$. Furthermore, because $q$ is continuous and $J_i \times J_i$ is compact, $L_i = q(J_i \times J_i)$ is compact, and thus we can extract a further subsequence of $\gamma_n$ such that $\ucpt{u_n,v_n}$ converges to some point in  $L_{i}$. Because $u_n \neq v_n$, and $u_n,v_n$ are on an interval $J_i$, we can choose $u_n < v_n$ for all $n$ w.o.l.g., and convergence of $\ucpt{u_n,v_n}$ implies $u_n \to u$ and $v_n \to v$ for some $u,v \in J_i$. 

Suppose then $u\neq v$, i.e. the sequence $\ucpt{u_n,v_n} \to \ucpt{u,v}$ converges to some point in the interior of $\Mob$. Because $\ucpt{u_n,v_n}$ is an interior critical point of $\DT_{\gamma_n}$, 
\begin{equation*}
    \langle \gamma_n(u_n) - \gamma_n(v_n), \dot{\gamma_n}(u_n)  \rangle = \langle \gamma_n(u_n) - \gamma_n(v_n), \dot{\gamma_n}(v_n)  \rangle = 0. \tag{$\dagger$}
\end{equation*}
Recall $\gamma_n \to \gamma$ under Whitney $C^k$ topology. This implies uniform convergence of $\gamma_n$ and $\dot{\gamma_n}$ to $\gamma$ and $\dot{\gamma}$ on $J_i$; thus
\begin{equation*}
    \langle \gamma(u) - \gamma(v), \dot{\gamma}(u)  \rangle = \langle \gamma(u) - \gamma(v), \dot{\gamma}(v)  \rangle = 0.
\end{equation*}
This implies $\ucpt{u,v} \in \interior{\Mob}$ is an interior critical point of $\DT_\gamma$. However, because $\ucpt{u,v} \in L_i$, it cannot not an interior critical point. Thus we conclude that $\ucpt{u_n,v_n} \to \ucpt{u}$ on $\partial \Mob$. 

We now show that this sequence is a contradiction of $\gamma$ being an embedding. Let $\gamma^{(j)} \in C^k(S^1,\R)$ denote the $j^\text{th}$-component of $\gamma$. We can then rewrite $(\dagger)$ as 
\begin{equation*}
    \sum_{j=1}^d (\gamma^{(j)}_n(u_n) - \gamma^{(j)}_n(v_n))\dot{\gamma}^{(j)}_n(u_n) = \sum_{j=1}^d (\gamma^{(j)}_n(u_n) - \gamma^{(j)}_n(v_n))\dot{\gamma}^{(j)}_n(u_n) = 0. \tag{$\ddagger$}
\end{equation*}
Because $u_n,v_n \in J_i$ and $J_i$ is a closed interval, the mean value theorem implies there is some $x_{n,j} \in (u_n,v_n)$ such that $\gamma^{(j)}_n(u_n) - \gamma^{(j)}_n(v_n)= \dot{\gamma}^{(j)}_n(x_{n,j})(u_n - v_n)$. Note that $x_{n,j} \to u$ for all $j$. Because $u_n \neq v_n$, we can thus rewrite $(\ddagger)$ as 
\begin{equation*}
   \sum_{j=1}^d \dot{\gamma}^{(j)}_n(x_{n,j})\dot{\gamma}^{(j)}_n(u_n) = \sum_{j=1}^d \dot{\gamma}^{(j)}_n(x_{n,j})\dot{\gamma}^{(j)}_n(v_n) = 0
\end{equation*}
Since we have uniform convergence of $\dot{\gamma}_n \to \dot{\gamma}$ on $J_i$, we then conclude that $\sum_{j=1}^d (\dot{\gamma}^{(j)}(u))^2 = 0$, which contradicts $\gamma$ being an embedding.  Hence there must be an open neighbourhood of $\gamma$ in $\Emb{k}{S^1}{\R^d}$, on which the critical points of the distance transform cannot have interior critical points in $L$. 
\end{proof}

\begin{lemma}
\label{lem:smooth_morse_compact}
    Let $k\geq 2$ and consider $\gamma \in \Emb{k}{S^1}{\R^d}$, such that $\DT_\gamma \rvert_{\interior{\Mob}}$ is Morse on $\interior{\Mob}$. For any compact $K \subset \interior{\Mob}$, there is an open neighbourhood $\sW \subset \Emb{2}{S^1}{\R^d}$ of $\gamma$, such that the critical points of $\DT_{\tilde{\gamma}}$ on $K$ are all non-degenerate for all $\tilde{\gamma} \in W$. 
\end{lemma}
\begin{proof}

     Consider the compatible pair of atlases $\Theta = \sett{(\theta_i, U_i)}$ for $S^1$, and $\Phi = \sett{(\varphi_j, W_j)}$ for $\Mob$, as described in~\Cref{prop:mobius_atlas}. Recall for any $j$ there is some pair of $i,i'$ such that $\varphi_j(W_j) \subset \theta_i(U_i) \times \theta_j(U_j) \subset \R^2$. For $f \in C^k(\Mob, \R)$, let $f_j = f \circ \inv{\varphi_j}: \varphi_j(W_j) \to \R$.  The map $f$ being Morse on the interior of $\Mob$ implies on $W_j \cap \interior{\Mob}$,  
     \begin{equation}
         \mu_j(t_1, t_2) := (\det H_{f_i}\rvert_{(t_1, t_2)})^2 + \norm{\nabla{f_j}}^2 > 0. \tag{$\star$}
     \end{equation}
     Since $W_j$ is an open cover of $\Mob$, consider the cover of $K$ given by $K_j = W_j \cap K$, and $\varphi_j(K_j)$. Because $K_j$ is compact and $\varphi_j$ a diffeomorphism,  $\varphi_j(K_j)$ is compact. Since $f_j$ is in $C^2$, and the R.H.S of $(\star)$ is a polynomial is the first and second derivatives of $f_j$, the function $\mu_j$ is also continuous. Due to $\varphi(K_j)$ being compact, the minimum of $\mu_j$ on  $\varphi_j(K_j)$ is attained by some $\eta_j > 0$. 

     Consider another map $\tilde{f} \in C^2(\Mob, \R)$ and define $\tilde{\mu}_j$ analogously.
     Since $\mu_j$ is a polynomial in the first and second derivatives of $f_j$, there is some choice of $\epsilon_i > 0$ such that if 
     \begin{equation*}
         \norm{D^rf_j - D^r\tilde{f} _j} <\epsilon_i,\ \text{ on } \varphi_j(K_j),\ r = 1,2 \tag{$\star \star$}
     \end{equation*}
     then $\tilde{\mu}_j > \eta_j/2 > 0$ on $\varphi_j(K_j)$.

     Consider then $f = \DT_\gamma$. Recall from~\Cref{prop:mobius_atlas} that the $r$\textsuperscript{th}-order derivatives of $f_j$ is also a polynomial in components of $\gamma_i := \gamma \circ \theta_i$ and $\gamma_{i'} :=\gamma \circ \theta_{i'}$, along with their derivatives up to order $r$. Since $\varphi_j(K_j) \subset \theta_i(U_i) \times \theta_{i'}(U_{i'})$, consider $Q_{jl} \subset \theta_l(U_l)$, which is the projection of $\varphi_j(K_j)$ onto one of the factors in $\theta_i(U_i) \times \theta_{i'}(U_{i'})$ for $l = i,i'$. Note that $Q_{jl}$ is compact, because $\varphi_j(K_j)$ is compact, and $\theta_l$ is continuous. By continuity of polynomials there is some $\delta_j > 0$ such that if 
    \begin{equation*}
         \norm{D^r
         \gamma_l - D^r\tilde{\gamma}_l } <\delta_j,\ \text{ on } Q_{jl},\ r = 0,1,2,\ l = i,i' \tag{$\star \star \star $},
     \end{equation*}
     then $(\star \star)$ holds. This in turn implies $\tilde{\mu_j}  > 0$ for $\tilde{f} = \DT_{\tilde{\gamma}}$, which is equivalent to saying all critical points of $ \DT_{\tilde{\gamma}}$ in $K_j$ are non-degenerate. 

     We now observe that the set of embeddings $\tilde{\gamma}$ satisfying  $(\star \star \star)$ constitute a $C^{k\geq 2}$ weak subbasic neighbourhood $\sW_{j}$ of $\gamma$, on which all critical points of $\DT_{\tilde{\gamma}}$ in $K_j$ are non-degenerate. Because $K$ is compact, it is covered by finitely many $K_j = K \cap W_j$. Therefore on the finite intersection $\sW = \cap_{i \in \cI} \sW_j$, which is again a neighbourhood of $\gamma$ in Whitney $C^k$-topology, for all $\tilde{\gamma} \in \sW$,  all critical points of $ \DT_{\tilde{\gamma}}$ in $K$ are non-degenerate.
\end{proof}

\label{ssec:smooth_open}
\begin{proposition} \label{prop:smooth_open}
        For $k \geq 2$, the subset of embeddings $\Xi \subset \Emb{k}{S^1}{\R^d}$ on which the distance transform is Morse-supported is open in Whitney $C^k$-topology.
\end{proposition}
\begin{proof}
Because the distance transform $\DT_\gamma$ of $\gamma \in \Xi$ has finitely many critical points on $\interior{\Mob}$, consider the open neighbourhood $\sU \subset 
\Emb{k}{S^1}{\R^d}$ of $\gamma$ as described in \Cref{lem:smooth_open_isolating}: for all $\tilde{\gamma} \in \sU$, there exists a neighbourhood $L$ of $\partial \Mob$ such that  $\DT_{\tilde{\gamma}}$ has no critical points on $L$. Then consider $K \subset \interior{\Mob}$ such that $\interior{L} \cup \interior{K} = \interior{\Mob}$. Because $\DT_\gamma\rvert_{\interior{\Mob}}$ is Morse due to $\gamma \in \Xi$, by~\Cref{lem:smooth_morse_compact} there exists an open neighbourhood $\sW \subset 
\Emb{k}{S^1}{\R^d}$ of $\gamma$, such that for all $\tilde{\gamma} \in \sW$, all critical points of $\DT_{\tilde{\gamma}}$ in $K$ are non-degenerate. Consider then  the open neighbourhood $\sU \cap \sW$ of $\gamma$ in $\Emb{2}{S^1}{\R^d}$, and $\tilde{\gamma} \in \sU \cap \sW$. Since $\interior{L} \cup \interior{K} = \interior{\Mob}$ and there are no critical points of $\DT_{\tilde{\gamma}}$ on $L$, all critical points of $\DT_{\tilde{\gamma}}$ in $\interior{\Mob}$ are non-degenerate and thus $\DT_{\tilde{\gamma}}\rvert_{\interior{\Mob}}$ is Morse. Thus for any $\gamma \in \Emb{k}{S^1}{\R^d}$, there is an open neighbourhood $\sU \cap \sW$ on which $\DT_{\tilde{\gamma}}$ satisfies~\labelcref{M1} for $\tilde{\gamma} \in \sU \cap \sW$. Since~\labelcref{M3} is automatically satisfied (\Cref{prop:M3}), and~\labelcref{M1} implies~\labelcref{M2} (\Cref{cor:M1impliesM2}), we conclude that the subset $\Xi$ on which~\labelcref{M1,M2,M3} are all satisfied for their distance transforms is open in Whitney $C^k$-topology. 
\end{proof}

\subsubsection{Density}\label{ssec:smooth_dense}
Having shown that Morse-supported embeddings are an open subset of embeddings, we now show that they are dense using the multi-jet transversality theorem of~\Cref{thm:damon}. We apply it to show that embeddings whose \functionname has no interior degenerate critical points (satisfying~\labelcref{M1}) are dense in Whitney $C^k$-topology. We recall that~\Cref{thm:damon} relates to maps whose multi-jets intersect a Whitney-stratified space in multi-jet space transversally. Our proof proceeds by framing embeddings whose \functionname violate~\labelcref{M1} as those which intersect a Whitney stratified subset $\chi_0$ of the multi-jet space $\Jet{2}{2}{S^1}{\R^d}$. We call this the \emph{degenerate locus} $\chi_0$. By showing this in~\Cref{lem:degenerate_locux_M1}, we convert the problem into showing that the embeddings whose multi-jets avoid $\chi_0$ in $\Jet{2}{2}{S^1}{\R^d}$ are dense. Using dimensionality arguments, we then show in~\Cref{prop:crit_ndg} that for embeddings, the avoidance of $\chi_0$ is equivalent to transverse intersection of the mult-jet with $\chi_0$. Finally, because~\Cref{prop:crit_ndg} implies transverse intersections with $\chi_0$ are dense, we show in~\Cref{thm:crit_ndg} that embeddings whose multijets avoid $\chi_0$ are dense in Whitney $C^k$-topology, and thus the subset of embeddings with Morse-supported \functionname are dense in Whitney $C^k$-topology.

\paragraph{The Degenerate Locus}
Let us consider the {two-fold jet bundle} \Jet{2}{2}{S^1}{\R^d} for maps in $\FSpace{k}{S^1}{\R^d}$ for $k \geq 2$. Let $Z = \sett{(\zeta_i, Z_i)}$ be an atlas of $\Jet{2}{2}{S^1}{\R^d}$ induced by an atlas $\Theta = \{(\theta_i, U_i)\}$ on $S^1$. Consider the prolongation $\prolong{2}{2} f : \Conf{2}{S^1} \to \Jet{2}{2}{S^1}{\R^d}$ of a map $f \in \FSpace{k}{S^1}{\R^d}$. On a chart $\zeta_i: Z_i \to \R^{6d + 2}$, we can express $\prolong{2}{2} f$ in local coordinates as
\begin{align}
      \zeta((\prolong{2}{2} f)(\fibre{\theta_1}{t_1}, \fibre{\theta_1}{t_1}) &= (t_1, f(t_1),\dot{f}(t_1),\ddot{f}(t_1), f(t_2), f(t_2),  \dot{f}(t_2), \ddot{f}(t_2)), 
\end{align}
(here we abuse notation and let $f(t_i) := f \circ \inv{\theta}(t_i)$, and similarly for the derivatives; see~\cref{eq:abuse_of_notation}). We let $(t_1,  x_1,x_1',x_1'', t_2, x_2,x_2',x_2'')$
denote the local coordinates on any such chart, where $t_i \in \R$ are local coordinates for $S^1$, and $x_i,x_i',x_i'' \in \R^d$.

We aim to represent the degenerate locus $\chi_0$ as a set of polynomial constraints on each chart. We first consider critical points of $\DT_f$, and represent them as the intersection between the image of $\prolong{2}{2}f$ and the \emph{critical locus} $C_0$. Due to the characterisation of gradients in~\cref{eq:dF}, consider a local map $C_i: \zeta_i(Z_i) \to \R^2$, given in local coordinates by 
\begin{equation}
    C_i(t_1,  x_1,x_1',x_1'', t_2, x_2,x_2',x_2'') = \pmqty{x_1' \vdot v \\ {x_2'} \vdot {v}} \qq{where} v = x_1 - x_2.
\end{equation}
For a fixed $f \in C^k(S^1,\R^d)$, the evaluation of $C_i \circ \inv{\zeta_i}$ on $\prolong{2}{2} f$ yields the gradients of $\DT_f$ in local coordinates given by the atlas $\Theta$ on $S^1$. We note that the zero-set of the gradient $\fibre{C_i}{0}$ is independent on the choice of chart $\zeta_i$ on $Z_i$. 

We can similarly represent the degeneracy of the Hessian as a polynomial constraint on each chart. Let us write the Hessian matrix of $\DT_f$ in local coordinates as a map on each chart $H_i: \zeta_i(Z_i) \to \R^3$, valued in the space of real symmetric $2\times 2$ matrices:
\begin{align}
    H_i(t_1,  x_1,x_1',x_1'', t_2, x_2,x_2',x_2'')   &= \pmqty{p(v, x_1', x_1'') & r(x_1', x_2') \\ r(x_1', x_2') & q(v, x_2', x_2'') } 
\end{align}
where functions $p,q,r$ are defined in~\cref{eq:pqr}. For a fixed $f \in C^k(S^1,\R^d)$, the evaluation of $H_i \circ \inv{\zeta_i}$ on $\prolong{2}{2} f$ yields the Hessian of $\DT_f$ in local coordinates.  Let $D_i := \det \circ H_i : \zeta_i(Z_i) \to \R$, where $\det(p,q,r) = pq-r^2$. Then the zero-set $\fibre{D_i}{0}$ characterises where maps have degenerate Hessians, and is independent of the choice of chart $\zeta_i$ on $V_i$.  

Having assembled maps $C_i, D_i$ on which characterise critical points and degeneracy of Hessians on $V_i$, we can define the {degenerate locus} $\chi_0$. Consider $\chi_i = (C_i, D_i) \to \R^2$. The intersection of $\prolong{2}{2} f$ with $\fibre{\chi_i}{0,0}$ characterises the degenerate critical points of $\DT_f$ on the local coordinate patch $Z_i$ of the jet space. Given an atlas $Z = \sett{(\zeta_i, Z_i)}$, we define the degenerate locus across the whole  $\Jet{2}{2}{S^1}{\R^d}$ by gluing together the zero sets across each $Z_i$:
\begin{equation} \label{eq:degenerate_locux_union}
    \chi_0 = \bigcup_i \inv{\zeta_i}(\fibre{\chi_i}{0})
\end{equation}
We remark that $\chi_0$ is independent of choice of atlas. Since $\chi_0$ is an algebraic set on any chart, it is a Whitney stratified subset of \Jet{2}{2}{M}{\R^d}~\cite[p.36-38]{Wall1975-rt}. 

\begin{lemma} \label{lem:degenerate_locux_M1} For $k\geq 2$ and $f \in \FSpace{k}{S^1}{\R^d}$, its \functionname $\DT_f: \ExpS{2}{S^1} \to \R$ has no degenerate critical points on its interior $\UConf{2}{S^1}$ (\labelcref{M1}), if and only if  $\prolong{2}{2}f \cap \chi_0 = \emptyset$ .
\end{lemma}
\begin{proof}
    Recall $\UConf{2}{S^1} \subset \ExpS{2}{S^1}$ is double-covered by $\Conf{2}{S^1}$. A point $\ucpt{z_1,z_2} \in \UConf{2}{S^1} = \interior{\ExpS{2}{S^1}}$ is a degenerate critical point of $\DT_f$, iff the lift of $\DT_f$ to $\Conf{2}{S^1}$ has $(z_1,z_2)$ as degenerate critical point. Consider then $\prolong{2}{2}f: \Conf{2}{S^1} \to \Jet{2}{2}{S^1}{\R^d}$.  We observe then that $(z_1,z_2)$ is a degenerate critical point of the lift of $\DT_f$ on $\Conf{2}{S^1}$, iff on a chart for $(z_1,z_2)$, 
    \begin{equation}
        \chi_i(\zeta((\prolong{2}{2} f)(z_1,z_2))= 0.
    \end{equation}
    This follows directly from the construction of $\fibre{\chi_i}{0}$. Because $\chi_0$ is the union over all fibres of $\chi_i$~\cref{eq:degenerate_locux_union}, an multijet $\prolong{2}{2}f$ has empty intersection with $\chi_0$, iff $\DT_f$ has no degenerate critical points. 
\end{proof}

\paragraph{The Degenerate Locus Near Multi-Jets of Embeddings}
We now work to show that the transverse intersection between the degenerate locus $\chi_0$ and a jet $\prolong{2}{2}f$ is empty if $f$ is a $C^2$-embedding. We proceed by describing the local topology of $\chi_0$ near a point $\xi \in \chi_0 \cap \prolong{2}{2}f$. It suffices to work on a local chart $\zeta_i: Z_i \to \R^{6d+2}$, where $\chi_0 \cap V_i$ can be described by $C_i = D_i = 0$. We first focus on the condition where $D_i = 0$, expressing the degeneracy of the Hessian of $\DT_f$. Since the Hessian is a $2 \times 2$ matrix, a degenerate Hessian is either zero, or has rank one. We consider the two cases individually. 

\begin{lemma} \label{lem:chi_mfd} Suppose $f \in \Emb{k}{S^1}{\R^d}$ for $k \geq 2$, and consider $\xi \in \chi_0 \cap \prolong{2}{2}f \cap V_i$. If $H_i(\xi) \neq 0$, then $\chi_0$ on a neighbourhood of $\xi$ is a codimension-3 submanifold of $\Jet{2}{2}{S^1}{\R^d}$. 
\end{lemma}
\begin{proof}
Since $H_i$ is degenerate and non-zero on $\xi$, we restrict ourselves to considering intersections $\xi \in \chi_0 \cap \prolong{2}{2}f$, at which $H_i$ has rank 1. The Jacobian of $\chi_i$ at $\xi$ is given by
    \begin{equation} \label{eq:dchi}
        \dd{\chi_i} = \pmqty{\dd{C_i} \\ \dd{D_i}} = 
                    \pmqty{x_1' & v & 0  & 0 & 0\\
                           x_2' & 0  & v & 0 & 0 \\
                           qx_1'' - px_2'' & 2(qx_1' - rx_2') & 2(px_2' - r x_1') & qv & -pv }\pmqty{\dd{v} \\ \dd{x_1'} \\ \dd{x_2'} \\ \dd{x_1''} \\ \dd{x_2''} }
    \end{equation}
    As $H(\xi) \neq 0$, the constraint that $D_i(\xi) = \det H_i(\xi) = pq-r^2 = 0$ implies at least one of $p,q \neq 0$. Furthermore, because $v \neq 0$, the row vectors in the Jacobian are linearly independent and thus the Jacobian has rank 3 at $\xi$. Since $H \neq 0$ is an open condition, $\chi_i: \zeta_i(Z_i) \to \R^3$ is a submersion on some open ball $U$ around $\xi$. By the submersion theorem, $\fibre{\chi_i}{0} \cap U$ is a submanifold with codimension 3. 
\end{proof}

We proceed to describe the case where $H_i(\xi) = 0$ for $\xi \in \chi_0 \cap \prolong{2}{2}f \cap V_i$. Since $H(\xi)= 0$ iff $p(\xi) = q(\xi) = r(\xi) = 0$, we see by substitution into \cref{eq:dchi} is $\dd{\chi}$ rank 2 at $\xi$ and no longer surjective, thus $\xi$ is no longer a regular point of $\chi$, and the argument in \Cref{lem:chi_mfd} no longer applies in this case. We deal with this by explicitly describing the local stratification of $\chi_0$ on a neighbourhood of $\xi$. 

\begin{lemma}\label{lem:chi_sing} Consider $\xi \in \chi_0 \cap V_i \cap \prolong{2}{2}f$ for $f \in \Emb{k}{S^1}{\R^d}$, for $k \geq 2$. If the Hessian $H_i(\xi) = 0$ vanishes at $\xi$, then $\xi$ lies on a co-dimension 5 stratum of $\chi_0$, and is a singular point of $\chi_0$.
\end{lemma}
\begin{proof}
    In this proof we let $n = 6d+2$ be the dimension of $\Jet{2}{2}{S^1}{\R^d}$. Consider $\xi \in \chi_0 \cap V_i \cap \prolong{2}{2}f$, which is equivalent to  $\xi \in \fibre{\chi_i}{0} \cap \prolong{2}{2}f$.  Let us rewrite the map $\chi_i: V_i \to \R^3$ as the composition
    \begin{equation*}
        \chi_i : V_i \xrightarrow{(H_i, C_i)} \R^3 \times \R^2 \xrightarrow{(\det, \Id)} \R \times \R^2. 
    \end{equation*}
    Let $\Upsilon = (H_i, C_i)$. We can then write $\fibre{\chi_i}{0}$ as 
    \begin{equation*}
        \fibre{\chi_i}{0} = \fibre{((\det, \Id) \circ \Upsilon)}{0} = \fibre{\Upsilon}{\fibre{(\det, \Id)}{0}} = \fibre{\Upsilon}{\fibre{{\det}}{0} \times \Bqty{(0,0)}}.  
    \end{equation*}
    Furthermore, if $H_i(\xi) = 0$, then $\xi \in \fibre{\Upsilon}{0}$. We show that $\Upsilon$ is a submersion on a neighbourhood $U$ of $\xi$. We first compute the Jacobian matrix of $\Upsilon$:
    \begin{equation}
      \dd\Upsilon =  \pmqty{\dd{H_i} \\ \dd{C_i}} = \pmqty{\dd{p} \\ \dd{q} \\ \dd{r} \\ \dd{(x_1' \vdot v)} \\ \dd{(x_2' \vdot v)}} = \pmqty{x_1'' & 2x_1' & 0 & v & 0 \\
             -x_2'' & 0 & 2x_2' & 0 & -v \\
             0 & -x_2' & -x_1' & 0 & 0 \\ 
             x_1' & v & 0  & 0 & 0\\
             x_2' & 0  & v & 0 & 0 } \pmqty{\dd{v} \\ \dd{x_1'} \\ \dd{x_2'} \\ \dd{x_1''} \\ \dd{x_2''}}. 
    \end{equation}
    Since $\xi \in \fibre{\chi_i}{0} \cap \prolong{2}{2}f$ for an embedding $f$, we have $v \neq 0$ and $x_1', x_2' \neq 0$ on a neighbourhood $U$ of $\xi$. We show that this implies on $U$, the Jacobian matrix has full rank, and thus $\Upsilon\vert_U$ is a submersion. We prove by contradiction: suppose the rows are not linearly independent; that is, there are coefficients $(a_0, a_1, a_2, a_3, a_4) \neq 0$, such that 
    \begin{equation}
       \pmqty{a_0 & a_1 & a_2 & a_3 & a_4} \pmqty{x_1'' & 2x_1' & 0 & v & 0 \\
             -x_2'' & 0 & 2x_2' & 0 & -v \\
             0 & -x_2' & -x_1' & 0 & 0 \\ 
             x_1' & v & 0  & 0 & 0\\
             x_2' & 0  & v & 0 & 0 } = 0.
    \end{equation}
    Since $v \neq 0$ at $\xi$, the fact that the bottom right $3 \times 2$ blocks of the Jacobian is zero implies $a_0 = a_1 = 0$. This leaves us with the system of equation
    \begin{align*}
        a_3 x_1' + a_4 x_2' &= 0 \\
        a_2 x_2' - a_3 v & = 0 \\
        a_2 x_1' - a_4 v & = 0
    \end{align*}
    Since $x_1', x_2', v \neq 0$ at $\xi$, either $a_2 = a_3 = a_4 =0$, or they are all non-zero. If we substitute $x_1'$ and $x_2'$ for $v$ in the top equation, then we are left with $2(a_3 a_4/a_2) v = 0$. Since $v \neq 0$ and $a_2,a_3,a_4$ are all non-zero, we reach a contradiction. Thus $\Upsilon$ is a submersion on a neighbourhood $U$ of $\xi$.

    If we choose $U$ sufficiently small, the submersion theorem then implies, we can chose a surjective chart $\mu: (U, \xi) \to (\R^{n},0)$ such that $\Upsilon\rvert_U \circ \inv{\mu}: \R^{n} \to \R^5$ is an orthogonal projection onto the first five coordinates. Thus,  $\fibre{\chi_i}{0} \cap U$ is locally homeomorphic to the following subset in $\R^n$:
    \begin{align*}
        \mu(\fibre{\chi_i}{0} \cap U) &= \mu\qty(\fibre{\Upsilon\rvert_U}{(\fibre{{\det}}{0} \times \Bqty{(0,0)}) \cap \Upsilon(U)}) \\
        &=  \fibre{(\Upsilon\rvert_U \circ \inv{\mu})}{(\fibre{{\det}}{0} \times \Bqty{(0,0)}) \cap \Upsilon(U)} \\
        &=
        \qty(\qty(\fibre{{\det}}{0} \times \Bqty{(0,0)} ) \times  \R^{n-5}) \cap  \mu(U) = \qty(\fibre{{\det}}{0} \times \Bqty{(0,0)} ) \times  \R^{n-5}
    \end{align*}
    where the last equality is due $\mu$ being surjective. 
    We now give an explicit Whitney stratification of $\fibre{{\det}}{0} \subset \R^3$. Recall it is the polynomial $\det(w) = w_1 w_2 - w_3^2  = 0$. We claim $\fibre{{\det}}{0}$ admits a Whitney stratification $\{0\} \subset \fibre{{\det}}{0}$; in other words, $\fibre{{\det}}{0}$ decomposes into strata $\fibre{{\det}}{0} = W^{(0)} \sqcup W^{(2)}$, where $W^{(2)} = \fibre{{\det}}{0}\setminus \{0\}$, and $W^{(0)} = \{0\}$. The differential 
    \begin{equation*}
        \dd{(\det)} = w_2 \dd{w_1} + w_1 \dd{w_2} - 2w_3\dd{w_3} 
    \end{equation*}
     is only singular at the origin $w_1 = w_2 = w_3 = 0$; Hence $W^{(2)} = \fibre{{\det}}{0}\cap \{w \neq 0\}$ is a locally closed, two-dimensional submanifold of $\R^3$. By definition, $W^{(0)}$ is zero dimensional. We show that Whitney condition (B) is satisfied in this case. In this case, this amounts to showing the following: for a sequence $y_i \in W^{(2)}$ such that $y_i \to 0$ such that the secant line $\ell_i$ between line $y_i$ and 0 tends to some line $\ell$, and the tangent planes $\mathsf{T}_i$ at $y_i$ converges to some plane $\mathsf{T}$, then $\ell \subset T$. In fact, in our case, the secant line $\ell_i$ always lies in the tangent planes $\mathsf{T}_i$.  Firstly, for any point $w = (w_1, w_2, w_3) \in W^{(2)}$, note that a point along the secant line $t \cdot (w_1, w_2, w_3)$ to the origin lies in $W^{(2)}$ too (for $t \neq 0$), since the condition $w_1w_2 - w^3 = 0$ is invariant w.r.t. rescaling. Hence $(w_1, w_2, w_3)$ as a tangent vector in $T_w W^{(2)} \subset T_w\R^3$, and the whole secant $t\cdot (w_1, w_2, w_3)$  is in the tangent plane $\mathsf{T}_w$ to $W^{(2)}$ at $w$. Thus, for any sequence $y_i \in W^{(2)}$, the secant $\ell_i$ between $y_i$ and $0$ lies in the tangent plane $\mathsf{T}_i$ at $y_i$, thus if $\ell_i \to \ell$ and $\mathsf{T}_i \to \mathsf{T}$, we have $\ell \subset \mathsf{T}$. 

     Having obtained a stratification of $\fibre{{\det}}{0}$, we can stratify the embedding of $\chi_0$ on the neighbourhood $U$ of $\xi$, recalling $\mu$ is a diffeomorphism:
     \begin{align*}
         \mu(\fibre{\chi_i}{0} \cap U) &= \qty(\fibre{{\det}}{0} \times \Bqty{(0,0)} ) \times  \R^{n-5}\\
          &= \qty((W^{0} \sqcup W^{(2)}) \times \Bqty{(0,0)} ) \times  \R^{n-5}\\
         &= \qty((0,0,0,0,0) \times \R^{n-5} )\sqcup \qty(  W^{(2)} \times \{(0,0)\} \times \R^{n-5}). 
     \end{align*}
     where the last equality is due to $W^{(0)} = (0,0,0) \in \R^3$. Since $\mu(\xi) = 0$,  we conclude that $\xi$ lies on a co-dimension 5 stratum of $\chi_0$.
\end{proof}

\begin{proposition} \label{prop:crit_ndg}
    Let $k \geq 2$. If $f \in \Emb{k}{S^1}{\R^d}$, then $\prolong{2}{2}f$ is transverse to the degenerate locus $\prolong{2}{2}f \transverse 
    \chi_0$ in \Jet{2}{2}{S^1}{\R^d}, if and only if   $\prolong{2}{2}f \cap 
    \chi_0 = \emptyset$. 
\end{proposition}
\begin{proof}
    Recall $\chi_0$ is a Whitney stratified space.  In \Cref{lem:chi_mfd,lem:chi_sing}, we have shown that $\prolong{2}{2}f$ intersects $\chi_0$ at strata with codimension at least three. If the intersection is non-empty and transverse, then at such intersections $\rank \dd (\prolong{2}{2}f)$ must be greater or equal to the codimension of the strata of $\chi_0$ at intersections. Since $\prolong{2}{2}f$ is an immersion of $\Conf{2}{S^1}$, this implies $\rank \dd (\prolong{2}{2}f) = \dim \Conf{2}{S^1} = 2$, which is less than the codimension of the $\chi_0$ strata at intersections.  Thus, the dimensional constraints imposed by transversality~\cref{eq:transversality} implies the intersection $\prolong{2}{2}f \cap 
    \chi_0$ is empty. 
\end{proof}
Having characterised embeddings whose~\functionname satisfy~\labelcref{M1} as those with transversal intersections of two-fold 2-jets with the critical locus, we can then apply the transversality theorem~\Cref{thm:damon} to show that such embeddings are dense in Whitney $C^k$ topology, for $k \geq 2$. We note that naive application of~\Cref{thm:damon} would imply density for $k \geq 3$; we extend this to $k = 2$ by borrowing the argument used in the proof of Theorem 9.25 in~\cite{arnal2023critical}.

\begin{proposition}\label{thm:crit_ndg}
    For $k \geq 2$, the subset of embeddings $\Emb{k}{S^1}{\R^d}$ whose \functionname satisfy~\labelcref{M1} are dense in $\Emb{k}{S^1}{\R^d}$ under Whitney $C^k$-topology. Hence the subset of embeddings $\Xi \subset \Emb{k}{S^1}{\R^d}$ (\cref{eq:Xi}) whose \functionname are Morse-supported are dense in $\Emb{k}{S^1}{\R^d}$ under Whitney $C^k$-topology.
\end{proposition}
\begin{proof}
    We first recall for every $1\leq k \leq r$, the $C^r$-embeddings $\Emb{r}{S^1}{\R^d}$ are a dense subset of the $C^k$-embeddings $\Emb{r}{S^1}{\R^d}$~\cite[Theorem 2.7]{hirsch2012differential}, under Whitney $C^k$-topology. Suppose $k \geq 2$; density implies for any open Whitney-$C^k$ neighbourhood $U$ of $f \in \Emb{k}{S^1}{\R^d}$, we can choose some $f_1 \in U$ such that $f_1 \in \Emb{r}{S^1}{\R^d}$ (for $k \geq 3$, we choose $f_1 = f$). Consider then $r \geq 3$. Because residual subsets are dense in Whitney $C^r$-topology~\cite[Theorem 4.4]{hirsch2012differential},~\Cref{thm:damon} implies that we can choose some $f_2 \in \Emb{r}{S^1}{\R^d}$ in any open Whitney-$C^r$ neighbourhood $V$ of $f_1$, such that $\prolong{2}{2}f_2$ intersects $\chi_0$ transversely. By~\Cref{prop:crit_ndg,lem:degenerate_locux_M1}, this is equivalent to saying $\DT_{f_2}$ has no non-degenerate critical points in the interior of $\Mob$. Because $k \leq r$, any Whitney $C^r$-neighbourhood $V$ of $f_1$ (as a $C^r$-embedding), is also a Whitney $C^k$-neighbourhood of $f_1$ (as a $C^k$-embedding). Since $f_1 \in U$, and $U$ is Whitney $C^k$ open, we can choose the Whitney $C^k$ neighbourhood $V$ of $f_1$ such that $V \subset U$, which implies $f_2 \in U$. Hence for any $f \in \Emb{k}{S^1}{\R^d}$, and any Whitney $C^k$ neighbourhood $U$ of $f$, we can find a $f_2 \subset U$ a $C^k$-embedding, such that $\DT_{f_2}$ has only non-degenerate critical point on the interior. Because of~\Cref{prop:M3,cor:M1impliesM2}, $\DT_{f_2}$ also satisfies~\labelcref{M2,M3}. Thus $\Xi$ is dense in $\Emb{k}{S^1}{\R^d}$ under Whitney $C^k$-topology.
\end{proof}

\subsection{Geometric Characterisation of Critical Points}
\label{ssec:crit_geometric_smooth}
Suppose $\DT_\gamma$ is Morse-supported for an embedding $\gamma \in \Emb{2}{S^1}{\R^d}$, we can relate the Morse index of interior critical points of $\DT_\gamma$ with geometric characteristics of the curve from the formulae for the gradient and Hessian~\cref{eq:dF,eq:d2F}. 
In that case, the entries of the Hessian~\cref{eq:pqr} at $\ucpt{z_1, z_2}$ can be simplified as 
\begin{align}
    p &= 1 - \kappa_{12}, \\
    q &= 1 - \kappa_{21}, \\
    r &= - \cos{\theta_{12}},
\end{align}
where $v_{ij} = x_j - x_i$, $\kappa_{ij} := \langle x_i'', v_{ij}\rangle$ and $\cos{\theta_{12}} = \langle x_1', x_2'\rangle$.

The Morse index of a critical point $\ucpt{z_1,z_2}$ in the interior $\UConf{2}{S^1}$ of $\Mob$ can be read off the entries of the determinant of the Hessian. If $\det H < 0$, then $\ucpt{z_1,z_2}$ is a saddle point; if $\det H >0$, then $p$ and $q$ is of the same sign; if that sign is positive, then $\ucpt{z_1,z_2}$ is a minimum; if that sign of negative, then $\ucpt{z_1,z_2}$ is a maximum. We shall now give a geometric interpretation of these critical points. 

The critical point condition
\begin{equation}
    \langle x_1', v_{12} \rangle  = \langle x_2', v_{21}\rangle =  0
\end{equation}
gives that 
$x_1'$ and $x_2'$ are perpendicular to $v_{12} = -v_{21} = x_2 - x_1$.
We will say two segments centred at $x_1$ and $x_2$
are \emph{opposing} if this condition is satisfied.
We can now simply restrict ourselves to the 3-dimensional space spanned by $\{x_1', x_2', v_{12}\}$.
The determinant of the Hessian $pq - r^2$ is
\begin{equation}
\label{eqn:hessian-xv}
    \frac{1}{\|x_1'\|\|x_2'\|}
    \left[ 
    (1 - \kappa_{12})( 1 - \kappa_{21}) - \cos^2{\theta_{12}}
    \right].
\end{equation}
where $\kappa_{ij} = \langle x_i''/\|x_i'\|, v_{12}\rangle$
and $\theta_{12}$ is the angle between $x_1'$ and $x_2'$.

To keep the following analysis simple,
we shall assume $|\kappa_{12}| = |\kappa_{21}|=: \kappa$.
Now (after removing the positive scaling) expression~\ref{eqn:hessian-xv} becomes
\begin{equation}
    \label{eqn:hessian-fixed-k}
    \sgn(\kappa_{12})\sgn(\kappa_{21})\kappa^2
    - (\sgn(\kappa_{12}) + \sgn(\kappa_{21}))\kappa + \sin^2{\theta_{12}}.
\end{equation}

We can now give a geometric characterisation of critical points, as shown in Figure~\ref{fig:critical-points}.

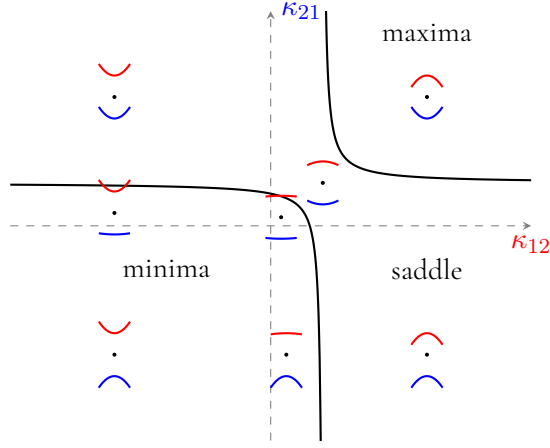
\begin{figure}[h]
    \centering
    \begin{tikzpicture}
\begin{axis}[
    axis line style={gray, dashed},
    axis lines=middle,
    xlabel={\textcolor{red}{$\kappa_{12}$}},
    ylabel={\textcolor{blue}{$\kappa_{21}$}},
    xmin=-5, xmax=5,
    ymin=-5, ymax=5,
    grid=both,
    every axis x label/.style={at={(current axis.right of origin)}, anchor=north},
    every axis y label/.style={at={(current axis.above origin)}, anchor=west},
    xtick=\empty,
    ytick=\empty
]

\addplot[black, thick, smooth, domain=1.01:5, samples=300] expression {(x - 0.75)/(x-1)};
\addplot[black, thick, smooth, domain=-5:0.99, samples=300] expression {(x - 0.75)/(x-1)};

\foreach \px/\py in {3/3, -3/3, 3/-3, -3/-3, 1/1, -3/0.3, 0.3/-3, 0.2/0.2} {
    \addplot[mark=*, mark size=0.05em, black, only marks] coordinates {(\px, \py)};
    \addplot[red,  thick, smooth, domain=-0.3+\px:0.3+\px] expression {-\px*(x-\px)^2 + \py + 0.5};
    \addplot[blue, thick, smooth, domain=-0.3+\px:0.3+\px] expression {\py*(x-\px)^2 + \py - 0.5};
}
\node at (axis cs: -2,-1) {minima};
\node at (axis cs: 3, 4.5) {maxima};
\node at (axis cs: 3, -1) {saddle};
\end{axis}
\end{tikzpicture}
    \caption{Classification of critical points in terms of $\kappa_{12}$ and $\kappa_{21}$ for $\theta_{12} = \pi/3$ where $d>2$ (\cref{eqn:hessian-xv}).
    Displayed are two dimensional cartoons of the 9 classes of critical point,
    organised by $\sgn(\kappa_{ij})$ and Morse index.}
    \label{fig:critical-points}
\end{figure}

\begin{description}
    \item[Curve away: $\sgn{\kappa_{12}}=\sgn(\kappa_{21})=-1$.]
    This case is particularly simple;
    the expression~\ref{eqn:hessian-fixed-k}
    is always positive, $p$ is positive, and thus
    all opposing segments that curve away are local minima.
    \item[Curve towards: $\sgn(\kappa_{12}) = \sgn(\kappa_{21})= 1$.]
    This occurs when two opposing segments are curving towards each other.
    Then, for $0\leq \kappa <1-\cos\theta_{12}$ we have a local minimum,
    for $\kappa >1 +\cos{\theta_{12}}$ we have a local maximum
    and for $1-\cos{\theta_{12}} < \kappa < 1 +\cos{\theta_{12}}$ we have a saddle point.
  
    \item[Curve together: $\sgn(\kappa_{12}) = - \sgn(\kappa_{21})$.]
    WLOG assume $\sgn(\kappa_{12}) = 1$.
    If $-\sin{\theta_{12}}<\kappa<\sin\theta_{12}$
    we have a local minimum,
    and if $\kappa <-\sin{\theta_{12}}$ or $\kappa > \sin{\theta_{12}}$
    we have a saddle point.
\end{description}

For the $\theta =0$ case, where $x_1', x_2'$ and $v_{12}$ are coplanar, we can give further intuition.
Define the radius of curvature relative as $\rho = \|v_{12}\|/\kappa =|\langle x_1''/\|x_1'\|, v_{12}/\|v_{12}\|\rangle|^{-1}$.
Consider the circle with centre the midpoint of $x_1$ and $x_2$ with radius $\frac{1}{2}\|v_{12}\|$.
If the radius of curvature $\rho$ is less than the radius $\frac{1}{2}\|v_{12}\|$ we have a local maximum,
if it is greater, we have a saddle point.

\begin{example}
\label{ex:ellipse1}
Consider an ellipse embedded in $\R^2$ with major axis $2a$ and minor axis $2b$.
There are exactly two critical points lying in $\UConf{2}{S^1}$ corresponding to the major and minor axes.
The radius of curvature on each side of the major axis is $b^2/a$
and $a^2/b$ for the minor axis.
As $b^2/a < a$ the major axis corresponds to a local (and global) maximum,
and as $a^2/b > b$ the minor axis corresponds to a saddle point.
\end{example}

We summarise these observations in the following theorem.

\begin{restatable}{theorem}{SmoothChar}
    \label{thm:SmoothChar}
    Suppose $\DT_\gamma$ is Morse-supported for an embedding $\gamma \in \Emb{2}{S^1}{\R^d}$.
    For a point $\ucpt{z_1, z_2} \in \Mob$ define
    $x_i = \gamma(z_i)$,
     $v_{ij} = x_j - x_i$, $\kappa_{ij} := \langle x_i''/_\|x_i'\|, v_{ij}\rangle$ and $\cos{\theta_{12}} = \langle x_1', x_2'\rangle$.
     Then $\ucpt{z_1, z_2}$
    is critical for $\DT_\gamma$ if 
    $$
    \langle x_1', v_{12} \rangle  = \langle x_2', v_{21}\rangle =  0
    $$
    If $(1 - \kappa_{12})( 1 - \kappa_{21}) <\cos^2{\theta_{12}}$ this is a saddle point.
    Otherwise by non-degeneracy $(1 - \kappa_{12})( 1 - \kappa_{21}) > \cos^2{\theta_{12}}$.
    Then, 
    if $\kappa_{12}, \kappa_{21} > 1$ this is a local maximum, 
    else if $\kappa_{12}, \kappa_{21} < 1$, this is a local minimum.
\end{restatable}


\section{Piecewise-Linear Loops}
\label{sec:pl-case}
In this section, we consider the case where $\gamma$ is finite piecewise linear. 
Similar to the analysis in~\cref{sec:smooth-case} for the smooth case, we wish to geometrically interpret the features in $\pershomf_\ast(\DT_\gamma)$ for generic embeddings $\gamma$ by relating the birth and death of such features to non-degenerate critical points of $\DT_\gamma$, and characterise these non-degenerate points geometrically.
However, the \functionname of a PL loop is only piecewise smooth, and thus the smooth Morse theory used in~\cref{sec:smooth-case} cannot be applied to the PL case. While there is a Morse theory for piecewise smooth functions~\cite{BARTELS1995385,Agrachev97} for generalised (Clarke) critical points~\cite{Clarke1975-zo,clarke1990optimization}, $\DT_\gamma$ may admit degenerate critical points that are excluded from the analysis of~\cite{BARTELS1995385,Agrachev97} (see~\Cref{rmk:clarke}). We also remark that even though $\gamma$ is piecewise linear, neither $\DTm_\gamma$ nor $\DT_\gamma$  are piecewise linear themselves, excluding the application of piecewise linear Morse theory~\cite{Banchoff1967-tn,Fugacci2020-qx}.

As such, we take a discrete approach by constructing a filtered simplicial complex $\nerve{\cV^\bullet}$ that is homotopy equivalent to the sublevel set filtration $\Gamma^\bullet$ of $\DT_\gamma$. The filtered complex $\nerve{\cV^\bullet}$ is the nerve of a filtered covering $\cV^\bullet$ of sublevel sets of $\DT_\gamma$. This allows us to apply the Conley index theory (described in~\cref{ssec:morse_sets}) for the Morse sets of $\nerve{\cV^\bullet}$ to describe the birth and death of features in $\pershomf_\ast(\DT_\gamma)$. In~\cref{ssec:pl-loops}, we first restrict our attention to finite piecewise linear embeddings $\gamma: S^1 \to \R^d$ which have no two edges parallel (see conditions~\labelcref{C1,C2,C3}). This allows us to subsequently characterise the Morse sets of the filtration $\nerve{\cV^\bullet}$ in terms of analogues of smooth non-degenerate critical points called $n$-saddles (see \Cref{ssec:morse_sets}).
Therefore, we are able to compute the homological critical values of $\DT_\gamma$
via the combinatorial object $\nerve{\cV^\bullet}$.
In~\cref{ssec:filtration-of-the-nerve} we explicitly describe the filtration of the simplicial complex $\nerve{\cV^\bullet}$, which allows us to exactly compute the persistence diagrams $\pershomf_\ast(\DT_\gamma)$  using algorithms for computing persistent homology.
Then in~\cref{ssec:pl-persistence} we show
that for an embedding satisfying the non-degeneracy conditions~\ref{C1}-\ref{C3}
all Morse sets (apart from at value $0$) are simplices, along with some of their faces,
and so can be thought of a critical points, in analogy with the smooth case.
Finally in~\cref{ssec:pl-geometric-characterisation}
we build a bijective correspondence between such critical simplices and points in $\UConf{2}{S^1}$
whose images under $\gamma$
satisfy a local geometric condition (\Cref{dfn:pl-critical-index}).
This condition is in terms of acuteness/obtuseness of the four angles created by a separation (chordal) vector,
and from which the index of the critical simplex can be read off.
The Morse type of these critical simplices bears similarity with the smooth case: if locally the opposing segments curve away from each other, the simplex has Conley index $0$,
if they curve towards each other, it has index $2$, and if they curve together, the simplex has index $1$.

\subsection{Piecewise-linear loops}
\label{ssec:pl-loops}
Suppose we are given a set of points $x_i \in \R^d$ cyclically indexed over $i = 0, \ldots, n-1 \in\Z_n$ for $n > 3$. We consider a piecewise linear loop interpolating these points $f: S^1 \to \R^d$.  In order to parametrise $f$ explicitly, we set out some notation. Let $u_i = x_{i+1} - x_i$ denote the finite difference vector between successive points in the cycle  $(x_i)_{i \in \Z_n}$. We define $l_i = \norm{u_i}$ to be the length of each linear segment interpolated between $x_i$ and $x_{i+1}$. Let us also consider the total length along the loop $L = \sum_{i = 0}^{n-1} l_i$, and partial lengths $L_i = \sum_{j < i}l_i/L$. We also recall the universal covering $\scre(t) := e^{2\pi \imath t}$ of $S^1$.

A piecewise linear interpolation of $(x_i)_{i \in \Z_n}$ is then a map $f: S^1 \to \R^d$ that sends $f(\scre(L_i)) = x_i$ , and 
\begin{equation} \label{eq:PLprimitive}
    f \circ \scre(t) = \frac{L}{l_i}\qty((t- L_i)x_{i+1} - (t - L_{i+1})x_i ) =  \frac{L}{l_i} (t - L_i)u_i + x_i \quad \text{for\ } t \in [L_i, L_{i+1}].
\end{equation}

\paragraph{Non-degenerate Piecewise Linear Loops}
We restrict ourselves to piecewise linear interpolations $f:S^1\to \R^d$ that satisfy the following properties. We call a cyclic sequence of data points $(x_i)_{i \in \Z_n}$  \emph{non-degenerate}, if 
\begin{enumerate}[label = (C\arabic*)]
    \item \label{C1} No distinct pair of points in the sequence are identical: $x_i \neq x_j$ for $i \neq j$;
    \item \label{C2} $f$ is a topological embedding: the interiors of interpolated segments may not cross one another; and
    \item \label{C3} No pair of distinct segments are parallel or anti-parallel: $|\innerprod{u_i}{u_j}| < \norm{u_i} \norm{u_j}$ for $i \neq j$.
\end{enumerate}
For $d > 2$, any `generic' cyclic sequence $(x_i)_{i \in \Z_n}$ is non-degenerate: we only require the set $\sett{x_0, \ldots, x_n}$ to be in \emph{general position}: no $k$ of such points lie in a $(k-2)$-dimensional affine subspace of $\R^d$ for $k = 2,\ldots, d+1$. The general position requirement immediately implies~\labelcref{C1} regardless of dimension $d$, as a pair of points lying on a $0$-dimensional affine subspace implies they are identical. The following lemmas show that~\labelcref{C2,C3} also follow from the general position assumption, if $d > 2$. Indeed, these observations extend to the case where points $x_i$ are vertices of a geometric graph, where edges between vertices are interpolated by linear segments. 
\begin{lemma}
    Consider a PL geometric graph in $\R^d$ with vertex set $X = (x_0,\ldots, x_{n-1}) \subset \R^d$. If $d > 2$, and the vertex set $X$ is in general position, then the PL geometric graph is an embedding of the underlying graph --- that is, the relative interior of an edge does not intersect another edge. 
\end{lemma}
\begin{proof}
    If $n =1$ (only one vertex) the map sending a vertex into $\R^d$ is trivially an embedding; if $n = 2$ (two vertices), then the general position requirement requires the two vertices two be distinct, and thus the interpolation $\overline{x_0x_1}$ is an embedding of the graph with one edge.

    Let us then consider the case where $n > 2$. First, we note that the general position requirement implies all points in $X$ are distinct. So for $i < j$, two closed segments $\overline{x_ix_{i+1}}$ and  $\overline{x_jx_{j+1}}$ can either intersect only at single point $x_j$ (with $j = i+1$) at the end of segment; or the relative interior of one segment intersects a single point along the other, which may either be in the relative interior or ends (with $j > i+1$). The latter  requires the four distinct points $x_i, x_{i+1},x_{j}, x_{j+1}$ to lie on a two dimensional affine subspace of $\R^d$, which contradicts the general position requirement on $X$ for $D > 2$. 
\end{proof}
\begin{lemma}
\label{lem:no-edges-parallel}
If $d > 2$, a PL geometric graph in $\R^d$ in general position has no two edges parallel.
\end{lemma}
\begin{proof}
    Suppose we have a PL geometric graph with edges 
    $\overline{x_i x_{i+1}}$ and $\overline{x_j x_{j+1}}$ parallel.
    Then for some $\lambda \in \R$, $x_{j+1} = x_j + \lambda(x_{i+1} - x_i)$ 
    hence $x_{i+1}$ is in the affine span of $x_i, x_{i+1}$, and $x_{j}$ which has affine dimension no more than 2,
    thus the graph is not in general position, a contradiction.
\end{proof}

\subsection{The Distance Function}
\label{ssec:the-distance-function}
We now consider the chordal distance transform $\DT_\gamma $ of a piecewise linear interpolated loop $\gamma: S^1 \to \R$, satisfying the embedding conditions~\labelcref{C1,C2}. Since $\gamma$ is piecewise linear, $\DT_\gamma$ is a piecewise smooth on $\ExpS{2}{S^1}$. Here we define $\DT_\gamma$ explicitly on pieces $V_{ij} \subset \ExpS{2}{S^1}$ associated to pairs of points from segments $\overline{x_i,x_{i+1}}$ and $\overline{x_j,x_{j+1}}$. Given we have chosen $\cU$ to be a good open cover, let $\Phi$ and $\Theta$ be the compatible atlases for $\ExpS{2}{S^1}$ and $S^1$ given in~\Cref{prop:mobius_atlas}. On a chart $\varphi_{ij}: W_{ij} \to \R^2$, recall from the previous section that we can express $\DT_\gamma$ in local coordinates $(t_1,t_2)$ as 
\begin{equation*}
      \DT_\gamma(t_1,t_2) := \frac{1}{2}\norm{\gamma(t_1) - \gamma(t_2)}^2
\end{equation*}
We can define $V_{ij} \subset W_{ij}$ to be the subset described in local coordinates by $t_1 \in [L_i, L_{i+1}]$ and $t_2 \in [L_j, L_{j+1}]$:
\begin{align}
    V_{ij} &= \inv{\varphi_{ij}}(([L_i,L_{i+1}] \times [L_j,L_{j+1}]) \cap \stripinf) \label{eq:PLpieces}.
\end{align}
The distance transform $\DT_\gamma$ thus evaluates on $V_{ij}$ the distance between a pair of points from segments $\overline{x_i,x_{i+1}}$ and $\overline{x_j,x_{j+1}}$ respectively. We give an illustration of the elements $V_{ij}$ in~\Cref{fig:PL-nerve}. If $i = j$, we observe that $V_{ii} = \ExpS{2}{[L_i, L_{i+1}]}$ is a triangle; otherwise, it is a rectangle. If $i$ and $j$ are adjacent on the cycle $\Z_n$, then it intersects the boundary $\partial \ExpS{2}{S^1}$ at one of its corners, corresponding to $\ucpt{z_i = z_j}$. If $i$ and $j$ are at least one element apart on the cycle $\Z_n$, then $V_{ij} \subset \UConf{2}{S^1}$. We note that $\bigcup_{ij \in \Z_n}V_{ij} = \ExpS{2}{S^1}$. That is, the collection of subsets $\cV = \{V_{ij}\}_{ij \in \Z_n}$ form a closed cover of $\ExpS{2}{S^1}$. 

We can now describe $\DT_\gamma$ by prescribing it on each piece $V_{ij}$ with local coordinates. If we consider $t_1 \in [L_i, L_{i+1}]$ and $t_2 \in [L_j, L_{j+1}]$, 
\begin{equation}
    \label{eqn:sdf-quadratic}
   \DT_\gamma (t_1,t_2) =  \frac{L^2}{2}\norm{s_1\tau_i - s_2\tau_j + \frac{x_i - x_j}{L}}^2 
\end{equation}
where as a shorthand we write $s_1 = t_1-L_i$,  $s_2 = t_2 - L_j$, and let $\tau_k = u_k/l_k$ denote the unit tangent vector. 
\begin{remark} \label{rmk:clarke}
    Note that while $\DT_\gamma$ is piecewise smooth, the piecewise smooth Morse theory of~\cite{BARTELS1995385,Agrachev97} cannot be applied to characterise all Clarke critical points of $\DT_\gamma$. If $(L_i, L_j)$ is a Clarke critical point of $f$, then it is incident on 4 smooth pieces $V_{i-1,j-1},V_{i-1,j},V_{i,j-1},V_{ij}$, which violates the condition for non-degeneracy; which requires any three of the four gradients of the smooth functions on $V_{i-1,j-1},V_{i-1,j},V_{i,j-1},V_{ij}$ evaluated at $(L_i, L_j)$ to be linearly independent (see~\cite[Def 2.2 ND1)]{Agrachev97} or~\cite[Def 1.3(ND1)]{BARTELS1995385}). As $\ExpS{2}{S^1}$ is a two-dimensional manifold, this any collection of three gradients cannot be linearly independent. Therefore, if $(L_i,L_j)$ is a Clarke critical point, it cannot be non-degenerate in the framework of~\cite{BARTELS1995385,Agrachev97}.
\end{remark}
\paragraph{Sublevel sets of the distance function}
Having defined $\DT_\gamma$ piecewise on each closed cover element $V_{ij}$, we now describe sublevel sets $\AutoDsq^a$ by their restrictions to each piece $V_{ij}$
 \begin{equation}
     V_{ij}^a := \AutoDsq^a \cap V_{ij}.
 \end{equation}
We note that $\cV^a := \{V_{ij}^a\}_{i,j\in \Z_n}$ is a cover of $\AutoDsq^a$. We recall there are charts that take $V_{ij}$ to a convex subset of $\R^2$. We show in the lemma below that this is also true for sublevel sets $ V_{ij}^a $, and intersections with other cover elements of $\cV^a$. We recall a good cover is a cover where any finite intersection of cover elements are contractible. 
\begin{lemma} \label{lem:cover_PL}
    $\cV^a $ is a good cover of $\AutoDsq^a$ by closed subsets; furthermore, there is an atlas $\Phi = \{(\varphi_{ij}, W_{ij})\}$ of $\ExpS{2}{S^1}$ that takes any non-empty intersection $V_{i_1j_1}^a \cap \cdots \cap V_{i_m j_m}^a$ homeomorphically to a closed convex subset 
    \begin{equation*}
        \varphi_{i_1j_1} (V_{i_1j_1}^a \cap \cdots \cap V_{i_m j_m}^a) \subset \R^2.
    \end{equation*}
\end{lemma}
 \begin{proof}
     W.l.o.g. we normalise $L = 1$.  We first prove that each $V_{ij}^a$ is convex. Let us choose compatible atlases $\Theta = \{(\vartheta_i, U_i)\}$ and $\Phi = \{(\varphi_{ij}, W_{ij})\}$ as described in~\Cref{prop:mobius_atlas}, where each $[L_i,L_{i+1}]$ is contained in some unique open arc $U_i \in \cU$. Evaluating $\DT_\gamma$ on local coordinates $s_i, s_j$~\cref{eqn:sdf-quadratic}. the derivatives and Hessians of $\DT_\gamma$ are given by 
\begin{align}
    \pdv{\DT_\gamma}{s_i} &=  s_i - \innerprod{\tau_i}{\tau_j} s_j + \innerprod{x_i - x_j}{\tau_i}\\
    \pdv{\DT_\gamma}{s_j} &= - \innerprod{\tau_i}{\tau_j} s_i + s_j + \innerprod{x_j - x_i}{\tau_j} \\
    H_{\DT_\gamma} &= \mqty(1  & -\innerprod{\tau_i}{\tau_j} \\ -\innerprod{\tau_i}{\tau_j} & 1 ). \label{eq:PL-hessian}
\end{align}
 Since  $|\langle \tau_i,\tau_j \rangle| \leq 1$, the Hessian of the squared distance function is positive semi-definite everywhere;  sublevel sets of $\DT_\gamma (s_i,s_j)$ (\cref{eqn:sdf-quadratic} as a function extended to $\R^2$)  are convex subsets. Thus the restriction of a non-empty sublevel set of $\DT_\gamma$ to a closed, convex subset $\varphi_{ij}(V_{ij})$ is also closed and convex.

 Consider a non-empty intersection between $m$ cover elements. Recall the transition maps of the atlas are given by the $\Z$-action that yielded the covering $q: \stripinf \twoheadrightarrow \ExpS{2}{S^1}$. Then there are elements $k_2, \ldots, k_m \in \Z$, such that
 \begin{equation}
     \varphi_{i_1j_1} (V_{i_1j_1}^a \cap \cdots \cap V_{i_m j_m}^a)  =   \varphi_{i_1j_1}(V_{i_1j_1}^a) \cap \left(k_2 \cdot \varphi_{i_2j_2}(V_{i_2 j_2}^a) \right)  \cap \cdots \cap \left(k_m \cdot \varphi_{i_mj_m}(V_{i_m j_m}^a) \right).
\end{equation}
Because the $\Z$-action is an affine transformation in $\R^2$, each $\left(k_l \cdot \varphi_{i_lj_l}(V_{i_l j_l}^a) \right)$ is still a closed convex subset, and thus the intersection is still closed and convex. Since $\varphi_{i_1,j_1}$ is a homeomorphism, we conclude that any non-empty intersection of cover elements in $\cV^a$ is contractible, and thus $\cV^a$ is a good cover of $\AutoDsq^a$. 
 \end{proof}

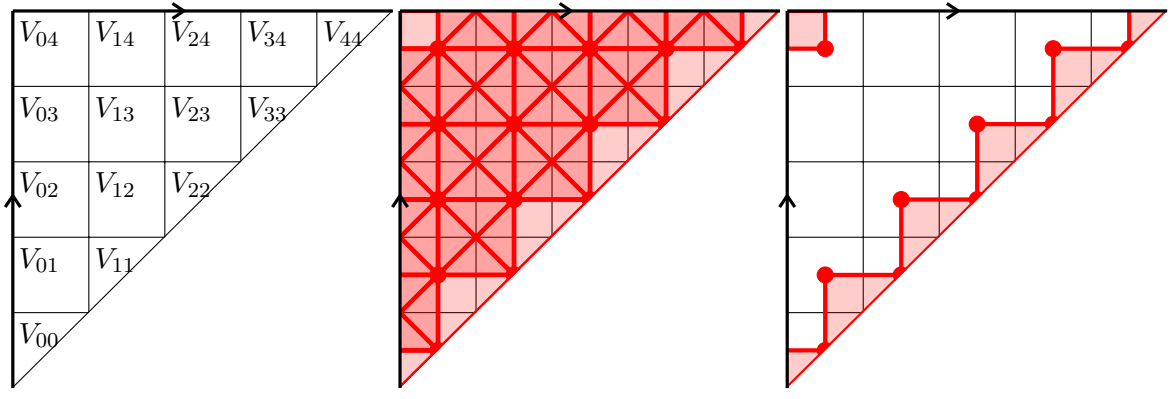
\begin{figure}[th]
    \centering
    \begin{minipage}{0.3\textwidth}
\begin{tikzpicture}[scale=1, cm={0,1,1,0,(0,0)}]
  \coordinate (A) at (0,0);
  \coordinate (B) at (5,0);
  \coordinate (C) at (5,5);
  \begin{scope}
  \clip (A) -- (B) -- (C) -- cycle;
  \draw[step=1.0] (0,0) grid (5,5);
  \draw[thick] (A) -- (B) -- (C) -- cycle;
  \end{scope}
    \foreach \x in {0, ..., 4} {
    \foreach \y in {0, ..., 4} {
        \begin{scope}[shift={(\x, \y)}]
        \ifnum\x<\y
        \else
        \node at (0.7, 0.35) {$V_{\y \x}$};
        \fi
        \end{scope}
    }
  }
  \draw [very thick] (0,0) -- (5,0);
  \draw [very thick] (2.4, 0.1) -- (2.55, 0) -- (2.4, -0.1);
  \draw [very thick] (5,0) -- (5,5);
  \draw [very thick] (5.1, 2.1) -- (5, 2.25) -- (4.9, 2.1);
\end{tikzpicture}
    \end{minipage}
    \begin{minipage}{0.3\textwidth}
\begin{tikzpicture}[scale=1, cm={0,1,1,0,(0,0)}]
  \coordinate (A) at (0,0);
  \coordinate (B) at (5,0);
  \coordinate (C) at (5,5);
  \begin{scope}
  \clip (A) -- (B) -- (C) -- cycle;
  \draw[step=1.0] (0,0) grid (5,5);
  \foreach \x/\y in {
  -1/-1, 0/-1, 1/-1, 2/-1, 3/-1,
  0/0, 1/0, 2/0, 3/0, 4/0,
  1/1, 2/1, 3/1, 4/1,
  2/2, 3/2, 4/2,
  3/3, 4/3,
  4/4
  } {
        \begin{scope}[shift={(\x, \y)}]
        \filldraw [color=red, fill=red, fill opacity=0.2, ultra thick] (0.5, 0.5) -- (1.5, 0.5) -- (1.5, 1.5) -- cycle;
        \filldraw [color=red, fill=red, fill opacity=0.2, ultra thick] (0.5, 0.5) -- (0.5, 1.5) -- (1.5, 1.5) -- cycle;
        \ifnum\x>\y
        \filldraw [color=red, fill=red, fill opacity=0.2, ultra thick] (0.5, 0.5) -- (1.5, 0.5) -- (0.5, 1.5) -- cycle;
        \filldraw [color=red, fill=red, fill opacity=0.2, ultra thick] (1.5, 0.5) -- (1.5, 1.5) -- (0.5, 1.5) -- cycle;
        \fi
        \filldraw [color=red, fill=red] (0.5, 0.5) circle (3pt);
        \end{scope}
  }

\filldraw [color=red, fill=red, fill opacity=0.2, ultra thick] (4.5, 0.0) -- (5.0, 0.0) -- (5.0, 0.5) -- (4.5, 0.5) -- cycle;
  \end{scope}
  \draw [very thick] (0,0) -- (5,0);
  \draw [very thick] (2.4, 0.1) -- (2.55, 0) -- (2.4, -0.1);
  \draw [very thick] (5,0) -- (5,5);
  \draw [very thick] (5.1, 2.1) -- (5, 2.25) -- (4.9, 2.1);
\end{tikzpicture}
    \end{minipage}
    \begin{minipage}{0.3\textwidth}
\begin{tikzpicture}[scale=1, cm={0,1,1,0,(0,0)}]
  \coordinate (A) at (0,0);
  \coordinate (B) at (5,0);
  \coordinate (C) at (5,5);
  \begin{scope}
  \clip (A) -- (B) -- (C) -- cycle;
  \draw[step=1.0] (0,0) grid (5,5);
  \foreach \x/\y in {
  -1/-1, 2/-1,
  0/0, 1/0, 4/0,
  1/1, 2/1,
  2/2, 3/2,
  3/3, 4/3,
  4/4
  } {
        \begin{scope}[shift={(\x, \y)}]
        \ifnum\x=\y
        \filldraw [color=red, fill=red, fill opacity=0.2, ultra thick] (0.5, 0.5) -- (1.5, 0.5) -- (1.5, 1.5) -- cycle;
        \fi
        \filldraw [red] (0.5, 0.5) circle (3pt);
        \end{scope}
  }

\filldraw [color=red, fill=red, fill opacity=0.2, ultra thick] (4.5, 0.0) -- (5.0, 0.0) -- (5.0, 0.5) -- (4.5, 0.5) -- cycle;
  \end{scope}
  \draw [very thick] (0,0) -- (5,0);
  \draw [very thick] (2.4, 0.1) -- (2.55, 0) -- (2.4, -0.1);
  \draw [very thick] (5,0) -- (5,5);
  \draw [very thick] (5.1, 2.1) -- (5, 2.25) -- (4.9, 2.1);
\end{tikzpicture}
    \end{minipage}
    \caption{Left: A closed cover $\cV$ of the M\"obius band $\ExpS{2}{S^1}$. Centre: its nerve $\nerve{\cV}$. Right: the nerve of the level set $\DT_\gamma = 0$ . given by $\nerve{\cV^0}$.}
    \label{fig:PL-nerve}
\end{figure}

\paragraph{Homotopy type of sublevel sets from Nerve}
The description of sublevel sets $\AutoDsq^{a}$ with a cover $\cV^a$ satisfying the properties described~\Cref{lem:cover_PL} allows us to apply the \emph{unified nerve theorem}~\cite[Theorem 5.9]{BAUER2023125503} to deduce the homotopy type of sublevel sets with a discrete construct called the nerve. Recall the nerve $\nerve{\cV}$ of $\cV$ is the simplicial complex whose $k$-simplices are non-empty $k$-fold intersections of cover elements.  We let $|\nerve{\cV}|$ denote the geometric realisation of the nerve. 

We now state a specialised version of the unified nerve theorem~\cite[Theorem 5.9]{BAUER2023125503}, derived from the its simplification in~\cite[Table 1]{BAUER2023125503}. We first recall some technical language in order to state the theorem. For a cover $\cV = \{V_i\}_{i \in \cI}$  of a space $X$, let us denote $V_J = \cap_{j \in J} V_j$. For $V_J$, its \emph{latching space} $L(V_J) \subset V_J$ is the union $\cup_{I \supsetneq J} V_I$. We say the cover $\cV$ satisfies the \emph{latching condition}, if for any $J \subset \cI$, the pair $(V_J, L(V_J))$ satisfies the homotopy extension property (see for example~\cite[\S 0]{hatcher}). 

\begin{theorem}[{\cite[Theorem 5.9]{BAUER2023125503}}]\label{thm:bauer}
    Let $X$ be a locally compact Hausdorff space. Suppose $\cV = \{V_i\}_{i \in \cI}$ is a finite and good cover of $X$ by closed subsets that satisfies the latching condition. Then there is a functorial homotopy equivalence $X \simeq \qty|\nerve{\cV}|$. 
\end{theorem}
The properties we have shown about $\cV^a$ in~\Cref{lem:cover_PL} are sufficient to satisfy the conditions of the nerve theorem, as closed, convex covers in Euclidean space satisfy the latching condition~\cite[Prop 5.15\&18]{BAUER2023125503}. 

\begin{proposition} \label{prop:nerve_to_sublevel}
    There is a functorial homotopy equivalence $\AutoDsq^{a} \simeq \qty|\nerve{\cV^a}|$.
\end{proposition}
\begin{proof}
    Recall from~\Cref{lem:cover_PL} that $\cV^a$ is a closed, good cover. By construction it is also finite. 
    Recall the latching space condition. Since we have a chart $\varphi_{ij}$ that takes  $V^a$ and its intersections with other cover elements to a homeomorphic closed, convex subset, the chart takes the pair $(V_J^a, L(V_J^a))$ to a homeomorphic pair generated by a cover of closed convex subset in $\R^2$, we can apply~\cite[Prop 5.15\&18]{BAUER2023125503} and deduce that $\cV^a$ satisfies the latching condition.   Thus \Cref{thm:bauer} implies there is a functorial homotopy equivalence $\AutoDsq^{a} \simeq \qty|\nerve{\cV^a}|$.
\end{proof}
Consequently, the persistence homology of the sublevel set filtration of $\ExpS{2}{S^1}$ by $\DT_\gamma$ can be recovered from a filtration of the nerve. For $a \leq b$, the inclusion of cover elements of $\cV^a$ into $\cV^b$ induces a filtration of simplicial complexes
\begin{equation}
    \nerve{\cV^\bullet}:\quad \nerve{\cV^0} \subseteq \cdots \subseteq  \nerve{\cV^a}  \subseteq \nerve{\cV^b} \subseteq \cdots \subseteq  \nerve{\cV^\infty} := \nerve{\cV}.
\end{equation}
If we take the geometric realisation of the nerves, we get a similar filtration of topological spaces. Because of the functorial equivalence of simplicial and singular homology, the persistent simplicial homology of the filtration above is equivalent to the persistent singular homology of the geometric realisation of the filtration nerve $\pershomf_i(\nerve{\cV^\bullet}) \cong \pershomf_i(|\nerve{\cV^\bullet}|)$. Since we have a functorial homotopy equivalence  $\AutoDsq^{a} \simeq \qty|\nerve{\cV^a}|$ by~\Cref{prop:nerve_to_sublevel}, we conclude that the persistence homology of the nerve filtration is isomorphic to the persistence homology arising from the sublevel set filtration of $\ExpS{2}{S^1}$ by $\DT_\gamma$:
\begin{equation} \label{eq:nerve_PH_equiv}
    \pershomf_i(\nerve{\cV^\bullet}) \cong \pershomf_i(\AutoDsq^{\bullet}) =: \pershomf_i(\DT_\gamma). 
\end{equation}
Since the homology of a finite simplicial complex $\nerve{\cV^a}$ is finite dimensional,~\Cref{prop:nerve_to_sublevel} directly implies that $\pershomf_i(\DT_\gamma)$ is tame.
\begin{corollary} \label{cor:PL_tame}
If $\gamma$ is a PL embedding of $S^1$ (i.e. it satisfies conditions~\labelcref{C1,C2}), then $\pershomf_i(\DT_\gamma)$ is tame. 
\end{corollary}

\subsection{The Filtration of the Nerve}
\label{ssec:filtration-of-the-nerve}
In the preceding section, we have show that the persistent homology of the filtration of $\ExpS{2}{S^1}$ by $\DT_\gamma$ is isomorphic to that of a filtration of $\nerve{\cV^\bullet}$. To give an account of $\nerve{\cV^\bullet}$, we first describe the full nerve $\nerve{\cV}$ explicitly. We then show how the geometry of the PL loop prescribes the times at which each simplex in the nerve enters the filtration. Finally, we show how homological critical Morse sets in the nerve filtration correspond to `critical points' in the underlying sublevel set filtration $\AutoDsq^\bullet$. 

\subsubsection{The Nerve}
\label{ssec:what-a-nerve}
 We now describe the \emph{nerve} of $\cV$ by considering the intersection between such squares and triangles (see~\Cref{fig:PL-nerve}). Recall the simplices of the nerve are generated by finite elements of cover elements. Let us first consider pairwise intersections of cover elements. The cover elements $V_{ij}$ and $V_{i'j'}$ have non-empty intersection only if at least one of the indices from each pair are the same or adjacent in $\Z_n$. Excluding the trivial case $\qty{ij} = \qty{i',j'}$, there are two cases of non-empty intersection to consider: 
\begin{enumerate}[ref={(E\arabic*)}, label=(E\arabic*)]
    \item \label{item:PLedge1} $j = j'$, and $i,i'$ are distinct and adjacent in $\Z_n$; or 
    \item \label{item:PLedge2}  $j,j'$ are distinct and adjacent in $\Z_n$, and $i,i'$ are also distinct and adjacent in $\Z_n$.
\end{enumerate}
We recall the quotient map $q: \stripinf \twoheadrightarrow \ExpS{2}{S^1}$, and describe these intersections as homeomorphic images of subsets of $\stripinf$. In the case described by~\labelcref{item:PLedge1}, where one index element is shared, the intersections are along closed intervals:
\begin{equation} \label{eq:intersection_2_edge}
    V_{ij} \cap V_{i+1,j} = q(\{L_i\} \times [L_j, L_{j+1}])
\end{equation}
Note that addition of indices is modulo $n$ here. We also remark that no intersection of cover elements other than  $V_{ij}$ and $V_{i+1,j}$ is equal to $q_2(\{L_i\} \times [L_j, L_{j+1}])$. In the second case described by~\labelcref{item:PLedge2}, the two cover elements intersect at one of their corners:
\begin{equation} \label{eq:intersection_2_corner}
    V_{ij} \cap V_{i+1,j+1} = q(L_{i+1}, L_{j+1}).
\end{equation}
The characterisation of two-way intersections by~\labelcref{item:PLedge1,item:PLedge2} and \cref{eq:intersection_2_edge,eq:intersection_2_corner} constrain the intersections between three and four elements to only take place at points $q(L_i, L_j)$ (see~\Cref{fig:PL-nerve}).  If  $i \neq j$, then $q(L_i, L_j)$ lies in the interior $\UConf{2}{S^1}$. Otherwise, $q(L_i, L_i)$ is a boundary point. For an interior point $i\neq j$, the point $q(L_{i+1}, L_{j+1}) \in \UConf{2}{S^1}$ is also the four way intersection
\begin{equation}  \label{eq:intersection_4}
   V_{ij} \cap V_{i+1,j} \cap V_{ij+1}  \cap V_{i+1,j+1} = q(L_{i+1}, L_{j+1}),
\end{equation}
along with three way intersections between cover elements indexed by three-subsets of $\{ \{i,j\},  \{i+1,j\} , \{i,j+1\} , \{i+1,j+1\} \}$. Otherwise, if $i=j$, then $q(L_{i+1}, L_{i+1}) $ is in a unique \emph{three way} intersection
\begin{equation} \label{eq:intersection_3_boundary}
     V_{ii} \cap V_{i,i+1} \cap V_{i+1, i+1} = q(L_{i+1}, L_{i+1}). 
\end{equation}
Having fully described the possible modes of intersections between cover elements of $\cV$, we can constructively the nerve $\nerve{\cV}$:
\begin{enumerate}
    \item We have vertices, corresponding to single cover elements of $\cV$, indexed over $(i,j)$ with $0 \leq i \leq j \leq n-1$;
    \item We have an edge between $(i,j)$ and $(i',j')$ corresponding to intersections between pairs of cover elements; such edges appear in the complex if \labelcref{item:PLedge1,item:PLedge2} are satisfied;
    \item We have a two simplex between $(i,i), (i,i+1), (i+1, i+1)$ for $i = 0,\ldots, n-1$, corresponding the intersections of cover elements at the boundary of $\ExpS{2}{S^1}$ (see right hand side panel of~\Cref{fig:PL-nerve}). 
    \item For $i \neq j$, we add a three simplex between $(i,j), (i+1,j), (i,j+1), (i+1, j+1)$, and all its faces. 
\end{enumerate} 

\begin{remark}
    We can make a further simplification to the nerve for every three simplex $(i,j), (i+1,j), (i,j+1), (i+1, j+1)$ where $i \neq j$. We note that the anti-diagonal edge $(i,j+1)(i+1,j)$ is a \emph{free faces} of  $(i,j), (i+1,j), (i,j+1), (i+1, j+1)$ in $\nerve{\cV}$. Thus, we can perform a simplicial collapse to remove simplices that are simultaneously co-faces of $(i,j+1)(i+1,j)$  and faces of $(i,j), (i+1,j), (i,j+1), (i+1, j+1)$, without changing the homotopy type of $\nerve{\cV}$. What remains of the faces of $(i,j), (i+1,j), (i,j+1), (i+1, j+1)$ are a pair of two-simplicies $(i,j)(i,j+1)(i+1,j+1)$ and  $(i,j)(i+1,j)(i+1,j+1)$ and all their faces. 
    Thus, instead of adding a three simplex between $(i,j), (i+1,j), (i,j+1), (i+1, j+1)$, and all its faces in step (4) above, we can construct a smaller simplicial complex by only adding a pair of two-simplicies $(i,j)(i,j+1)(i+1,j+1)$ and  $(i,j)(i+1,j)(i+1,j+1)$ for $i \neq j$, as the faces of such simplices are added already in steps (1)-(3). 
\end{remark}
\subsubsection{The Filtration}
We now give an explicit description of the filtered nerve $\qty(\nerve{\cV^a})_{a\in\R}$. For a simplex  $J = (i_0,j_0),\ldots, (i_m, j_m)$ in the nerve $\nerve{\cV}$, consider $F: \nerve{\cV} \to \R$ given by 
\begin{equation}
    F(J) = \inf \{a \in \R \ : \ J \in \nerve{\cV^a} \}
\end{equation}
Recall any simplex $J \in \nerve{\cV^a}$ corresponds to a non-empty intersection of cover elements $V_J^a$. Recall by definition that $V_J^a$ is a compact set generated by 
\begin{equation}
    V_J^a = V_J \cap \AutoDsq^a.
\end{equation}
By compactness, $\DT_\gamma$ attains a minimum on $V_J$. Thus, a simplex $J$ is present in $\nerve{\cV^a}$, if and only if $a \geq \min_{V_J} \DT_\gamma$. In other words, 
\begin{equation}
   F(J) =  \min_{V_J} \DT_\gamma. 
\end{equation}
Note that $F$ is monotone with respect to the face poset of the nerve: for $I \subset J$ a face of $J$, because $V_J \subseteq V_I$, it follows from the definition of $F$ as a minimum that $F(I) \leq F(J)$. In particular, this implies $F$ is a filter function of the filtration $\qty(\nerve{\cV^a})_{a\in\R}$, and the sublevel set $\sublevel{F}{a}$ of simplices gives the nerve $\nerve{\cV^a}$.

Having characterised the intersections of cover elements in~\cref{eq:intersection_2_corner,eq:intersection_2_edge,eq:intersection_4,eq:intersection_3_boundary}, we can evaluate the minima  $\min_{V_J} \DT_\gamma$ on each simplex and derive $F$. We give a description of these cases below.
\paragraph{0-simplices}
For a 0-simplex $J = (i,j)$, the filtration value $F((i,j))$ can simply be inferred from the minimal distance between two line segments $\overline{x_ix_{i+1}}$ and $\overline{x_jx_{j+1}}$:
\begin{equation}
    F((i,j)) = \frac{1}{2}d(\overline{x_ix_{i+1}}, \overline{x_jx_{j+1}})^2.
\end{equation}
We remark that if $V_{ij}$ intersects the boundary, i.e. $j = i$ or $j = i+1$, then the two line segments overlap at a point, and $F((i,j)) = 0$.
\paragraph{1-simplices}
We recall from~\cref{eq:intersection_2_edge,eq:intersection_2_corner} that the intersection between two cover adjacent cover elements can either happen along an edge or a corner. For the case of the edge in~\cref{eq:intersection_2_edge}, the minimum of $\DT_\gamma$ can be inferred from the distance between $x_i$ and the line segment $\overline{x_j,x_{j+1}}$. Otherwise if the cover elements intersect at a corner, then we can simply infer it with point wise distances:
\begin{align}
    \text{\labelcref{item:PLedge1}}&: \quad F((i,j)(i+1,j)) = \frac{1}{2} d(x_i, \overline{x_j x_{j+1}})^2 \\
    \text{\labelcref{item:PLedge2}}&: \quad F((i,j)(i+1,j+1)) = \frac{1}{2} \norm{x_{i+1} - x_{j+1}}^2.
\end{align}
\paragraph{2-simplices and 3-simplices}
For $J$ a 2-simplex or 3-simplex, the intersection $V_J$ is a single point~\cref{eq:intersection_4,eq:intersection_3_boundary}. For the interior case $i \neq j$ described in~\cref{eq:intersection_4}, a 3-simplex $J = (i,j), (i+1,j), (i,j+1), (i+1, j+1)$ and all 2-simplices in its faces enters the filtration at value 
\begin{equation}
    F((i,j), (i+1,j), (i,j+1), (i+1, j+1)) = \frac{1}{2}\norm{x_{i+1} - x_{j+1}}^2. 
\end{equation}
In the boundary where $i = j$ as described in~\cref{eq:intersection_3_boundary}, since $q(L_{i+1}, L_{i+1})$ is on the boundary, 
\begin{equation}
    F((i,i), (i,i+1), (i+1, i+1)) = 0.
\end{equation}
This thus accounts for all simplices in $\nerve{\cV}$ and we have defined the filtration $\nerve{\cV^\bullet}$ as a sublevel set filtration of the filter function $F$.

\subsection{Persistent Homology}
\label{ssec:pl-persistence}

We turn our attention then to the persistent homology of the sublevel set filtration of $\ExpS{2}{S^1}$ by $\DT_\gamma$, which is isomorphic to the filtration of the nerve $\pershomf_i(\nerve{\cV^\bullet}) \cong \pershomf_i(\Gamma^\bullet)$ (\cref{eq:nerve_PH_equiv}). We are interested in homological critical values of the nerve filtration $\nerve{\cV^\bullet}$. Since we have a finite simplicial complex, there are only finitely many values in the filtration $\nerve{\cV^\bullet}$ at which the simplicial complex actually changes by addition of simplices. Let us denote those distinct values by $0 = a_0 < a_1 < \ldots < a_m = \max F$,
which is the ordered set of values $a \in \R$ where $\fibre{F}{a} \neq \emptyset$. For all values $s,t$ such that  $a_i \leq s \leq t < a_{i+1}$, the nerve is constant, and thus $t \in (a_i, a_{i+1})$ are trivially homological regular values. We can thus restrict our attention to $a_0, \ldots, a_m$.

In analogy with the smooth case where Morse theory relates how critical points of smooth functions changes the homotopy type of sublevel sets, we now discuss how the addition of individual simplices in the filtration $\nerve{\cV^\bullet}$ changes the homology of the sublevel set. We recall the from~\Cref{ssec:morse_sets} how we can characterise them with the Conley indices of Morse sets. In the following, we shall say a Morse set $S_i$ of a simplicial complex is \emph{generated by a $k$-simplex} if $\closure{S_i}$ is a the closure of a $k$-simplex. We shall also define the subcomplex $S_{\partial \Mob}$ as the union of all (closures of) 2-simplices that are not faces of a 3-simplex, i.e. those lying along the boundary $\partial \Mob$.

\begin{proposition}
    \label{prop:morse-set-pl}
    Consider a PL embedded loop in $\R^d$
    satisfying~\labelcref{C1,C2,C3},
    with corresponding filtration $\qty(\nerve{\cV^a})_{a \in \R}$.
    For values \(a < 0\), 
    no Morse sets exist.
    For \(a = 0\), the only Morse set is $S_{\partial\Mob}$.
    For values \(a > 0\), all  Morse sets are generated by 0, 1, or 3-simplices.
\end{proposition}

For the case \(a >0\) we shall split this into cases by the dimension of the simplex being added.
In the following, we shall say a vertex of $\nerve{\cU}$ is an \emph{interior} vertex if its corresponding cover element
does not intersect the boundary $\partial \Mob$. 

The first lemma shows that Morse sets that are top-dimension 1 (only edges and vertices) consist of no more than one edge. 

\begin{lemma}
    If a pair of  edges $e,e'$ that are incident on the same interior vertex $v$ are added to the filtration at the same time, then either
    \begin{itemize}
        \item $v$ is in the link of both edges; or
        \item edges $e,e'$ are faces of a 3-simplex, and $v,e,e'$ are added along with the unique three simplex containing $e,e'$ in its faces.
    \end{itemize}
\end{lemma}
\begin{proof}
    Suppose $v$ is not in the link of $e = (v,w)$ and $e' = (v,w')$. Then $e,e'$ being added to the filtration at the same time implies the unique minimum (as $v$ in interior, and~\labelcref{C3} condition) of the convex function on the cover element associated to $v$ is contained in the intersection between the squares belonging to $v,w,w'$. Since the intersection is a single point corresponding to the intersection between cover elements corresponding $v,w,w'$ and some other vertex $w''$, the 3-simplex containing $e,e'$ in its faces is added at the same time. 
\end{proof}

The lemma above shows that if an edge is added at a step in the filtration, along with other simplices, then either it is in the face of a 3-simplex being added; or other simplices added at the same time (apart from vertices) are not connected to the edge.
For the 2-simplex case,
we observe that by the structure of the cover, 2-simplices are always added as a face of a 3-simplex.
The next lemma now deals with the 3-simplex case.

\begin{lemma}
    If a pair of interior 3-simplices $\sigma,\sigma'$ incident on the same edge $e$ are added at the same time, then the edge is in the link of both $\sigma$ and $\sigma'$. 
\end{lemma}
\begin{proof}
    Consider $\sigma = uvab$ and $\sigma' = uva'b'$. The cover elements $u$ and $v$ intersect at a closed line segment, with end points being the quadruple intersections between cover elements $u,v,w,z$  and $u,v,w',z'$ respectively. If the end points have the same function value $a$, then by the convexity of the distance function on individual cover elements, the function along the relative interior of the segment is either equal to $a$, or strictly less than $a$. Since the distance function is positive definite quadratic when restricted to each cover element (due to the~\labelcref{C3} condition), the distance function cannot be constant along line segments, and thus the relative interior of the segment must take value strictly less than $a$. Since the segment is the intersection between cover elements indexed by $u$ and $v$, the edge $uv$ in the nerve must then be added at an earlier point in the filtration than $a$. Thus the edge $uv$ is in the link of both $\sigma$ and $\sigma'$. 
\end{proof}

\begin{proof}[Proof of Proposition~\ref{prop:morse-set-pl}]
For \(a < 0\), as the filtration is non-negative, $\nerve{\cV^a}$ is empty.
For \(a = 0\), as shown in Figure~\ref{fig:PL-nerve},
the only component $S$ is $S_{\partial \Mob}$.
For \(a > 0\)
The lemmas above
constrain the Morse sets of the filtration. The two observations,
imply that
when a simplex is added to the filtration
not as a face
when an edge or a 3-simplex is added to the filtration, the simplex is disconnected with other simplices added at the same time, other than those in its faces. Thus, the Morse sets $S_i$ at each step of the filtration (connected components of simplices added) are either a   single vertex, edge, or 3-simplex, along with a subset of simplices in their faces.
\end{proof}

For $S_{\partial \Mob}$, its link is empty and $S_{\partial \Mob}$ is homotopy equivalent to $S^1$,
hence its Conley index is isomorphic to $H_\bullet(S^1)$. Note that this is not reduced homology, so $S_{\partial \Mob}$ is not a $1$-saddle.

If the Morse set enters the filtration at $a > 0$, since the closure of each Morse set is a closed simplex, it is contractible. Thus we can infer its Conley indices via the links of the Morse set. 
Let $\sigma \in \nerve{\cV^a}$.
We describe the possible 
the possible subcomplexes $K \leq \closure{\sigma}$
that are realisable as $\closure{\sigma}\cap\nerve{\cV^{a-\varepsilon}}$
for some $\varepsilon$ such that no simplices
are added in $(a - \varepsilon, a)$.
For vertices and edges, all possible subcomplexes are realisable.
Let $(i, j), (i +1, j), (i, j+1)$ and $(j +1, j+1)$ be the indices of the cover
giving rise to a 3-simplex in the nerve.
The possible subcomplexes (apart from $\closure{\sigma}$) we need to consider are the subgraphs of the complete graph on these vertices.
Subgraphs containing isolated vertices
are not realisable:
a vertex $v$
being isolated implies
that the minimum of the distance function
restricted to $V_v$ is in the interior of $V_v$. 
As the sublevel set restricted to $U^v$ is a non-degenerate
ellipse (intersected with $V_v$) by condition~\ref{C3},
these can not intersect with the corner
$V_{ij} \cap V_{i + 1, j} \cap V_{i j+1} \cap V_{i+1, j+1}$
without first intersecting one of the line segments ($V_{ij} \cap V_{i+1, j}$ etc.),
thus the 3-simplex cannot be added immediately after a vertex.
Any subgraph containing a diagonal edge
$(i, j), (i+1, j+1)$ or $(i +1,j), (i, j+1)$ 
is also not realisable as the existence of this edge implies the existence of the 3-simplex.
The remaining subgraphs are realisable.
This is tabulated in Table~\ref{tab:pwlinear-critical-points} where we observe these Morse sets are $n$-saddles.

\begin{table}
    \centering
\begin{tabular}{c c c c c c c}
    & dim & $\closure{\sigma}$ & $\closure{\sigma}\cap \AutoDsq^a$ & tangent angles & $\beta_\bullet(\closure{\sigma}, \closure{\sigma}\cap \AutoDsq^a$ & index \\ 
    \hline
    \noalign{\smallskip}
    1&0&
    \begin{tikzpicture}
        \draw (0, 0) rectangle (1, 1);
        \filldraw [red] (0.5, 0.5) circle (2pt);
    \end{tikzpicture}
    &
    \begin{tikzpicture}
        \draw (0, 0) rectangle (1, 1);
    \end{tikzpicture}
    &
    \begin{tikzpicture}
        \draw [thick] (0, 0) -- (0, 1);
        \draw [thick] (1, 0) -- (1, 1); 
        \draw [magenta, thick, dashed] (0, 0.5) -- (1, 0.5);
    \end{tikzpicture}
    &
    1, 0, 0
    &
    0
    \\
    2&&
    \begin{tikzpicture}
        \draw (0, 0) rectangle (1, 1);
        \filldraw [red] (0.5, 0.5) circle (2pt);
    \end{tikzpicture}
    &
    \begin{tikzpicture}
        \draw (0, 0) rectangle (1, 1);
        \filldraw [red] (0.5, 0.5) circle (2pt);
    \end{tikzpicture}
    &
    &
    0, 0, 0
    &
    --
    \\
    \hline
    \noalign{\smallskip}
    3&1 
    &
    \begin{tikzpicture}
        \draw (0, 0) rectangle (2, 1);
        \draw (1, 0) -- (1, 1);
        \filldraw [red] (0.5, 0.5) circle (2pt);
        \filldraw [red] (1.5, 0.5) circle (2pt);
        \draw [red, thick] (0.5, 0.5) -- (1.5, 0.5);
    \end{tikzpicture}
    &
    \begin{tikzpicture}
        \draw (0, 0) rectangle (2, 1);
        \draw (1, 0) -- (1, 1);
    \end{tikzpicture}
    &
    \begin{tikzpicture}
        \draw [thick] (0, 0) -- (0, 1);
        \filldraw [black] (1, 0.5) circle (2pt);
        \draw [thick] (1, 0.5) -- (2, 1);
        \draw [thick] (1, 0.5) -- (2, 0);
        \draw [magenta, thick, dashed] (0, 0.5) -- (1, 0.5);
    \end{tikzpicture}
    &
    1, 0, 0
    &
    0
    \\
    4&&
    &
    \begin{tikzpicture}
        \draw (0, 0) rectangle (2, 1);
        \draw (1, 0) -- (1, 1);
        \filldraw [red] (0.5, 0.5) circle (2pt);
    \end{tikzpicture}
    &
    &
    0, 0, 0
    &
    --
    \\
    5&&
    &
    \begin{tikzpicture}
        \draw (0, 0) rectangle (2, 1);
        \draw (1, 0) -- (1, 1);
        \filldraw [red] (0.5, 0.5) circle (2pt);
        \filldraw [red] (1.5, 0.5) circle (2pt);
    \end{tikzpicture}
    &
    \begin{tikzpicture}
        \draw [thick] (0, 0) -- (0, 1);
        \filldraw [black] (2, 0.5) circle (2pt);
        \draw [thick] (2, 0.5) -- (1, 1);
        \draw [thick] (2, 0.5) -- (1, 0);
        \draw [magenta, thick, dashed] (0, 0.5) -- (2, 0.5);
    \end{tikzpicture}
    &
    0, 1, 0
    &
    1
    \\
    6&&
    &
    \begin{tikzpicture}
        \draw (0, 0) rectangle (2, 1);
        \draw (1, 0) -- (1, 1);
        \filldraw [red] (0.5, 0.5) circle (2pt);
        \filldraw [red] (1.5, 0.5) circle (2pt);
        \draw [red, thick] (0.5, 0.5) -- (1.5, 0.5);
    \end{tikzpicture}
    &
    &
    0, 0, 0
    &
    --
    \\
    \hline\noalign{\smallskip}
    7&3
    &
    \begin{tikzpicture}
        \draw (0, 0) rectangle (2,2);
        \draw (1, 0) -- (1, 2);
        \draw (0, 1) -- (2, 1);
        \filldraw [red] (0.5, 0.5) circle (2pt);
        \filldraw [red] (1.5, 0.5) circle (2pt);
        \filldraw [red] (0.5, 1.5) circle (2pt);
        \filldraw [red] (1.5, 1.5) circle (2pt);
        \filldraw [color=red, fill=red, fill opacity=0.2, thick] (0.5, 0.5) -- (1.5, 0.5) -- (0.5, 1.5) -- cycle;
        \filldraw [color=red, fill=red, fill opacity=0.2, thick] (0.5, 0.5) -- (1.5, 0.5) -- (1.5, 1.5) -- cycle;
        \filldraw [color=red, fill=red, fill opacity=0.2, thick] (0.5, 0.5) -- (0.5, 1.5) -- (1.5, 1.5) -- cycle;
        \filldraw [color=red, fill=red, fill opacity=0.2, thick] (1.5, 0.5) -- (1.5, 1.5) -- (0.5, 1.5) -- cycle;
    \end{tikzpicture}
    &
    \begin{tikzpicture}
        \draw (0, 0) rectangle (2,2);
        \draw (1, 0) -- (1, 2);
        \draw (0, 1) -- (2, 1);
    \end{tikzpicture}
    &
    \begin{tikzpicture}
        \filldraw [black] (1, 0.5) circle (2pt);
        \draw [thick] (0, 1) -- (1, 0.5);
        \draw [thick] (0, 0) -- (1, 0.5);
        \filldraw [black] (2, 0.5) circle (2pt);
        \draw [thick] (2, 0.5) -- (3, 1);
        \draw [thick] (2, 0.5) -- (3, 0);
        \draw [magenta, thick, dashed] (1, 0.5) -- (2, 0.5);
    \end{tikzpicture}
    &
    1, 0, 0
    &
    0
    \\
    8&&
    &
    \begin{tikzpicture}
        \draw (0, 0) rectangle (2,2);
        \draw (1, 0) -- (1, 2);
        \draw (0, 1) -- (2, 1);
        \filldraw [red] (0.5, 1.5) circle (2pt);
        \filldraw [red] (1.5, 1.5) circle (2pt);
        \draw [red, thick] (0.5, 1.5) -- (1.5, 1.5);
    \end{tikzpicture}
    &
    &
    0, 0, 0
    &
    --
    \\
    9&&
    &
    \begin{tikzpicture}
        \draw (0, 0) rectangle (2,2);
        \draw (1, 0) -- (1, 2);
        \draw (0, 1) -- (2, 1);
        \filldraw [red] (0.5, 1.5) circle (2pt);
        \filldraw [red] (1.5, 1.5) circle (2pt);
        \filldraw [red] (0.5, 0.5) circle (2pt);
        \filldraw [red] (1.5, 0.5) circle (2pt);
        \draw [red, thick] (0.5, 0.5) -- (1.5, 0.5);
        \draw [red, thick] (0.5, 1.5) -- (1.5, 1.5);
    \end{tikzpicture}
    &
    \begin{tikzpicture}
        \filldraw [black] (1, 0.5) circle (2pt);
        \draw [thick] (0, 1) -- (1, 0.5);
        \draw [thick] (0, 0) -- (1, 0.5);
        \filldraw [black] (3, 0.5) circle (2pt);
        \draw [thick] (2, 0) -- (3, 0.5);
        \draw [thick] (2, 1) -- (3, 0.5);
        \draw [magenta, thick, dashed] (1, 0.5) -- (3, 0.5);
    \end{tikzpicture}
    &
    0, 1, 0
    &
    1
    \\
    10&&
    &
    \begin{tikzpicture}
        \draw (0, 0) rectangle (2,2);
        \draw (1, 0) -- (1, 2);
        \draw (0, 1) -- (2, 1);
        \filldraw [red] (0.5, 1.5) circle (2pt);
        \filldraw [red] (1.5, 1.5) circle (2pt);
        \filldraw [red] (0.5, 0.5) circle (2pt);
        \draw [red, thick] (0.5, 1.5) -- (1.5, 1.5);
        \draw [red, thick] (0.5, 0.5) -- (0.5, 1.5);
    \end{tikzpicture}
    &
    &
    0, 0, 0
    &
    -- \\
    11&&
    &
    \begin{tikzpicture}
        \draw (0, 0) rectangle (2,2);
        \draw (1, 0) -- (1, 2);
        \draw (0, 1) -- (2, 1);
        \filldraw [red] (0.5, 1.5) circle (2pt);
        \filldraw [red] (1.5, 1.5) circle (2pt);
        \filldraw [red] (0.5, 0.5) circle (2pt);
        \filldraw [red] (1.5, 0.5) circle (2pt);
        \draw [red, thick] (0.5, 1.5) -- (1.5, 1.5);
        \draw [red, thick] (0.5, 0.5) -- (0.5, 1.5);
        \draw [red, thick] (0.5, 0.5) -- (1.5, 0.5);
    \end{tikzpicture}
    &
    &
    0, 0, 0
    &
    --
    \\
    12&&
    &
    \begin{tikzpicture}
        \draw (0, 0) rectangle (2,2);
        \draw (1, 0) -- (1, 2);
        \draw (0, 1) -- (2, 1);
        \filldraw [red] (0.5, 1.5) circle (2pt);
        \filldraw [red] (1.5, 1.5) circle (2pt);
        \filldraw [red] (0.5, 0.5) circle (2pt);
        \filldraw [red] (1.5, 0.5) circle (2pt);
        \draw [red, thick] (0.5, 1.5) -- (1.5, 1.5);
        \draw [red, thick] (0.5, 0.5) -- (0.5, 1.5);
        \draw [red, thick] (0.5, 0.5) -- (1.5, 0.5);
        \draw [red, thick] (1.5, 0.5) -- (1.5, 1.5);
    \end{tikzpicture}
    &
    \begin{tikzpicture}
        \filldraw [black] (0, 0.5) circle (2pt);
        \draw [thick] (0, 0.5) -- (1, 0);
        \draw [thick] (0, 0.5) -- (1, 1);
        \filldraw [black] (3, 0.5) circle (2pt);
        \draw [thick] (2, 0) -- (3, 0.5);
        \draw [thick] (2, 1) -- (3, 0.5);
        \draw [magenta, thick, dashed] (0, 0.5) -- (3, 0.5);
    \end{tikzpicture}
    &
    0, 0, 1
    &
    2
    \\
    13&&
    &
    \begin{tikzpicture}
        \draw (0, 0) rectangle (2,2);
        \draw (1, 0) -- (1, 2);
        \draw (0, 1) -- (2, 1);
        \filldraw [red] (0.5, 0.5) circle (2pt);
        \filldraw [red] (1.5, 0.5) circle (2pt);
        \filldraw [red] (0.5, 1.5) circle (2pt);
        \filldraw [red] (1.5, 1.5) circle (2pt);
        \filldraw [color=red, fill=red, fill opacity=0.2, thick] (0.5, 0.5) -- (1.5, 0.5) -- (0.5, 1.5) -- cycle;
        \filldraw [color=red, fill=red, fill opacity=0.2, thick] (0.5, 0.5) -- (1.5, 0.5) -- (1.5, 1.5) -- cycle;
        \filldraw [color=red, fill=red, fill opacity=0.2, thick] (0.5, 0.5) -- (0.5, 1.5) -- (1.5, 1.5) -- cycle;
        \filldraw [color=red, fill=red, fill opacity=0.2, thick] (1.5, 0.5) -- (1.5, 1.5) -- (0.5, 1.5) -- cycle;
    \end{tikzpicture}
    &
    &
    0, 0, 0
    &
    --
    \\
\end{tabular}
    \caption{
    Pairs of closed simplices $\closure{\sigma}$
    alive at $a$
    and possible subcomplexes $\closure{\sigma}\cap \AutoDsq^a$
    and their Conley index.}
    \label{tab:pwlinear-critical-points}
\end{table}

\subsubsection{Geometric characterisation of critical points}
\label{ssec:pl-geometric-characterisation}

In this section we will show a bijective
correspondence between certain points in $\UConf{2}{S^1}$ whose image under an embedding $\gamma$
satisfies certain geometric conditions (given in \cref{dfn:pl-critical-index})
and homological critical simplices ($n$-saddles)
of $\nerve{\cV^\bullet}$.

\paragraph{Local Geometry of PL Loops}
In subsequent discussions, we find it useful to describe the local geometry of a PL loop using the following quantities, illustrated in~\Cref{fig:tangentangle}. On each chart $\vartheta_i: U_i \to \R$ of $S^1$, we define the one-sided unit `tangent' vectors as the limiting tangents,
\begin{equation} \label{eq:PLtangent}
    \tau^\pm(t)  = \lim_{h \to 0^\pm} \frac{f\circ \inv{\vartheta_i}(t +h/L)-  f\circ\inv{\vartheta_i}(t)}{h} \in \R^d.
\end{equation}
 If $t \in (L_i, L_{i+1})$, i.e. $f\circ \inv{\vartheta_i}$ is not one of the vertices $x_j$, then the limits agree $\tau^+(t) = \tau^-(t)$. Furthermore,  if we assume that the vertices of $f$ are in general position~\labelcref{C3}, then the limits agree iff $f(z)$ is in the interior of a segment, as no three vertex points are allowed to be collinear. Note that due to the parametrisation of $f$ \cref{eq:PLprimitive}, $\tau^\pm = u_i/l_i$ for some $i$, and thus $\norm{\tau^\pm} = 1$ are always unit tangent vectors. In subsequent discussions, we can assume the atlas $\Theta$ is fixed for given cyclic data, and abuse notation and let $\tau^\pm(z) =  \tau^\pm(t)$ for $t = \vartheta_i(z)$.

We find it useful to characterise the relative directions of the loop between two points $z_1, z_2 \in S^1$. Letting $v_{1,2} = f(z_2) - f(z_1)$ denote the relative position between two points on the loop, we define the \emph{tangent angles of $f$
 at $z_1$ relative to $z_2$}
 as the angles $\theta^\pm_f(z_1, z_2) \in [0,\pi)$
between $v_{1,2}$ and $\pm \tau^\pm (z_1)$, as illustrated in~\Cref{fig:tangentangle}:
 \begin{align} \label{eq:tangentangles}
     \cos(\theta^\pm_f(z_1, z_2)) &= \pm \innerprod{\tau^\pm(z_1)}{v_{1,2}},\\
     \cos(\theta^\pm_f(z_2, z_1)) &= \pm \innerprod{\tau^\pm(z_2)}{v_{2,1}} =  \mp \innerprod{\tau^\pm(z_2)}{v_{1,2}}.
 \end{align}

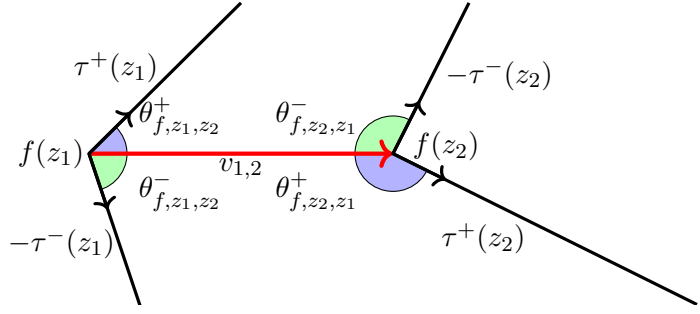
\begin{figure}[ht]
\centering
\begin{tikzpicture}
\clip (-3.5, -2) -- (-3.5, 2) -- (6, 2) -- (6, -2);
\begin{scope}
\clip (0, 2) -- (-2, 0) -- (2, 0);
\draw[fill=blue!30] (-2, 0) circle (0.5);
\end{scope}
\begin{scope}
\clip (-1, -3) -- (-2, 0) -- (2, 0);
\draw[fill=green!30] (-2, 0) circle (0.5);
\end{scope}

\begin{scope}
\clip (6, -2) -- (2, 0) -- (-2, 0);
\draw[fill=blue!30] (2, 0) circle (0.5);
\end{scope}
\begin{scope}
\clip (3, 2) -- (2, 0) -- (-2, 0);
\draw[fill=green!30] (2, 0) circle (0.5);
\end{scope}

\draw[->][ultra thick, red] (-2, 0) -- (2, 0);
\node at (0, -0.2) {$v_{1,2}$};
\draw[very thick] (-1, -3) -- (-2, 0) -- (0, 2);
\draw[very thick] (6, -2) -- (2, 0) -- (3, 2);

\draw[->, very thick] (2, 0) -- (2 + 0.35, 0.7);
\node at (2 + 0.7 * 1.0 + 0.35*2, -0.35 * 1.0 + 0.7*2) {$-\tau^-(z_2)$};
\draw[->, very thick] (2, 0) -- (2 + 0.7, -0.35);
\node at (2 - 0.35 * 0.6 + 0.7*2, -0.7 * 0.6 - 0.35*2) {$\tau^+(z_2)$};

\draw[->, very thick] (-2, 0) -- (-2 + 0.247, -0.742);
\node at (-2 + -0.553 * 1.1 + 0.247, -0.553 * 0.8 + -0.742) {$-\tau^-(z_1)$};
\draw[->, very thick] (-2, 0) -- (-2 + 0.553, 0.553);
\node at (-2 - 0.247 * 0.8 + 0.553, +0.742 * 0.8 + 0.553) {$\tau^+(z_1)$};

\node at (-2.5, 0) {$f(z_1)$};
\node at (2.7, 0.1) {$f(z_2)$};

\node at (-0.8, 0.5) {$\theta^+_{f,z_1,z_2}$};
\node at (-0.8, -0.5) {$\theta^-_{f,z_1,z_2}$};
\node at (1.0, 0.5) {$\theta^-_{f,z_2,z_1}$};
\node at (1.0, -0.5) {$\theta^+_{f,z_2,z_1}$};

\end{tikzpicture}
\caption{\label{fig:tangentangle} 
Tangent angles of $f$ at $z_1$ relative to $z_2$ and $z_2$ relative to $z_1$ (Definition~\ref{dfn:pl-curve-away}).
The picture shows the two segments of a clockwise-oriented loop $f:S^1 \to \R$ corresponding to $z_1, z_2$.
Here, $f$ curves towards $f(z_2)$ at $f(z_1)$
and curves away from $f(z_1)$ at $f(z_2)$.
Hence $\ucpt{z_1, z_2}$ satisfies condition \ref{plmain-CTog} of Theorem~\ref{thm:PLmain} and so corresponds to a homological critical
point of index 1, a saddle point.}
\end{figure}

\begin{definition}
\label{dfn:pl-curve-away}
If $\theta^+_f(z_1, z_2)$ and $\theta^-_f(z_1, z_2)$
are both acute, then we shall say 
\emph{$f$ curves towards $f(z_2)$ at $f(z_1)$}.
If neither are acute we shall say
\emph{$f$ curves away from $f(z_2)$ at $f(z_1)$}.
\end{definition}

\begin{definition}
    \label{dfn:pl-critical-index}
    For $f: S^1 \to \R^d$ piecewise-linear,
    we define three conditions for a point $\ucpt{z_1, z_2} \in \UConf{2}{S^1}$:
    \begin{enumerate}[label = (K\arabic*), start = 0]
    \item \label{plmain-CAw}$f$ curves away from $f(z_2)$ at $f(z_1)$, and away from $f(z_1)$ at $f(z_2)$;
    \item \label{plmain-CTog} $f$ curves towards $f(z_2)$ at $f(z_1)$, and away from $f(z_1)$ at $f(z_2)$, or vice versa;
    \item \label{plmain-CTow} $f$ curves towards $f(z_2)$ at $f(z_1)$, and towards $f(z_1)$ at $f(z_2)$.
    \end{enumerate}
    Note that at most one of these conditions may hold.
    If \ref{plmain-CAw}, \ref{plmain-CTog}, or \ref{plmain-CTow} hold
    we define the index $\lambda(\ucpt{z_1, z_2})$ as 0, 1, or 2 respectively. 
    We let $C_a^i \subset \fibre{f}{a}$ be the points in the level set at $a$ which satisfy condition $(Ki)$. 
\end{definition}

\begin{restatable}{theorem}{PLmain}
\label{thm:PLmain}
Suppose $\gamma: S^1 \to \R^d$ is a piecewise linear interpolation of cyclic data $(x_i)_{i \in \Z_n}$ satisfying the non-degeneracy conditions~\labelcref{C1,C2,C3}. For $a > 0$ and $\epsilon >0$ sufficiently small, 
\begin{equation}
    H_\bullet(\AutoDsq^a, \AutoDsq^{a-\epsilon}) \cong \bigoplus_{i = 0}^2 \bigoplus_{w \in C_a^i} \tilde{H}_\bullet(S^i),
\end{equation}
where $C_a^i$ are points in $\fibre{\DT_\gamma}{a}$ described in Definition~\ref{dfn:pl-critical-index} satisfying the geometric properties~\labelcref{plmain-CAw,plmain-CTog,plmain-CTow}. For $a = 0$, we also have $H_\bullet(\AutoDsq^0, \AutoDsq^{-\epsilon}) \cong H_\bullet(S^1)$.
\end{restatable}

\begin{proof}
    From Proposition~\ref{prop:morse-set-pl},
    the only critical Morse sets are $S_{\partial \Mob}$ at $a=0$
    and those generated by simplices for $a>0$.
    For $\ucpt{z_1, z_2} \in \UConf{2}{S^1}$
    define 
    $$
    \sigma\ucpt{z_1, z_2} = \bigcap_{V_{ij}:\ucpt{z_1, z_2} \in V_{i j}} V_{ij},
    $$
    the highest dimensional simplex containing $\ucpt{z_1, z_2}$.
    Suppose $\DT_\gamma(\ucpt{z_1, z_2}) = a$
    so $\sigma\ucpt{z_1, z_2}$
    is in the filtration at value $a$.
    We will build a correspondence, given in Table~\ref{tab:angle-correspondence}, between $\theta^\pm_{\gamma, z_1, z_2}, \theta^\pm_{\gamma, z_2, z_1}$
    and the subcomplex
    of $\sigma(z_1,z_2)$
    that exists at filtration value $a-\epsilon$,
    for $\epsilon$ sufficiently small that no simplices
    are added within $[a-\epsilon, a)$.
    Unless this subcomplex is $\sigma\ucpt{z_1, z_2}$
    itself, $\sigma\ucpt{z_1, z_2}$
    also first appears at $a$.
    In this case, from the proof of Proposition~\ref{prop:morse-set-pl},
    $\sigma \ucpt{z_1, z_2}$ will be a Morse set.
    We can then refer to Table~\ref{tab:pwlinear-critical-points}
    to see if $\sigma\ucpt{z_1, z_2}$ is critical
    and deduce its Conley index.

    We shall now show the correspondence in Table~\ref{tab:angle-correspondence}.
    Suppose $f(z_1) = x_i$
    and $f(z_2) = x_j$ for $i\neq j$, i.e. $\sigma\ucpt{z_1, z_2}$ is a 3-simplex.
    Then using the cosine rule
    we have for $0<\delta \leq l_i$ that
    $$
    \DT_\gamma(\ucpt{\exp(2\pi\imath(\frac{\delta}{L}+L_i)), z_2}) =
    \frac{1}{2}\|\gamma(z_2) - (x_i + \delta\tau^+(z_1))\|^2
    < \frac{1}{2}\|v_{1,2}\|^2 = \DT_\gamma(\ucpt{z_1, z_2})
    $$
    if and only if $\delta < 2\|v_{1,2}\|\cos{\theta^+_{f, z_1, z_2}}$.
    Therefore
    the simplex $V_{i,j-1} \cap V_{i,j}$ is born strictly before 
    $\sigma\ucpt{z_1, z_2} = V_{i-1,j-1} \cap V_{i j-1} \cap V_{i-1, j} \cap V_{i j}$
    if and only if $\theta^+_{\gamma,z_1, z_2}$ is acute.
    Likewise, the simplices
    $$
        V_{i-1,j-1} \cap V_{i-1,j},\quad
        V_{i-1,j} \cap V_{ij},\quad \text{ and } \quad
        V_{i-1,j-1} \cap V_{ij-1}
    $$
     are born strictly before
     $\sigma\ucpt{z_1, z_2}$
    exactly when $\theta^-_{\gamma,z_1, z_2}$, $\theta^+_{\gamma,z_2, z_1}$,
    and $\theta^-_{\gamma,z_2, z_1}$ are respectively acute.
    Now consider the case where $\sigma(z_1, z_2)$ is 0 or 1 dimensional, so w.l.o.g. $f(z_1)$ lies in the interior of $\overline{x_i x_{i+1}}$
    and $\tau_1^+ = \tau_1^-$.
    Then unless both $\theta^\pm_{\gamma, z_1, z_2}$ are right,
    $\sigma(z_1, z_2)$ exists before $a$.
    This covers the dimension 0 case.
    For the dimension 1 case, using a similar cosine rule argument
    we have that the simplices $U^{ij}$ and $U^{i j+1}$
    are born strictly before $U^{ij} \cap U^{i j+1}$
    if and only if both $\theta^\pm_{\gamma, z_1, z_2}$ are right
    and both $\theta^\pm_{\gamma, z_2, z_1}$ are respectively acute.
    Considering the different combinations of the tangent angles being acute, right, or obtuse gives Table~\ref{tab:pwlinear-critical-points}.
    \end{proof}

The geometric characterisations of critical points for smooth loops
(Theorem~\ref{thm:SmoothChar})
versus piecewise-linear loops (Theorem~\ref{thm:PLmain})
bare similarities. 
A critical point of the \functionname for a smooth loop where 
$\kappa_{12}, \kappa_{21} < 0$, a local minimum, 
would also correspond\footnote{
Under an appropriate discretisation with a vertex
at each of the points in the critical chord endpoints,
and where the discretisation is sufficiently fine.}
to a local minimum in the piecewise-linear case, satisfying~\ref{plmain-CAw}.
Similarly, for $|\kappa_{12}| \approx |\kappa_{21}| \gg 0$,
local maxima and saddle points of the smooth loop correspond
to local maxima and saddle points of an appropriate PL approximation.
However, when the relative curvature $\kappa_{12}/\kappa_{21}$ is very large or small,
the opposing segments are approximately parallel and thus are saddle points,
a phenomenon not seen in the PL case.

\begin{table}
    \centering
         \begin{tabular}{|c|c|c|c|}
         \hline
         $\theta_{\gamma, z_1, z_2}^\pm$ & $\theta_{\gamma, z_2, z_1}^\pm$ & Table~\ref{tab:pwlinear-critical-points} lines& index\\
         \hline
         \multirow{6}{2em}{RR}
        & RR & 1, 3, 7& 0\\
        & AR & 4, 8 & --\\
        & OR & 3, 7& 0\\
        & AA & 5, 9& 1\\
        & AO & 2, 4, 8& --\\
        & OO & 3, 7& 0\\
        \hline
        \multirow{5}{2em}{AR}
        & AR & 10& --\\
        & OR & 8& --\\
        & AA & 11& --\\
        & AO & 6, 10& --\\
        & OO & 8& --\\
        \hline
        \multirow{4}{2em}{OR}
        & OR & 7& 0\\
        & AA & 9& 1\\
        & AO & 6, 8& --\\
        & OO & 7& 0\\
        \hline
        \multirow{3}{2em}{AA}
        & AA & 12& 2\\
        & AO & 6, 11& --\\
        & OO & 9& 1\\
        \hline
        \multirow{2}{2em}{AO}
        & AO & 2, 6, 10& --\\
        & OO & 8& --\\
        \hline
        \multirow{1}{2em}{OO}
        & OO & 7& 0\\ \hline
    \end{tabular}
    \begin{tabular}{|c| c c c c c c|}
    \hline
        \diagbox{$\theta_{\gamma, z_1, z_2}^\pm$}{index}{$\theta_{\gamma, z_2, z_1}^\pm$}&AA&RR&OR&OO&AR&AO \\ \hline
         AA &2 &1 &1 &1 &--&--\\
         RR &1 &0 &0 &0 &--&--\\
         OR &1 &0 &0 &0 &--&--\\
         OO &1 &0 &0 &0 &--&--\\ 
         AR &--&--&--&--&--&--\\
         AO &--&--&--&--&--&-- \\ \hline
    \end{tabular}
    \caption{Correspondence between tangent angles and $n$-saddles.\\
    Key: 
    R=right angle, A=acute angle, O=obtuse angle.}
    \label{tab:angle-correspondence}
\end{table}

\newpage

\section{Higher Order Finite Subset Spaces and the Volume Transform}
\label{sec:ho-functions}

In order to generalise our approach to higher order spaces, we need to know more about their topological structure. If higher order finite subset spaces could be triangulated then they would be more tractable for computations, such as the pipeline we have constructed for the CDT. However, to do this we must first understand what these higher order finite subset spaces are. Difficulties arise from the fact that even if you start with a manifold $M$, higher order finite subset spaces need not be manifolds or manifolds with boundary (See section 7 of \cite{KallelSjerve2009} and the last remark  of section 1.2 of \cite{tuffleyexp}). A detailed understanding of the topology of these spaces will allow us to create a generalisation of our pipeline of the CDT.

\begin{remark}
    We should note that, in the generalized approach there are two options, either factor through the symmetric product or the $n$-th finite subset space. If one was to factor to just the symmetric product, then for functions depending only on the set, you would introduce duplicate computations, as we have that the $\ExpS{n}{X}$ is a quotient of $SP^n(X)$.
\end{remark}

\subsection{Topological properties of higher order finite subset spaces}

In \Cref{subsubsec:StratHO} we discussed how finite subset spaces are stratified and how they fit together. We might also like to know about the overall topological structure of these spaces. For a thorough overview, \cite{KallelSjerve2009}, \cite{lazovskis2026simpleconnectednessranspace}, \cite{handel} are recommended. What we find, is that for $n\geq 3$, we have that $\ExpS{n}{X}$ is simply connected \cite{KallelSjerve2009}. A recent result from Lazovskis \cite{lazovskis2026simpleconnectednessranspace}, provides a different proof approach for $n\geq 4$, using the result that the map $\pi_1(\ExpS{n}{X}) \rightarrow \pi_1(\ExpS{n+2}{X})$ induced by the inclusion of spaces is trivial for all $n$.

\begin{theorem}
    \cite{KallelSjerve2009} \cite{lazovskis2026simpleconnectednessranspace} For every positive integer $n \geq 3$, $\pi_1 (\ExpS{n}{X})$ is trivial.
\end{theorem}

So when considering a filtration of a higher order $n$-th finite subset space, there can be no point of the form $(a,\infty)$ in the $H_1$ persistence diagram.

As was mentioned in \Cref{sec:Background:config}, if we start with a manifold $X$ for higher orders we are not guaranteed that $\ExpS{n}{X}$ is a manifold or manifold with boundary. If we restrict ourselves to closed manifolds (compact without boundary) then we have a complete description as follows:

\begin{theorem} \label{thm:almost_never_manifold}
    \cite{KallelSjerve2009} Let $X$ be a closed manifold of dimension $d\geq 1$, then $\ExpS{n}{X}$ is a closed manifold if and only if, $d=1$ and $n=3$ \emph{or}, $d=2$ and $n=2$.
\end{theorem}

This result excludes the trivial case of $n=1$. To prove this result, a lemma is used that doesn't require $X$ to be a closed manifold.

\begin{lemma}
    \cite{KallelSjerve2009} If $X$ is a manifold of dimension $d>2$, then $\ExpS{n}{X}$ is never a manifold for $n \geq 2$.
\end{lemma}

For higher order $n$-th finite subset spaces we may not have a manifold structure, but we do still have the following:

\begin{lemma}
    \cite{KallelSjerve2009} For $X$ a simplicial complex then, $\ExpS{n}{X}$ has a CW-decomposition with top cells in dimension $n \ \mathrm{dim}(X)$, so that $H_{*}(\ExpS{n}{X})$ is trivial for $* > n \ \mathrm{dim}(X)$.
\end{lemma}

Furthermore, in section 5.3 of \cite{TuffPHD} we have that if $X$ is triangulated then so is $\ExpS{n}{X}$. Hence, although in most cases we lose the manifold structure we are theoretically able to construct a triangulation. This would enable us to perform computations when considering filtrations on the higher order finite subset spaces

\subsection{Volume functions}

Above we saw that we can generalize our approach from the CDT to higher order spaces. To do this we need to consider higher order functions. In this subsection we show the continuity and stability of a higher order volume function, namely the square volume of the $k$-dimensional simplex of $k+1$ points.

The natural generalization from the \functionname is to consider generalisations of distance functions to $(k+1)$-points where $k > 1$, which are invariant with respect to the isometries of the space, and are only dependent on the set of points. We call such functions \emph{geometric functions}. A natural candidate is the $k$-dimensional volume of convex hulls of $(k+1)$ points: for a pair of points, this reduces to the distance between points. 

We first show how these geometric functions induce a function on $\ExpS{k+1}{X}$. Let $X$ be a compact space and let $Y$ be a metric space and a continuous function $\mathfrak{e}: X \rightarrow Y$. Consider a geometric function $\Gamma^{n}_{\mathfrak{e}} : X^{k+1} \rightarrow \mathbb{R}_{\geq 0}$, defined using $\mathfrak{e}$ and such that it is invariant when we permute $(x_0, x_1, \dots , x_k)$, and is only dependent on the set of points, then it will factor through a quotient map from $X^{k+1}$ to the $(k+1)$-th finite subset space $\ExpS{k+1}{X}$ (note that if we do not wish it to be only dependent on the set of points then one could factor to the symmetric product $SP^{k+1}(X)$):
\begin{equation}
    \begin{tikzcd}[cramped]
	{X^{k+1}} & {\mathbb{R}_{\geq 0}} \\
	{\ExpS{k+1}{X}}
	\arrow["\Gamma_\mathfrak{e}", from=1-1, to=1-2]
	\arrow["q"', two heads, from=1-1, to=2-1]
	\arrow["{\exists ! g_{\mathfrak{e}}}"', dashed, from=2-1, to=1-2]
\end{tikzcd}
\end{equation}
Here $g_{\mathfrak{e}}: \ExpS{k+1}{X} \to [0,\infty)$ is the unique, continuous map such that the diagram commutes, and is given explicitly by 
\begin{align}
    g_{\mathfrak{e}}(\{x_0, x_1, \dots , x_k\}) &= \Gamma_\mathfrak{e}(x_0, x_1, \dots , x_k). \label{eq:GeomFN}
\end{align}
In particular, we focus on the following case. For $k+1$ points on a space $X$, and a map $\mathfrak{e}: X \to \mathbb{R}^d$, we can consider the set of points $\{\varphi(x_0), \ldots, \varphi(x_k)\}$, and consider the $k$-dimensional square volume of the convex hull of those points. We define this map as $\Gamma_\mathfrak{e}^k : X^{k+1} \rightarrow \mathbb{R}_{\geq 0}$ for $k \leq d$ by, $$\Gamma_\mathfrak{e}^k(x_0 , \dots , x_k) := Vol(\Delta^k(\mathfrak{e}(x_0), \dots , \mathfrak{e}(x_k)))^2 $$ We then define the volume transform, $\VolT_k : C(X,\mathbb{R}^d) \rightarrow C(\ExpS{k+1}{X}, \mathbb{R})$ by, $ \mathfrak{e} \mapsto g_{\mathfrak{e}}$. We have the following continuity result, if we impose the following metrics on the function spaces induced by the sup-norm, which is attained by maxima as $X$ and $\ExpS{k+1}{X}$ are both compact.  For $\mathfrak{e}_1,\mathfrak{e}_2 \in C(X,\mathbb{R}^d)$ we define $d(\mathfrak{e}_1,\mathfrak{e}_2) := \max_{x \in X} \|\mathfrak{e}_1(x) - \mathfrak{e}_2(x)\|_2 $; for $h_1,h_2 \in C(\ExpS{k+1}{X}, \mathbb{R})$, we let $d(h_1,h_2):= \max_{\mathbf{x} \in \ExpS{k+1}{X}} |h_1(\mathbf{x}) -h_2(\mathbf{x})|$.

\begin{restatable}{theorem}{ContStabVol}
\label{thm:ContStabVol} 
If $X$ is a compact space, then the volume transform $$\VolT_k : C(X,\mathbb{R}^d) \rightarrow C(\ExpS{k+1}{X}, \mathbb{R})$$ is continuous for both spaces equipped with the sup-norm. In particular, if $d(\mathfrak{e}_1,\mathfrak{e}_2) \leq M$ for $\mathfrak{e}_1,\mathfrak{e}_2 \in C(X,\R^d)$, and 
$$D = \max(\diam(\mathfrak{e}_1), \diam(\mathfrak{e}_2)),$$
we have a bound
\begin{equation}
d(\VolT_k(\mathfrak{e}_1),\VolT_k(\mathfrak{e}_2)) \leq \frac{(k+2) \sqrt{k(k+1)} }{(k!)^2 2^{k-2}} M ( M + D) \ \sqrt{(k^2+2k+2) D^2 + (2k+2)}^{k+1}. 
\end{equation}
\end{restatable}

\begin{proof}

Let $\mathfrak{e}_1, \mathfrak{e}_2 \in C(X, \mathbb{R}^N)$, with $\Gamma_\mathfrak{e}^k $ and $g_{\mathfrak{e}} $ as above. The volume squared,  $ Vol(\Delta^k(\mathfrak{e}(x_0), \dots , \mathfrak{e}(x_k)))^2 $ can be given by using the determinant of the Cayley-Menger matrix:

\[
A_{\mathfrak{e}}^k(\mathbf{x}) = 
\begin{pmatrix} 
0 & d_{01}^2 & d_{02}^2 & \cdots & d_{0k}^2 & 1 \\
d_{01}^2 & 0 & d_{12}^2 & \cdots & d_{1k}^2 & 1 \\
d_{02}^2 & d_{12}^2 & 0 & \cdots & d_{2k}^2 & 1 \\
\vdots & \vdots & \vdots & \ddots & \vdots & \vdots \\
d_{0k}^2 & d_{1k}^2 & d_{2k}^2 & \cdots & 0 & 1 \\
1 & 1 & 1 & \cdots & 1 & 0
\end{pmatrix}
\]
where $\mathbf{x} = (x_0, \dots , x_k)$ $ \in X^{k+1}$,  \(d_{ij} = \|\mathfrak{e}(x_i) - \mathfrak{e}(x_j)\|_2\); then $Vol(\Delta^k(\mathfrak{e}(x_0), \dots , \mathfrak{e}(x_k)))^2  = \frac{(-1)^{k+1}}{(k!)^2 2^k} \det{A_{\mathfrak{e}}^k(\mathbf{x})}$.

Let $B = A_{\mathfrak{e}_1}^k(\mathbf{x}) - A_{\mathfrak{e}_2}^k(\mathbf{x})$, then $B $ has at most $(k+2)^{2} - 3k -4 = k(k+1)$ non-zero terms. So now we wish to bound the difference of the squares of  \(d_{ij} = \|\mathfrak{e}(x_i) - \mathfrak{e}(x_j)\|_2\) for $\mathfrak{e}_1$ and $\mathfrak{e}_2$, i.e wishing to bound $ b_{ij} = \|\mathfrak{e}_1(x_i) - \mathfrak{e}_1(x_j)\|_2^2 - \|\mathfrak{e}_2(x_i) - \mathfrak{e}_2(x_j)\|_2^2$. If we let $M:=\max_{x \in X} \|\mathfrak{e}_1(x) - \mathfrak{e}_2(x)\|_2 $, we find that $$|b_{ij}| \leq 4M ( M + \text{min}\{\text{diam}(\mathfrak{e}_1(X)), \text{diam}(\mathfrak{e}_2(X))\})$$

So, $$\|B\|_2 \leq \sqrt{k(k+1)} \ \  4M ( M + \text{min}\{\text{diam}(\mathfrak{e}_1(X)), \text{diam}(\mathfrak{e}_2(X))\}) $$

Let us then bound $d(\VolT_k(\mathfrak{e}_1), \VolT_k(\mathfrak{e}_2)) = \max_{\mathbf{x} \in X^{k+1}} |\Gamma^k_{\mathfrak{e}_1}(\mathbf{x}) - \Gamma^k_{\mathfrak{e}_2}(\mathbf{x})|$. 
\[
\begin{aligned}
\bigl|\Gamma^k_{\mathfrak{e}_1}(\mathbf{x}) - \Gamma^k_{\mathfrak{e}_2}(\mathbf{x})\bigr|
&= \bigl|
Vol\bigl(\Delta^k(\mathfrak{e}_1(x_0), \dots, \mathfrak{e}_1(x_k))\bigr)^2
- Vol\bigl(\Delta^k(\mathfrak{e}_2(x_0), \dots, \mathfrak{e}_2(x_k))\bigr)^2
\bigr| \\[0.5em]
&= \left|
\frac{(-1)^{k+1}}{(k!)^2 2^{k}} \det A_{\mathfrak{e}_1}^k(\mathbf{x})
- \frac{(-1)^{k+1}}{(k!)^2 2^{k}} \det A_{\mathfrak{e}_2}^k(\mathbf{x})
\right| \\[0.5em]
&= \frac{1}{(k!)^2 2^{k}}
\bigl|\det A_{\mathfrak{e}_1}^k(\mathbf{x}) - \det A_{\mathfrak{e}_2}^k(\mathbf{x})\bigr| \\[0.5em]
&\le
\frac{1}{(k!)^2 2^{k}} (k+2)
\bigl\|A_{\mathfrak{e}_1}^k(\mathbf{x}) - A_{\mathfrak{e}_2}^k(\mathbf{x})\bigr\|_2 
\max\!\left\{
\bigl\|A_{\mathfrak{e}_1}^k(\mathbf{x})\bigr\|_2,
\bigl\|A_{\mathfrak{e}_2}^k(\mathbf{x})\bigr\|_2
\right\}^{k+1} \\[0.5em]
&\le
\frac{1}{(k!)^2 2^{k}} (k+2)\sqrt{k(k+1)} \, 4M
\bigl(M + \min\{\diam(\mathfrak{e}_1(X)), \diam(\mathfrak{e}_2(X))\}\bigr) \\
&\qquad \cdot
\max\!\left\{
\bigl\|A_{\mathfrak{e}_1}^k(\mathbf{x})\bigr\|_2,
\bigl\|A_{\mathfrak{e}_2}^k(\mathbf{x})\bigr\|_2
\right\}^{k+1} \\[0.5em]
&\le
\frac{(k+2)\sqrt{k(k+1)} M
\bigl(M + \min\{\diam(\mathfrak{e}_1(X)), \diam(\mathfrak{e}_2(X))\}\bigr)}
{(k!)^2 2^{k-2}} \\
&\qquad \cdot
\max\!\Bigl\{
\sqrt{(k^2+2k+2)\diam(\mathfrak{e}_1(X))^2 + (2k+2)^2}, \\
&\qquad\qquad
\sqrt{(k^2+2k+2)\diam(\mathfrak{e}_2(X))^2 + (2k+2)^2}
\Bigr\}^{k+1}.
\end{aligned}
\]

With the first inequality coming from theorem 2.12 of \cite{IpsenRehman2008}.
This then provides an upper bound on $d(\VolT_k(\mathfrak{e}_1), \VolT_k(\mathfrak{e}_2))$ as the supremum of the left hand side of the inequality.

Let us fix $\mathfrak{e}_1$. Because $|\diam(\mathfrak{e}_1)  - \diam(\mathfrak{e}_2)|\leq 2M$, for any $\epsilon > 0$, we can find a sufficiently small $M > 0$, such that for all $\mathfrak{e}_2$ in an $M$-ball of $\mathfrak{e}_1$, the right hand side of the equality is bounded below $\epsilon$. Thus $\VolT$ is continuous with respect to the sup-norm metrics.

Let $D = \text{max}\{\text{diam}(\mathfrak{e}_1(X)), \text{diam}(\mathfrak{e}_2(X))\}$. We can then simplify the inequality to 

$$ d(\VolT_k(\mathfrak{e}_1), \VolT_k(\mathfrak{e}_2)) \leq \frac{(k+2) \sqrt{k(k+1)} }{(k!)^2 2^{k-2}} M ( M + D) \ \sqrt{(k^2+2k+2) D^2 + (2k+2)}^{k+1}.$$

\end{proof}

Having defined a continuous function $g_\mathfrak{e} : \ExpS{k+1}{X} \to \R$, we can construct a sublevel set filtration and consider its persistent homology. Since $\ExpS{k+1}{X}$ is compact and and trianguable (see 5.3 Discussion of \cite{TuffPHD}), continuous functions on $\ExpS{k+1}{X}$ always have $q$-tame persistent homology~\cite[Theorem 2.22]{Chazal2016-bo}. Thus, we can consider the persistence diagrams $\pershomf_i(g_\mathfrak{e})$ in the metric space of diagrams, equipped with bottleneck distance. Because the map $\pershomf$ is a continuous transformation of functions from   \Cref{thm:ContStabVol} to the space of diagrams, the composition of $\VolT_k: C(X, \mathbb{R}^d) \to C(\ExpS{k+1}{S^1}, \R)$ with $\pershomf$ to yield $\pershomf_i(g_\mathfrak{e})$ is a continuous function.

\begin{corollary}
     $\pershomf \circ \VolT_k: C(X,\mathbb{R}^d) \to \Dgm$ continuous w.r.t. bottleneck distance on $\Dgm$, for $X$ triangulable.
\end{corollary}

\section{Discussion and Future Work}\label{sec:disft}

The motivation for this work was to find a geometric descriptor of embedded loops that would be invariant under internal symmetries like diffeomorphisms of $S^1$ and external symmetries such as euclidean transformations, and stable under perturbation. In this article, we have proved that the composition $\pershomf \circ \DT$ of the chordal distance transform $\DT$ and persistent homology $\pershomf$ satisfies these desirable properties. Furthermore, we have characterised the behaviour of the map $\pershomf \circ \DT$ for generic $C^2$ and finite piecewise linear embeddings, and shown such persistence diagrams are generically tame. In the generic setting, we have also given a geometric description of the critical points that generate the birth and death of homological features in the persistence diagram $\pershomf \circ \DT(\gamma)$, allowing practitioners who apply this method to interpret $\pershomf \circ \DT(\gamma)$ in terms of the geometry of the embedded loop. While the \functionname captures geometric features associated to pairs of points on the embedded loop, we show that it is a special case of  the volume transformation, a generalisation that captures higher order geometric information associated to multiple points along the loop. We show the volume transform of a compact space embedded in Euclidean space is also continuous and stable.

As future work, we would like to apply a similar Morse-theoretic approach to understanding generalisations of the chordal distance transform to other spaces and higher order point configurations. An immediate generalisation of the chordal distance transform of $S^1$ is to that of $S^2$, where the two-point finite subspace is a manifold $\ExpS{2}{S^2}=\mathbb{C}P^2$; we anticipate a lot of the Morse-theoretic techniques in this work can be transferred easily to analyse that case. However, challenges arise for other spaces as finite subset spaces are manifolds only in a few particular cases (see~\Cref{thm:almost_never_manifold}). Nonetheless, we can explore a similar approach with stratified Morse theory~\cite{Goresky1988-ym} or a recent analogous theory for CW complexes~\cite{nanda2026stratified}. 
Finally, we would like to put this methodology to the test on real world scientific problems, such as the analysis of cells morphologies, ring polymers, and organoid shape where the interpretability of the critical points of the \functionname and the invariance of our descriptor could prove useful in classification tasks.

\section*{Acknowledgements}
James Binnie: This work was supported by the Additional Funding Programme for Mathematical Sciences, delivered by EPSRC (EP/V521917/1) and the Heilbronn Institute for Mathematical Research. Ka Man Yim conducted work on this project while supported by a UKRI Future Leaders Fellowship [grant number MR/W01176X/1; PI: J Harvey], and subsequently by the EPSRC grant ``Mathematical Foundations of Intelligence: An `Erlangen Programme' for AI" [grant number EP/Y028872/1]. 
Otto Sumray was supported by the ELBE postdoctoral fellows programme.

The authors thank David Hien for sharing his insights in discussions on this topic. James Binnie and Ka Man Yim are grateful to John Harvey for his thoughtful suggestions. Ka Man Yim would also like to thank John Harvey and Gesine Reinert for their advice as academic mentors.
 
\appendix

\section{Morton's Construction of the 2\textsuperscript{nd} Symmetric Product of $S^1$} \label{app:mobius}

In this section, we show constructively that $\ExpS{2}{S^1}$ is homeomorphic to the M\"{o}bius band $\Mob$, via the construction by Morton~\cite{Morton1967-or}
Morton's construction involves deriving the universal covering $\bar{\scre}$ of $\ExpS{2}{S^1}$ by an infinite strip $[0,1] \times \R$, through the universal covering $\tilde{\scre}$ of $S^1 \times S^1$ by $\R^2$: they are related by a commutative diagram 
\begin{equation}
    \begin{tikzcd}[ampersand replacement=\&]
	{\R^2} \& {[0,1] \times \R} \\
	{S^1 \times S^1} \& \ExpS{2}{S^1}
	\arrow["\tilde{q}", "{/(\Z \rtimes S_2)}"', two heads, from=1-1, to=1-2]
	\arrow["{/\Z^2}", "\screii"', two heads, from=1-1, to=2-1]
	\arrow["{/\Z}", "\bar{\scre}"', two heads,  from=1-2, to=2-2]
	\arrow["q", "{/S_2}"', two heads, from=2-1, to=2-2]
\end{tikzcd} \tag{\cref{eq:master_mob_diagram}}
\end{equation}
The main idea behind the construction is to decompose the successive quotients $q \circ \tilde{e}$ into another pair of successive quotients $\bar{\scre} \circ \tilde{q}$, where $\bar{\scre}$ is the universal covering of the M\"obius band $\Mob$ by a $\Z$-action on $[0,1] \times \R$. 

We leverage this construction to induce an atlas of the $\Z$-covering $\bar{\scre}$ from an atlas for the $\Z$-covering $\scre: \R \to S^1$. An atlas of $\scre$ allows us to give local coordinates for $S^1$, and an atlas it induces on $\bar{\scre}$ gives local coordinates for $\ExpS{2}{S^1}$ in terms of local coordinates of $S^1$. We show that in the subsequent section. Here we set out, in concise terms, the technical apparatus needed for the construction of such an atlas. We develop Morton's construction in two steps. First, we show that $\ExpS{2}{S^1}$ can be written as the quotient of $\R^2$ by an action of the semi-direct product $G = \Z^2 \rtimes S_2$. We then show that this quotient is a composition of the two quotients maps $\bar{\scre} \circ \tilde{q}$. The first is a quotient of $\R^2$ by the action of a normal subgroup $N \cong \Z \rtimes S_2$ of $G$. This yields an orbit space $\R^2/N$ homeomorphic to the diagonal strip $[0,1] \times \R$. The second quotient is that of the diagonal strip by a $\Z$-action, which gives the universal covering of the M\"obius band $\Mob$, as illustrated in~\Cref{fig:mobius_commutative}.  

\paragraph{$\ExpS{2}{S^1}$ as a Quotient of $\R^2$ by $\Z^2 \rtimes S_2$}
Recall $\ExpS{2}{S^1}$ is the quotient of the torus $q: S^1 \times S^1 \twoheadrightarrow \ExpS{2}{S^1}$ by the action of the symmetric group $S_2$. Because the torus is itself the quotient $\R^2 \twoheadrightarrow S^1 \times S^1$ of $\R^2$ by a $\Z^2$-covering action, we can express $\ExpS{2}{S^1}$ as a quotient $p: \R^2 \twoheadrightarrow \ExpS{2}{S^1}$, obtained by composing the quotient $\R^2 \twoheadrightarrow S^1 \times S^1$ with $q$. We recall the quotient $\screii := \scre \times \scre : \R^2 {\twoheadrightarrow}  S^1 \times S^1$ is given by the product of $\scre: \R \to S^1$ where
\begin{equation}
    \scre(t) = e^{2\pi \imath  t}. 
\end{equation}
Following~\cite{Morton1967-or}, we first show that the quotient map $p$ is equivalent to the quotient of $\R^2$ by an action of the semi-direct product $G = \Z^2 \rtimes S_2$. The group $G$ has presentation
\begin{equation}
    G = \langle \alpha_1, \alpha_2, \rho\  \vert \ \rho^2 = e,\ \rho \alpha_1 = \alpha_2 \rho,\ \alpha_1 \alpha_2 = \alpha_2 \alpha_1 \rangle 
\end{equation}
where we identify $\alpha_i = (t_i, e)$ with the generators of $\Z^2$, 
\begin{equation*}
    t_1 = (1,0),\ t_2 = (0,1), 
\end{equation*}
and $\rho = (0,r)$ with the generator $r$ of $S_2 = \langle r \ \vert \ r^2 = e\rangle$. Because $\Z^2$ generated by $\alpha_1,\alpha_2$ is a normal subgroup of $G$, we can express any element $g \in G$ uniquely as $g =  \alpha_1^{k_1}\alpha_2^{k_2} \rho^{k_3}$ for $k_1,k_2,k_3 \in \Z$.

Let us define an action of $G$ on $\R^2$. Denoting coordinates in $\R^2$ by $(x_1,x_2)$, we have
\begin{equation}
    \alpha_1 \cdot (x_1,x_2) = (x_1 + 1, x_2)\qc  \alpha_2 \cdot (x_1,x_2) = (x_1 , x_2 + 1)\qc \rho \cdot (x_1,x_2) = (x_2,x_1).
\end{equation}
Thus, for any element $g =  \alpha_1^{k_1}\alpha_2^{k_2} \rho^{k_3}$, 
\begin{equation}
    g \cdot (x_1, x_2) = (x_{r^{k_3}(1)} + k_1, x_{r^{k_3}(2)} + k_2). 
\end{equation}
Let us denote the quotient of $\R^2$ by the $G$-action $p': \R^2 \twoheadrightarrow \R^2/G$. We claim that $p' = p$ up to homeomorphism. We can check by direct computation that the quotient $p = q \circ 
\screii : \R^2 \twoheadrightarrow \ExpS{2}{S^1}$  sends points along an orbit $G \cdot (x_1,x_2)$ to the same point $\{ \scre(x_1), \scre(x_2)\} \in \ExpS{2}{S^1}$, and the fibres $\fibre{p}{\{ \scre(x_1), \scre(x_2)\}}$ are precisely the orbits $G \cdot (x_1,x_2)$. Thus, the universal property of quotient maps imply we have a homeomorphism $\ExpS{2}{S^1} \homeo \R^2 /G$ such that the following diagram commutes:
 \begin{equation} \label{eq:Mob1}
    \begin{tikzcd}
        \R^2 \arrow[r, "\screii", "/\Z^2"', two heads] \arrow[d, "p'", "/G"', two heads ]& S^1 \times S^1 \arrow[d, "q", "/S_2"', two heads] \\
        \R^2/G \arrow[r, "\homeo"] & \ExpS{2}{S^1}.
    \end{tikzcd}
\end{equation}
Henceforth we consider $\ExpS{2}{S^1}$ and $\R^2/G$ interchangeably. 

\paragraph{The Diagonal Presentation}
We give another presentation of $G$. If we define $\mu = \rho \alpha_1$ and $\nu = \inv{\alpha_1} \alpha_2$, then we can write 
\begin{equation} \label{eq:G_diag_present}
     G = \langle \mu, \nu, \rho\  \vert \ \rho^2 = e,\ \nu \rho \nu = \rho,\ \mu \rho =  \rho\nu \mu  \rangle 
\end{equation}
We observe from this presentation of $G$ that we can write $G$ as another semi-direct product $G \cong N \rtimes H$, where $N = \Z \rtimes S_2 = \langle \nu, \rho \ \vert \ \nu \rho \nu  = \rho \rangle$ and $G / N = H \cong \langle \mu \rangle \cong \Z$. We are also free to write any element of $G$ uniquely as $g = \mu^{n_1} \nu^{n_2} \rho^{n_3}$. 

Let us consider the action of $N$ on $\R^2$. We relate the action of $N$ to the full action of $G$ by considering how the $G$-action acts on the $N$-orbits. The action of $G$ on $N$-orbits $(N \cdot x) \subset \R^2$ is a well-defined map sending $N$-orbits to $N$-orbits, as $N$ being a normal subgroup of $G$ implies 
\begin{equation}
    g \cdot (N\cdot x) = N \cdot(g \cdot x). 
\end{equation}
Furthermore, the action on $\R^2/N$ is given by that of $H \cong \langle \mu \rangle$: because we can write any element of $G$ as  $g = \mu^k n$, for some  $n \in N$, the action of $g$ sends 
\begin{equation} \label{eq:Zaction_on_R2modN}
    g \cdot (N \cdot x) = \mu^{k} \cdot (N \cdot x) = N \cdot (\mu^{k} \cdot x).
\end{equation}
This implies the $G$-action on $N$-orbits is a $\Z$-action. Furthermore, because $\langle \mu \rangle \cap N = e$, the $N$-orbit of $\mu^{k} \cdot x$ is distinct from that of $x$ for $k \neq 0$, hence the $\Z$-action is a free action on $\R^2/N$. Henceforth, let us denote the $\Z$-action on $\R^2/N$ by
\begin{equation}
    k \star (N \cdot x) := \mu^k \cdot (N \cdot x).
\end{equation}
Because the $\Z$-action on $N$-orbits sends $N \cdot x$ in $\R^2$ to $G \cdot x$, the composition of the $\Z$-action with the $N$-action is the $G$-action on $\R^2$:
\begin{equation}\label{eq:Mob2}
    \begin{tikzcd}[ampersand replacement=\&]
	{\R^2} \\
	{\R^2/N} \& {\R^2/G}
	\arrow[two heads, from=1-1, to=2-1]
	\arrow[two heads, from=1-1, to=2-2]
	\arrow["{/\Z}"', two heads, from=2-1, to=2-2]
\end{tikzcd}.
\end{equation}

\paragraph{The Infinite Strip and the M\"obius band} We now show that the orbit space $\R^2/N$ is homeomorphic to $\stripinf := [0,1] \times \R$. Let us consider the action of generators $\mu,\nu,\rho$ of $G$ on $\R^2$. If we perform a change of coordinates $u = x_2-x_1$ and $v = x_1 + x_2$, the action of generators $\mu,\nu,\rho$ can be given by
\begin{equation} \label{eq:diagonal_generators_action}
  \mu\cdot(u,v) = (1-u, v+ 1)\qc  \nu\cdot(u,v) = (u+2,v)
   \qc \rho \cdot(u,v) = (-u, v).
\end{equation}
Since $N$ is generated by $\nu$ and $\rho$, its action is restricted to fibres $\fibre{v}{c} \homeo \R $ of the trivial bundle $v: \R^2 \twoheadrightarrow \R$, as the action of $\nu$ and $\rho$ keeps the $v$ coordinate constant, and only shifts the $u$ coordinate parametrising the fibres $\fibre{v}{c}$. Furthermore, the action on each fibre $\fibre{v}{c}$ is independent of $c$.  Writing $\R^2 = \R \times \R$ with the first factor the $u$ coordinate parametrising fibre $\fibre{v}{c}$, and the second factor the $v$ coordinate, we can thus express the $N$-orbit space as
\begin{equation*}
    \R^2 / N  = \R/N \times \R  .
\end{equation*}
The reader can verify that the subset where $u \in [0,1]$ in each fibre $\fibre{v}{c}$ is a fundamental domain of the action of $N$. In fact~\cite[Lemmas 1 \& 2]{Morton1967-or} show that $[0,1] \homeo \R/N$, and we have a homeomorphism 
\begin{equation}
   \zeta:  [0,1] \times \R = \inv{u}[0,1]  \hookrightarrow \R^2 \twoheadrightarrow \R^2/N = \R/N  \times \R
\end{equation}
The forward map $\zeta$ send $x \in \inv{u}[0,1]$ to $N\cdot x$, and its inverse $\inv{\zeta}$ sends orbits  $N \cdot x$ to its unique intersection with $\inv{u}[0,1]$. 

Let us then consider the $\Z$-action on $[0,1] \times \R$, given by $k \star x := \fibre{\zeta}{k \star \zeta(x)}$: For $(u,v) \in [0,1] \times \R$,
\begin{align}
    1 \star (u,v) &= \fibre{\zeta}{1 \star \zeta(u,v)} = \fibre{\zeta}{1 \star (N \cdot (u,v))} = \fibre{\zeta}{N \cdot (\mu \cdot (u,v))} \label{eq:Z_action_strip} \\
                 &= \fibre{\zeta}{N \cdot (1-u,v+1)} = (1-u, v+1). \nonumber
\end{align}
We note that $1 \star (u,v) \in [0,1] \times \R$, hence the $\Z$-action on $[0,1] \times \R$ is well-defined. Moreover, the $\Z$-action is that of the \emph{glide reflection}; the action of a generator $1$ restricts to a homeomorphism between fibres $\fibre{v}{c} \cap \inv{u}[0,1]$ and  $\fibre{v}{c+1} \cap \inv{u}[0,1]$, where the $u$ coordinate is reflected $ u \mapsto 1-u$ in its image. The quotient $\bar{\scre}: [0,1] \times \R \twoheadrightarrow \Mob$   by this $\Z$-action is precisely the M\"obius band $\Mob$. This implies $(\R^2/N)/\Z \homeo ([0,1] \times \R)/\Z  = \Mob$. Furthermore, $\bar{\scre}$ is the universal covering of the M\"obius band.

Since we have also shown that $\ExpS{2}{S^1} \homeo \R^2/G $ in~\cref{eq:Mob1} and  $\R^2/G \homeo (\R^2/N)/\Z $ in~\cref{eq:Mob2}, we have a sequence of homeomorphisms that establishes the two-finite subset space is homeomorphic to the M\"obius band: 
\begin{equation}
    \ExpS{2}{S^1} \homeo \R^2/G \homeo (\R^2/N)/\Z \homeo ([0,1] \times \R)/\Z  = \Mob.
\end{equation}
Combining and summarising the commutative diagrams~\cref{eq:Mob1,eq:Mob2}, we obtain a \emph{bundle morphism} $\tilde{q}: \screii \to \bar{\scre}$:
\begin{equation}
    \begin{tikzcd}[ampersand replacement=\&]
	{\R^2} \& {\inv{u}[0,1]} \\
	{S^1 \times S^1} \& \ExpS{2}{S^1}
	\arrow["\tilde{q}", "{/(\Z \rtimes S_2)}"', two heads, from=1-1, to=1-2]
	\arrow["{/\Z^2}", "\screii"', two heads, from=1-1, to=2-1]
	\arrow["{/\Z}", "\bar{\scre}"', two heads,  from=1-2, to=2-2]
	\arrow["q", "{/S_2}"', two heads, from=2-1, to=2-2]
\end{tikzcd} \tag{\cref{eq:master_mob_diagram}}
\end{equation}

\bibliographystyle{plain}
\bibliography{refs}

@article{Bleile2023PersistencePopulation,
  title={Persistence diagrams as morphological signatures of cells: A method to measure and compare cells within a population},
  author={Bokor Bleile, Yossi and Yadav, Pooja and Koehl, Patrice and Rehfeldt, Florian},
  journal={PLOS Computational Biology},
  volume={22},
  number={1},
  pages={e1013890},
  year={2026},
  publisher={Public Library of Science San Francisco, CA USA}
}

@article{arnal2023critical,
  title={Critical points of the distance function to a generic submanifold},
  author={Arnal, Charles and Cohen-Steiner, David and Divol, Vincent},
  journal={arXiv preprint arXiv:2312.13147},
  year={2023}
}

@article{damon1997generic,
  title={Generic properties of solutions to partial differential equations},
  author={Damon, James},
  journal={Archive for Rational Mechanics and Analysis},
  volume={140},
  number={4},
  pages={353--403},
  year={1997},
  publisher={Springer}
}

@book{hirsch2012differential,
  title={Differential topology},
  author={Hirsch, Morris W},
  volume={33},
  year={2012},
  publisher={Springer Science \& Business Media}
}

@book{golubitsky2012stable,
  title={Stable mappings and their singularities},
  author={Golubitsky, Martin and Guillemin, Victor},
  volume={14},
  year={2012},
  publisher={Springer Science \& Business Media}
}

@article{boissonnat2019reach,
  title={The reach, metric distortion, geodesic convexity and the variation of tangent spaces},
  author={Boissonnat, Jean-Daniel and Lieutier, Andr{\'e} and Wintraecken, Mathijs},
  journal={Journal of applied and computational topology},
  volume={3},
  number={1},
  pages={29--58},
  year={2019},
  publisher={Springer}
}

@article{federer1959curvature,
  title={Curvature measures},
  author={Federer, Herbert},
  journal={Transactions of the American Mathematical Society},
  volume={93},
  number={3},
  pages={418--491},
  year={1959},
  publisher={JSTOR}
}

@article{Turner2024Extended,
	author = {Turner, Katharine and Robins, Vanessa and Morgan, James},
	date = {2024/11/01},
	doi = {10.1007/s41468-024-00175-8},
	id = {Turner2024},
	isbn = {2367-1734},
	journal = {Journal of Applied and Computational Topology},
	number = {7},
	pages = {2111-2154},
	title = {The extended persistent homology transform of manifolds with boundary},
	url = {https://doi.org/10.1007/s41468-024-00175-8},
	volume = {8},
	year = {2024}}

@article{tuffleyexp,
    title={Finite subset spaces of S1},
  author={Tuffley, Christopher},
  journal={Algebraic \& Geometric Topology},
  volume={2},
  number={2},
  pages={1119--1145},
  year={2002},
  publisher={Mathematical Sciences Publishers}
}

@article{tuffley2003finite,
  title={Finite subset spaces of graphs and punctured surfaces},
  author={Tuffley, Christopher},
  journal={Algebraic \& Geometric Topology},
  volume={3},
  number={2},
  pages={873--904},
  year={2003},
  publisher={Mathematical Sciences Publishers}
}

@article{BAUER2023125503,
title = {A unified view on the functorial nerve theorem and its variations},
journal = {Expositiones Mathematicae},
volume = {41},
number = {4},
pages = {125503},
year = {2023},
issn = {0723-0869},
doi = {https://doi.org/10.1016/j.exmath.2023.04.005},
url = {https://www.sciencedirect.com/science/article/pii/S0723086923000415},
author = {Ulrich Bauer and Michael Kerber and Fabian Roll and Alexander Rolle}
}

@incollection{mrozek2025dynamics,
  title={Dynamics of Combinatorial Multivector Fields},
  author={Mrozek, Marian and Wanner, Thomas},
  booktitle={Connection Matrices in Combinatorial Topological Dynamics},
  pages={105--118},
  year={2025},
  publisher={Springer}
}

@phdthesis{hien2025topological,
  title={Topological Signatures for Analyzing Recurrence in Sampled Dynamics},
  author={Hien, David Thomas},
  year={2025},
  url = {https://mediatum.ub.tum.de/doc/1772185/1772185.pdf},
  school={Technische Universit{\"a}t M{\"u}nchen}
}

@article{handel,
    author = {Handel, D.},
    title = {Some Homotopy Proporties of Spaces of Finite Subsets of Topological Spaces},
    journal = {Houston Journal of Mathematics} ,
    publisher = {University of Houston},
    year = {2000}
}

@book{kriegl1997convenient,
  title={The convenient setting of global analysis},
  author={Kriegl, Andreas and Michor, Peter W},
  volume={53},
  year={1997},
  publisher={American Mathematical Soc.}
}

@book{hatcher, 
author = {Hatcher, Allen}, 
title = {{Algebraic Topology}},
publisher = {Cambridge Univ. Press}, year = {2002}, url = {https://pi.math.cornell.edu/~hatcher/AT/AT.pdf}}

@book{Orbifolds, 
author = {Adem, Alejandro Leida, Johann and Ruan, Yongbin}, 
title = {{Orbifolds and String Topology}},
publisher = {Cambridge Univ. Press}, 
year = {2009},
ISBN = {9780511543081},
DOI = {https://doi.org/10.1017/CBO9780511543081}
}

@article{Bott,
    author = {Bott, R.},
    title = {On the third symmetric potency of S1},
    journal = {Fund. Math.} ,
    volume={39},
    year = {1952}
}

@article{Knud,
  title={Configuration spaces in algebraic topology},
  author={Knudsen, Ben},
  journal={arXiv preprint arXiv:1803.11165},
  year={2018}
}

@inbook{CohenConf,
author = {Frederick R. Cohen},
title = {INTRODUCTION TO CONFIGURATION SPACES AND THEIR APPLICATIONS},
booktitle = {Braids},
chapter = {},
publisher = { World Scientific Publishing Co. Pte. Ltd.},
year = {2009},
pages = {183-261},
doi = {10.1142/9789814291415_0003},
URL ={https://www.worldscientific.com/doi/abs/10.1142/9789814291415_0003},
eprint = {https://www.worldscientific.com/doi/pdf/10.1142/9789814291415_0003}
   
}

@article{kallel2025conf,
  title={Configuration Spaces of Points: A User's Guide},
  author={Kallel, Sadok},
  journal={arXiv preprint arXiv:2407.11092},
  year={2024}
}

@misc{TuffPHD,
      title={Finite subset spaces of graphs and surfaces, PhD Thesis}, 
      author={Tuffley, Christopher},
      year={2003},
}

@BOOK{Nicolaescu2011-oh,
  title     = "An Invitation to Morse Theory",
  author    = "Nicolaescu, Liviu I",
  publisher = "Springer",
  series    = "Universitext",
  edition   =  2,
  month     =  nov,
  year      =  2011,
  address   = "New York, NY",
  copyright = "https://www.springernature.com/gp/researchers/text-and-data-mining",
  language  = "en"
}

@BOOK{Mazzucchelli2011-tr,
  title     = "Critical point theory for Lagrangian systems",
  author    = "Mazzucchelli, Marco",
  publisher = "Springer",
  series    = "Progress in Mathematics",
  edition   =  2012,
  month     =  nov,
  year      =  2011,
  address   = "Basel, Switzerland",
  copyright = "https://www.springernature.com/gp/researchers/text-and-data-mining",
  language  = "en"
}

@ARTICLE{Morton1967-or,
  title     = "Symmetric products of the circle",
  author    = "Morton, H R",
  journal   = "Math. Proc. Camb. Philos. Soc.",
  publisher = "Cambridge University Press (CUP)",
  volume    =  63,
  number    =  2,
  pages     = "349--352",
  month     =  apr,
  year      =  1967,
  language  = "en"
}

@BOOK{Banyaga2004-my,
  title     = "Lectures on Morse Homology",
  author    = "Banyaga, Augustin and Hurtubise, David",
  publisher = "Springer",
  series    = "Texts in the Mathematical Sciences",
  edition   =  2004,
  month     =  oct,
  year      =  2004,
  address   = "New York, NY",
  language  = "en"
}

@INCOLLECTION{Wall1975-rt,
  title     = "Regular stratifications",
  booktitle = "Lecture Notes in Mathematics",
  author    = "Wall, C T C",
  publisher = "Springer Berlin Heidelberg",
  pages     = "332--344",
  series    = "Lecture Notes in Mathematics",
  year      =  1975,
  address   = "Berlin, Heidelberg"
}

@article{IpsenRehman2008,
  author  = {Ipsen, Ilse C. F. and Rehman, Rizwana},
  title   = {Perturbation Bounds for Determinants and Characteristic Polynomials},
  journal = {SIAM Journal on Matrix Analysis and Applications},
  volume  = {30},
  number  = {2},
  pages   = {762--776},
  year    = {2008},
  publisher = {Society for Industrial and Applied Mathematics}
}

@BOOK{Milnor1963-mi,
  title     = "Morse theory. ({AM-51)}, volume 51",
  author    = "Milnor, John",
  publisher = "Princeton University Press",
  series    = "Annals of Mathematics Studies",
  month     =  may,
  year      =  1963,
  address   = "Princeton, NJ",
  language  = "en"
}

@article{CohenSteiner2007,
  title     = "Stability of persistence diagrams",
  author    = "Cohen-Steiner, David and Edelsbrunner, Herbert and Harer, John",
  journal   = "Discrete Comput. Geom.",
  publisher = "Springer Science and Business Media LLC",
  volume    =  37,
  number    =  1,
  pages     = "103--120",
  month     =  jan,
  year      =  2007,
  language  = "en"
}

@book{edelsbrunner2010computational,
  title={Computational topology: an introduction},
  author={Edelsbrunner, Herbert and Harer, John},
  year={2010},
  publisher={American Mathematical Soc.}
}

@article{edelsbrunner2008persistent,
  title={Persistent homology-a survey},
  author={Edelsbrunner, Herbert and Harer, John and others},
  journal={Contemporary mathematics},
  volume={453},
  number={26},
  pages={257--282},
  year={2008},
  publisher={Providence, RI: American Mathematical Society}
}

@article{AYALA2017903,
title = {Local structures on stratified spaces},
journal = {Advances in Mathematics},
volume = {307},
pages = {903-1028},
year = {2017},
issn = {0001-8708},
doi = {https://doi.org/10.1016/j.aim.2016.11.032},
url = {https://www.sciencedirect.com/science/article/pii/S0001870816316097},
author = {David Ayala and John Francis and Hiro Lee Tanaka},
}

@misc{Lurie2017HigherAlgebra,
  author       = {Lurie, Jacob},
  title        = {Higher Algebra},
  year         = {2017},
  note         = {Preprint},
  url          = {https://people.math.harvard.edu/~lurie/papers/HA.pdf},
}

@BOOK{Chazal2016-bo,
  title     = "The structure and stability of persistence modules",
  author    = "Chazal, Frederic and de Silva, Vin and Glisse, Marc and Oudot,
               Steve Y",
  publisher = "Springer International Publishing",
  series    = "SpringerBriefs in Mathematics",
  edition   =  1,
  month     =  oct,
  year      =  2016,
  address   = "Cham, Switzerland",
  language  = "en"
}

@article{BorsukUlam1931SymmetricProducts,
  author  = {Borsuk, Karol and Ulam, Stanislaw},
  title   = {On symmetric products of topological spaces},
  journal = {Bulletin of the American Mathematical Society},
  volume  = {37},
  number  = {12},
  pages   = {875--882},
  year    = {1931},
  month   = {December}
}

@article{HacquardLebovici2024Euler,
  title   = {Euler Characteristic Tools for Topological Data Analysis},
  author  = {Hacquard, Olympio and Lebovici, Vadim},
  journal = {Journal of Machine Learning Research},
  volume  = {25},
  year    = {2024},
  pages   = {1--39},
  month   = jul
}

@article{Munch02012025,
author = {Elizabeth Munch},
title = {An Invitation to the Euler Characteristic Transform},
journal = {The American Mathematical Monthly},
volume = {132},
number = {1},
pages = {15--25},
year = {2025},
publisher = {Taylor \& Francis},
doi = {10.1080/00029890.2024.2409616},


URL = { 
    
        https://doi.org/10.1080/00029890.2024.2409616
    
    

},
eprint = { 
    
        https://doi.org/10.1080/00029890.2024.2409616
    
    

}

}

@book{NixonAguado2012Feature,
  title     = {Feature Extraction and Image Processing for Computer Vision},
  chapter   = {7},
  author    = {Nixon, Mark and Aguado, Alberto},
  year      = {2012},
  publisher = {Elsevier},
  address   = {Oxford, UK}
}

@article{KUHL1982236,
title = {Elliptic Fourier features of a closed contour},
journal = {Computer Graphics and Image Processing},
volume = {18},
number = {3},
pages = {236-258},
year = {1982},
issn = {0146-664X},
doi = {https://doi.org/10.1016/0146-664X(82)90034-X},
url = {https://www.sciencedirect.com/science/article/pii/0146664X8290034X},
author = {Frank P Kuhl and Charles R Giardina}

}

@article{Gower1975GPA,
  author  = {Gower, J. C.},
  title   = {Generalized Procrustes Analysis},
  journal = {Psychometrika},
  volume  = {40},
  number  = {1},
  pages   = {33--51},
  year    = {1975},
  month   = mar,
  doi     = {10.1007/BF02291478}
}

@article{AdamsRohlfSlice2004Geomorph,
  author  = {Adams, Dean C. and Rohlf, F. James and Slice, Dennis E.},
  title   = {Geometric Morphometrics: Ten Years of Progress Following the {``}Revolution{''}},
  journal = {Italian Journal of Zoology},
  volume  = {71},
  number  = {1},
  pages   = {5--16},
  year    = {2004},
  doi     = {10.1080/11250000409356545}
}

@article{ghrist1999configurationspacesbraidgroups,
  title={Configuration spaces and braid groups on graphs in robotics},
  author={Ghrist, Robert},
  journal={arXiv preprint math/9905023},
  year={1999}
}

@article{KallelSjerve2009,
  author  = {Kallel, Sadok and Sjerve, Denis},
  title   = {Remarks on Finite Subset Spaces},
  journal = {Homology, Homotopy and Applications},
  volume  = {11},
  number  = {2},
  year    = {2009},
  pages   = {229--250}
}

@article{lazovskis2026simpleconnectednessranspace,
  title={Simple connectedness of the Ran space},
  author={Lazovskis, J{\=a}nis},
  journal={arXiv preprint arXiv:2602.09815},
  year={2026}
}

@ARTICLE{Benjamin2023-qb,
  title    = "Homology of homologous knotted proteins",
  author   = "Benjamin, Katherine and Mukta, Lamisah and Moryoussef, Gabriel
              and Uren, Christopher and Harrington, Heather A and Tillmann,
              Ulrike and Barbensi, Agnese",
  journal  = "J. R. Soc. Interface",
  volume   =  20,
  number   =  201,
  pages    = "20220727",
  month    =  apr,
  year     =  2023,
  language = "en"
}

@article{vaughan1977rectangles,
  title={Rectangles and simple closed curves},
  author={Vaughan, H},
  journal={Lecture, Univ. of},
  year={1977}
}

@inproceedings{meyerson1981balancing,
  title={Balancing acts},
  author={Meyerson, Mark D},
  booktitle={Topology Proc},
  volume={6},
  number={1},
  pages={59--75},
  year={1981}
}

@article{Turner2014PHT,
  title   = {Persistent homology transform for modeling shapes and surfaces},
  author  = {Turner, Katharine and Mukherjee, Sayan and Boyer, Doug M.},
  journal = {Information and Inference: A Journal of the IMA},
  volume  = {3},
  pages   = {310--344},
  year    = {2014},
  doi     = {10.1093/imaiai/iau011}
}

@ARTICLE{Ali2023-xo,
  title     = "A survey of vectorization methods in topological data analysis",
  author    = "Ali, Dashti and Asaad, Aras and Jimenez, Maria-Jose and Nanda,
               Vidit and Paluzo-Hidalgo, Eduardo and Soriano-Trigueros, Manuel",
  journal   = "IEEE Trans. Pattern Anal. Mach. Intell.",
  publisher = "Institute of Electrical and Electronics Engineers (IEEE)",
  volume    =  45,
  number    =  12,
  pages     = "14069--14080",
  month     =  dec,
  year      =  2023,
  copyright = "https://creativecommons.org/licenses/by/4.0/legalcode",
  language  = "en",
  doi       = {10.1109/TPAMI.2023.3308391}
}

@article{yang2026topological,
  title={Topological shape transform for thymus structures},
  author={Yang, Haochen and Lebovici, Vadim and Tarcevski, Andreas and Tchernev, Liliana and Zuklys, Saulius and Holl{\"a}nder, Georg A and Byrne, Helen M and Harrington, Heather A},
  journal={arXiv preprint arXiv:2602.18889},
  year={2026}
}

@article{ichinomiya2025machine,
  title={Machine learning of time series data using persistent homology},
  author={Ichinomiya, Takashi},
  journal={Scientific Reports},
  volume={15},
  number={1},
  pages={20508},
  year={2025},
  publisher={Nature Publishing Group UK London}
}

@INCOLLECTION{Eckmann1995-vb,
  title     = "Recurrence plots of dynamical systems",
  booktitle = "Turbulence, Strange Attractors and Chaos",
  author    = "Eckmann, J-P and Kamphorst, S Oliffson and Ruelle, D",
  publisher = "WORLD SCIENTIFIC",
  pages     = "441--445",
  month     =  9,
  year      =  1995,
  doi = {10.1209/0295-5075/4/9/004}
}

@book{clarke1990optimization,
  title={Optimization and nonsmooth analysis},
  author={Clarke, Frank H},
  year={1990},
  doi={10.1137/1.9781611971309},
  publisher={SIAM}
}

@ARTICLE{Clarke1975-zo,
  title     = "Generalized gradients and applications",
  author    = "Clarke, Frank H",
  journal   = "Trans. Am. Math. Soc.",
  publisher = "American Mathematical Society (AMS)",
  volume    =  205,
  pages     = "247--247",
  year      =  1975,
  language  = "en",
  doi = {10.1090/S0002-9947-1975-0367131-6}
}

@article{BARTELS1995385,
title = {Continuous selections of linear functions and nonsmooth critical point theory},
journal = {Nonlinear Analysis: Theory, Methods \& Applications},
volume = {24},
number = {3},
pages = {385-407},
year = {1995},
issn = {0362-546X},
doi = {https://doi.org/10.1016/0362-546X(95)91645-6},
url = {https://www.sciencedirect.com/science/article/pii/0362546X95916456},
author = {Sven G. Bartels and Ludwig Kuntz and Stefan Scholtes}
}

@article{Agrachev97,
	author = {Agrachev, A. A. and Pallaschke, D. and Scholtes, S.},
	doi = {10.1023/A:1021810206861},
	isbn = {1573-8698},
	journal = {Journal of Dynamical and Control Systems},
	number = {4},
	pages = {449--469},
	title = {On Morse Theory for Piecewise Smooth Functions},
	url = {https://doi.org/10.1023/A:1021810206861},
	volume = {3},
	year = {1997},
}

@ARTICLE{Fugacci2020-qx,
  title     = "Critical sets of {PL} and discrete Morse theory: A
               correspondence",
  author    = "Fugacci, Ulderico and Landi, Claudia and Varl{\i}, Hanife",
  journal   = "Comput. Graph.",
  publisher = "Elsevier BV",
  volume    =  90,
  pages     = "43--50",
  month     =  aug,
  year      =  2020,
  language  = "en",
  doi = {10.1016/j.cag.2020.05.020}
}

@ARTICLE{Banchoff1967-tn,
  title     = "Critical points and curvature for embedded polyhedra",
  author    = "Banchoff, Thomas",
  journal   = "J. Differential Geom.",
  publisher = "International Press of Boston",
  volume    =  1,
  number    = "3-4",
  pages     = "245--256",
  month     =  jan,
  year      =  1967,
  doi = {10.4310/jdg/1214428092}
}

@article{nanda2026stratified,
  title={Stratified Morse Theory for Cell Complexes},
  author={Nanda, Vidit and Tombari, Francesca},
  journal={arXiv preprint arXiv:2601.18343},
  year={2026}
}

@BOOK{Goresky1988-ym,
  title     = "Stratified Morse theory",
  author    = "Goresky, Mark and MacPherson, Robert",
  publisher = "Springer",
  series    = "Ergebnisse der Mathematik und ihrer Grenzgebiete. 3. Folge",
  edition   =  1988,
  month     =  feb,
  year      =  1988,
  address   = "Berlin, Germany",
  doi = {10.1007/978-3-642-71714-7}
}

\end{document}